\documentclass[11pt]{article}
\usepackage[top = 2.5cm, bottom = 2.5cm, left = 2.5cm, right = 2.5cm]{geometry}
\usepackage{graphicx}
\usepackage{amssymb,amsfonts}
\usepackage{amsmath}
\usepackage{amsthm}
\usepackage{authblk}
\usepackage{epstopdf}
\usepackage{empheq}
\usepackage{array}
\usepackage{caption}
\usepackage{subcaption}
\usepackage{colonequals}
\usepackage{bbm}
\usepackage{bm}
\usepackage{multirow}
\usepackage{booktabs}
\usepackage{mathrsfs}
\usepackage{enumitem}
\usepackage[usenames, dvipsnames]{color}
\usepackage{authblk}
\usepackage{hyperref}
\hypersetup
{
	colorlinks,
	citecolor = black,
	filecolor = black,
	linkcolor = black,
	urlcolor  = black
}
\usepackage{natbib}
\usepackage{easyReview}

\theoremstyle{definition}

\newtheorem{lemma}{Lemma}
\newtheorem{theorem}{Theorem}
\newtheorem{corollary}{Corollary}
\newtheorem{proposition}{Proposition}

\theoremstyle{remark}

\newcommand{\ddt}[1][]{\frac{d#1}{dt}} %strong derivative
\def\wddt{\partial_t} %weak derivative
\def\frame{F}
\def\lagframe{\bold{F}}
\def\Mzero{M_0} %initial volume
\def\Mt{M_t} %volume at time t
\def\Mtzero{M_{t_0}} %volume at time t
\def\R{\mathbb{R}}

\def\D{\mathcal{D}\mathit{iff}}

\def\C {\mathcal{C}}
\def\priodensity{p} %density of disease
\def\priodensitybis{q}
\def\lagdensity{\bold{p}} %density of disease in Lagrange coordinates
\def\lagdensitybis{\bold{q}} %density of disease in Lagrange coordinates after change of var
\def\testfun{h} %test function for weak formulation
\def\id{\mathit{id}} %identity map
\def\diffusioncoef{r} %diffusion coefficients
\def\diffusiontensor{S}%diffusion tensor
\def\newdifftensor{\bold{S}}%diffusion tensor in Lagrange coordinates
\def\jerkgeneric{\mathfrak j} %jerk
\def\jerk{j} %jerk
\def\jerkname{yank}
\def\weight{\omega} %weight of V norm
\def\reaction{R}
\def\reg{m}
\def\forcefunction{Q}
\def\elastictensor{A} %elastic tensor
\def\elastictensorgeneric{\mathscr A} %elastic tensor
\newcommand{\jac}[1]{J\hspace{-1pt}#1}
\newcommand{\genC}[1][]{{{\mathfrak C}_{#1}}} %generic constant
\newcommand{\parabolic}{\mathcal{L}}
\newcommand{\gl}{\mathrm{GL}}
\renewcommand{\top}{T}
\newcommand{\dm}{d}

\def\id{id_{\R^\dm}}
\newcommand{\scp}[2]{\big\langle {#1}\, , \, {#2}\big\rangle}

\newcommand{\lform}[2]{\big( {#1}\, | \, {#2}\big)}
\newcommand{\Lform}[2]{\Big( {#1}\, \big| \, {#2}\Big)}

\newcommand{\overbar}[1]{\mkern 1.5mu\overline{\mkern-1.5mu#1}\mkern 1.5mu}
\allowdisplaybreaks[1]

\graphicspath{{figures/}}

\title{Diffeomorphic shape evolution coupled with a reaction-diffusion PDE on a growth potential}
\date{}

\author[1]{Dai-Ni~Hsieh\thanks{dnhsieh@jhu.edu}}
\author[2]{Sylvain~Arguill\`ere\thanks{sylvain.arguillere@univ-lille.fr}}
\author[1]{Nicolas~Charon \thanks{charon@cis.jhu.edu}}
\author[1]{Laurent Younes \thanks{laurent.younes@jhu.edu}}
\affil[1]{Department of Applied Mathematics and Statistics, Johns Hopkins University, Baltimore, MD, USA}
\affil[2]{Laboratoire Paul Painlev\'e, University of Lille, France}

\begin{document}
\maketitle

\begin{abstract}
  This paper studies a longitudinal shape transformation model in which shapes are deformed in response to an internal growth potential that evolves according to an advection reaction diffusion process. This model extends prior works that considered a static growth potential, i.e., the initial growth potential is only advected by diffeomorphisms. We focus on the mathematical study of the corresponding system of coupled PDEs describing the joint dynamics of the diffeomorphic transformation together with the growth potential on the moving domain. Specifically, we prove the uniqueness and long time existence of solutions to this system with reasonable initial and boundary conditions as well as regularization on deformation fields. In addition, we provide a few simple simulations of this model in the case of isotropic elastic materials in 2D.
\end{abstract}

\section{Introduction}

We study in this paper a system of coupled evolution equations describing shape changes (e.g., growth, or atrophy) for a free domain in $\R^\dm$. A first equation, modeled as a diffusion-convection-reaction equation, defines the evolution of a scalar function defined on the domain,  this scalar function being roughly interpreted as a ``growth potential'' that  determines shape changes. The second equation describes the relationship between this potential and a smooth Eulerian velocity field and is modeled as a linear, typically high-order, partial differential equation (PDE). The free-form evolution of the domain then follows the flow associated with this velocity field. 

There is  a significant amount of literature on growth models and shape changes, where the dominant approach (in 3D) uses finite elasticity, modeling the strain tensor as a product of a growth tensor (non necessarily achievable by a 3D motion) and a correction term that makes it achievable, using this correction term as a replacement of the strain tensor in a hyper-elastic energy \citep{rodriguez1994stress}. Minimizing the energy of this residual stress then leads to PDEs describing the displacement defined on the original domain (assumed to be at rest) into the final one. We refer to several survey papers such as \citet{menzel2012frontiers,humphrey2003review,ambrosi2011perspectives} for references. 

In this paper, we tackle the shape change problem with a different approach. First, we use a dynamical model of time-dependent shapes, which allows us to analyze each infinitesimal step using linear models. Second, our model includes no residual stress but rather assumes that a new elastic equilibrium is reached at each instant. Our focus is indeed on slow evolution of living tissues, in which shape changes occur over several years and tissues can be assumed to remain constantly at rest. More precisely, our model  assumes that at each time, an infinitesimal force places the shape into a new equilibrium, which becomes the new reference configuration. The force is assumed to derive from a potential, itself associated to the solution of a reaction-convection-diffusion equation, while the new equilibrium is obtained as the solution of a linear equation that characterizes the minimizer of a regularized deformation energy. Figure \ref{fig:expl} provides an example of such an evolution.

\begin{figure}[hbt!]
    \centering
    \begin{subfigure}{0.32\textwidth}
        \centering
        \includegraphics[trim = 30 50 30 50, clip, width = \textwidth]{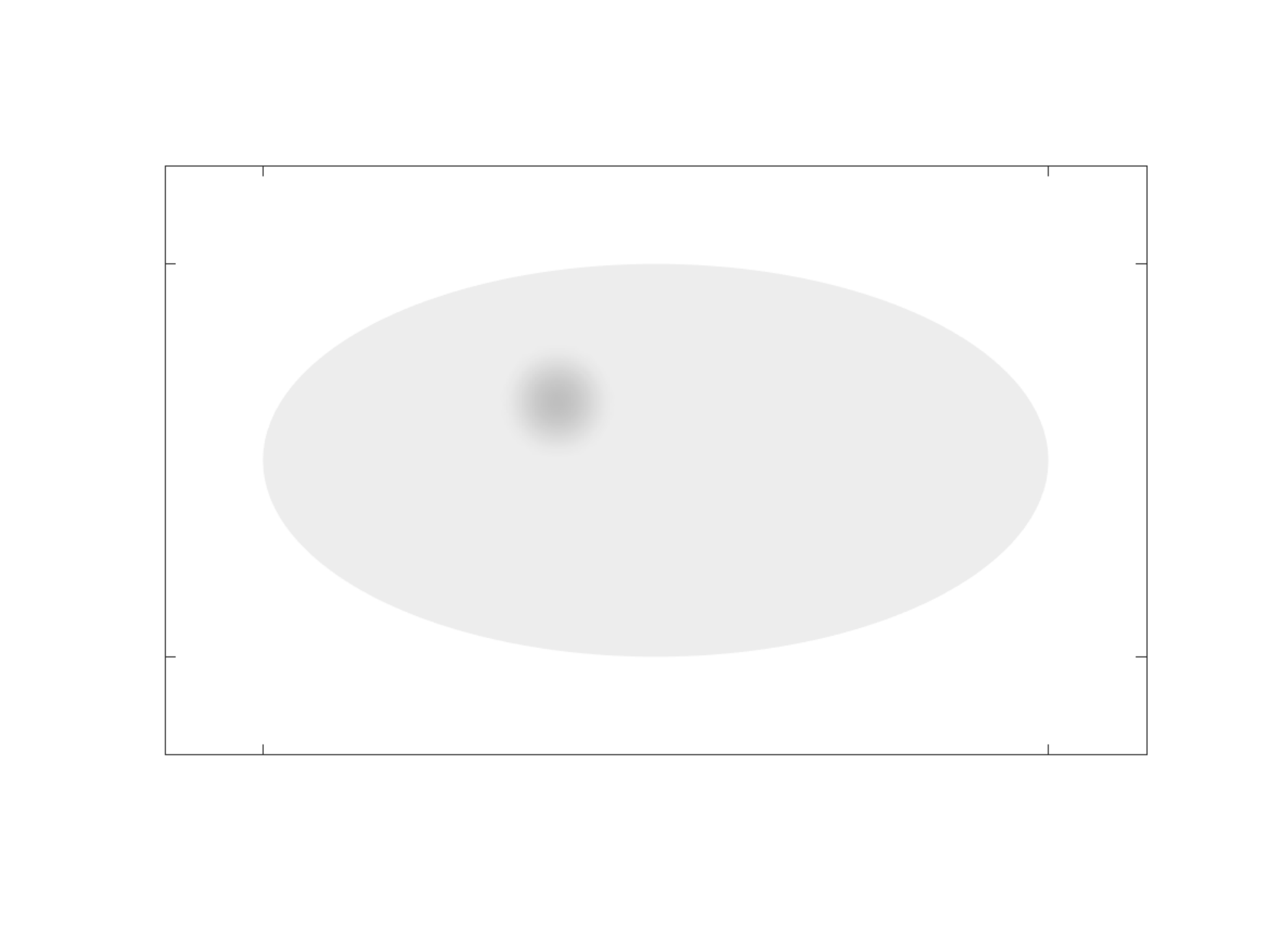}
    \end{subfigure}
    \begin{subfigure}{0.32\textwidth}
        \centering
        \includegraphics[trim = 30 50 30 50, clip, width = \textwidth]{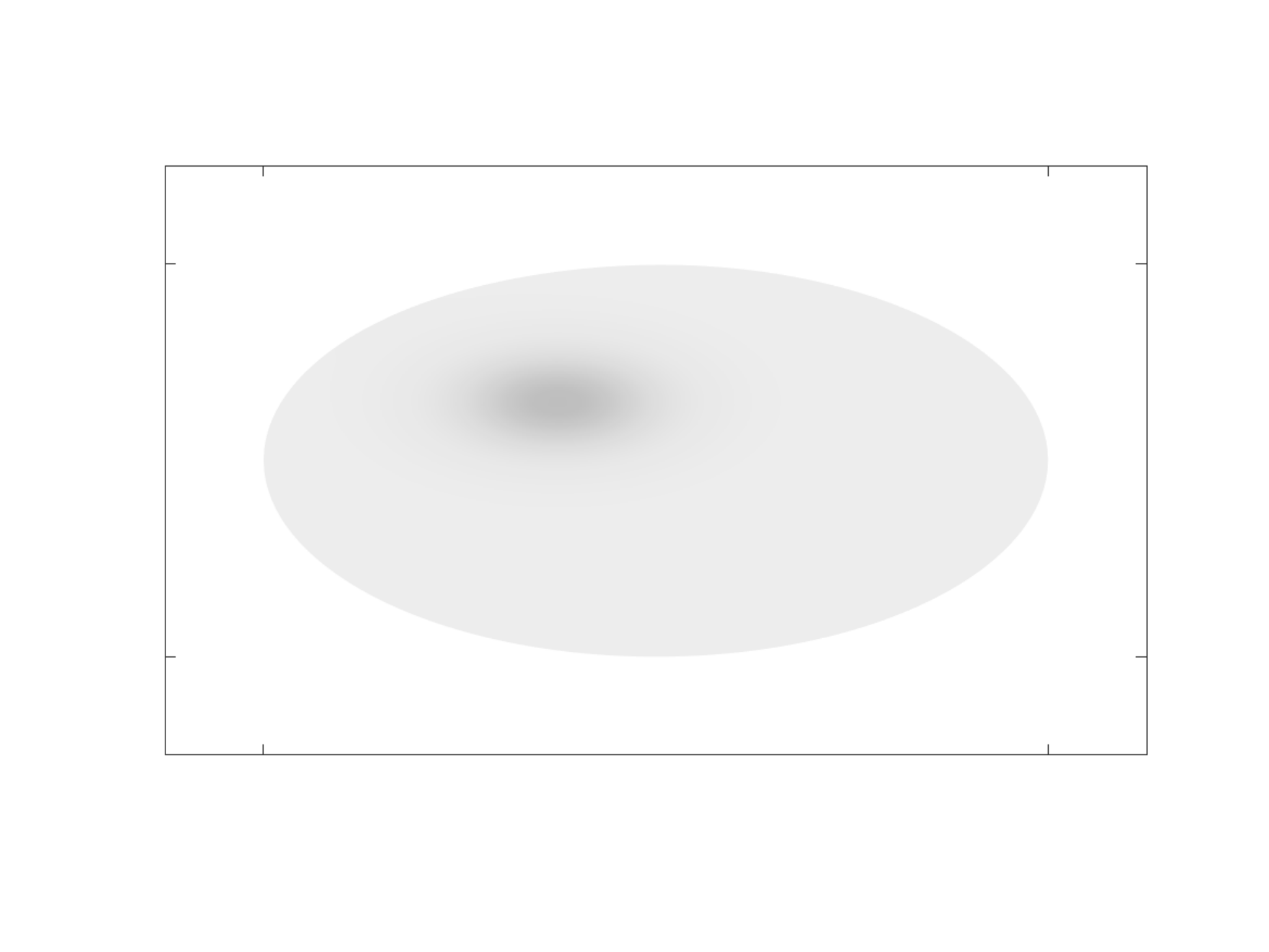}
    \end{subfigure}
    \begin{subfigure}{0.32\textwidth}
        \centering
        \includegraphics[trim = 30 50 30 50, clip, width = \textwidth]{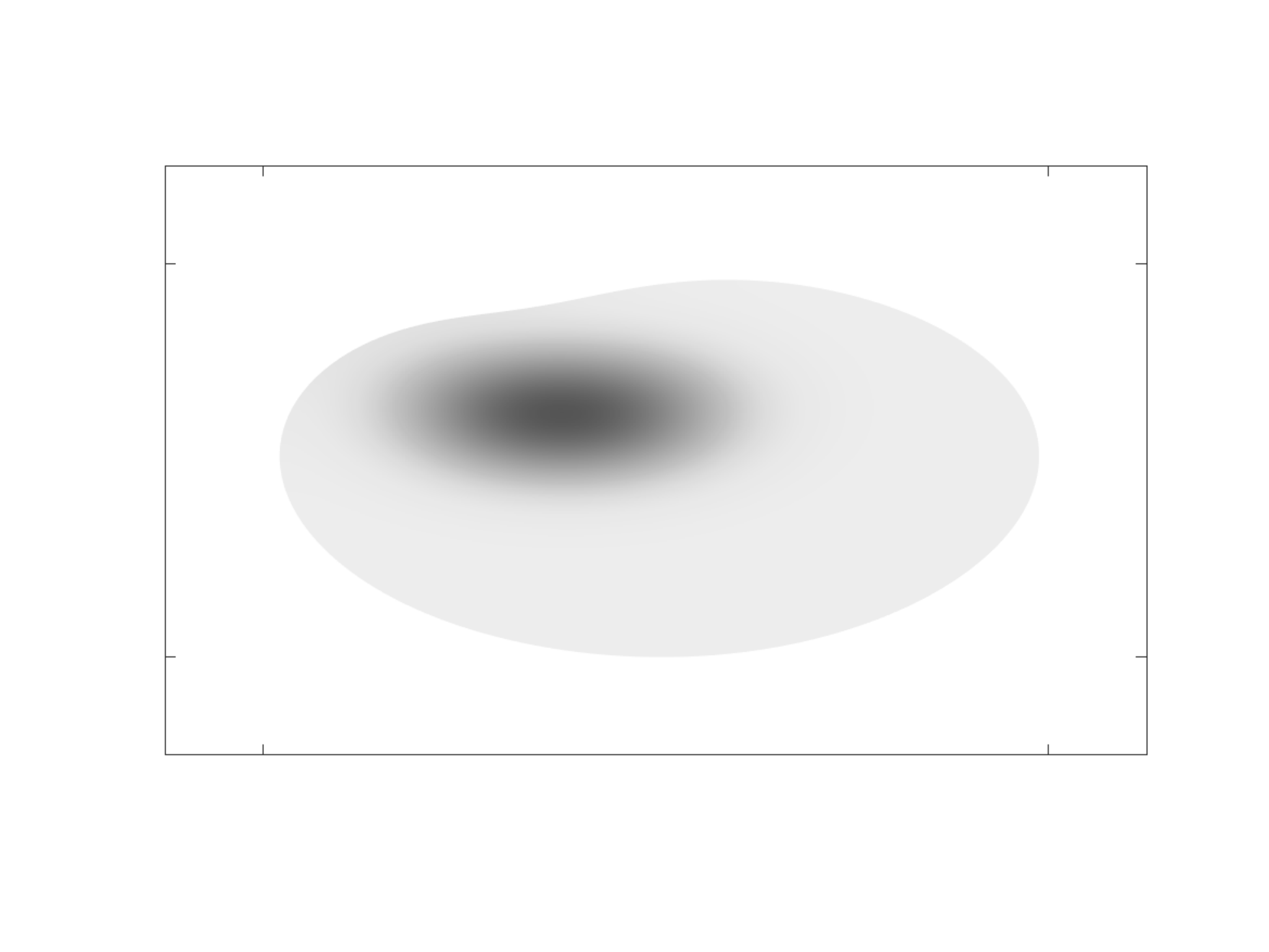}
    \end{subfigure}
    \\
    \begin{subfigure}{0.32\textwidth}
        \centering
        \includegraphics[trim = 30 50 30 50, clip, width = \textwidth]{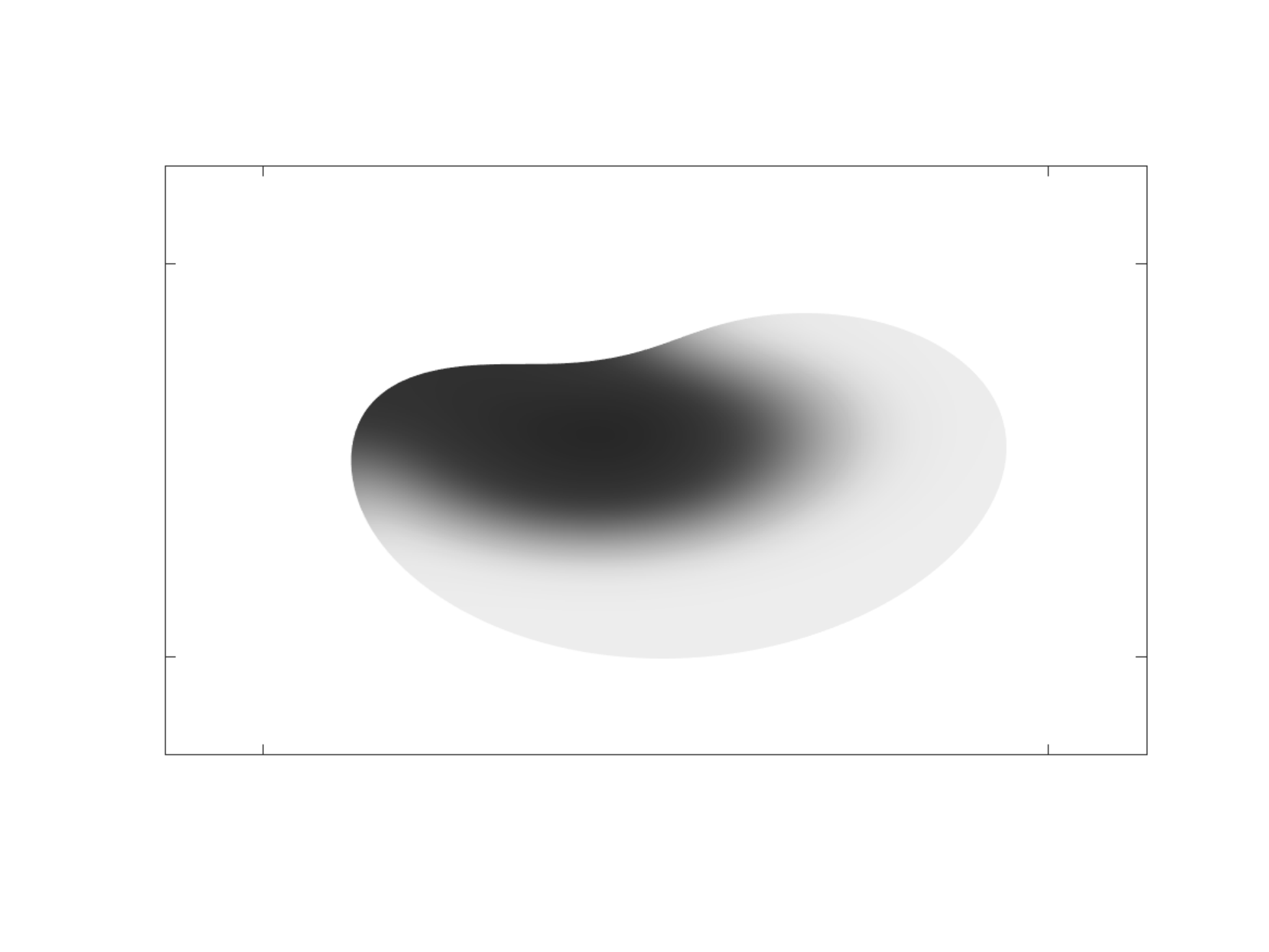}
    \end{subfigure}
    \begin{subfigure}{0.32\textwidth}
        \centering
        \includegraphics[trim = 30 50 30 50, clip, width = \textwidth]{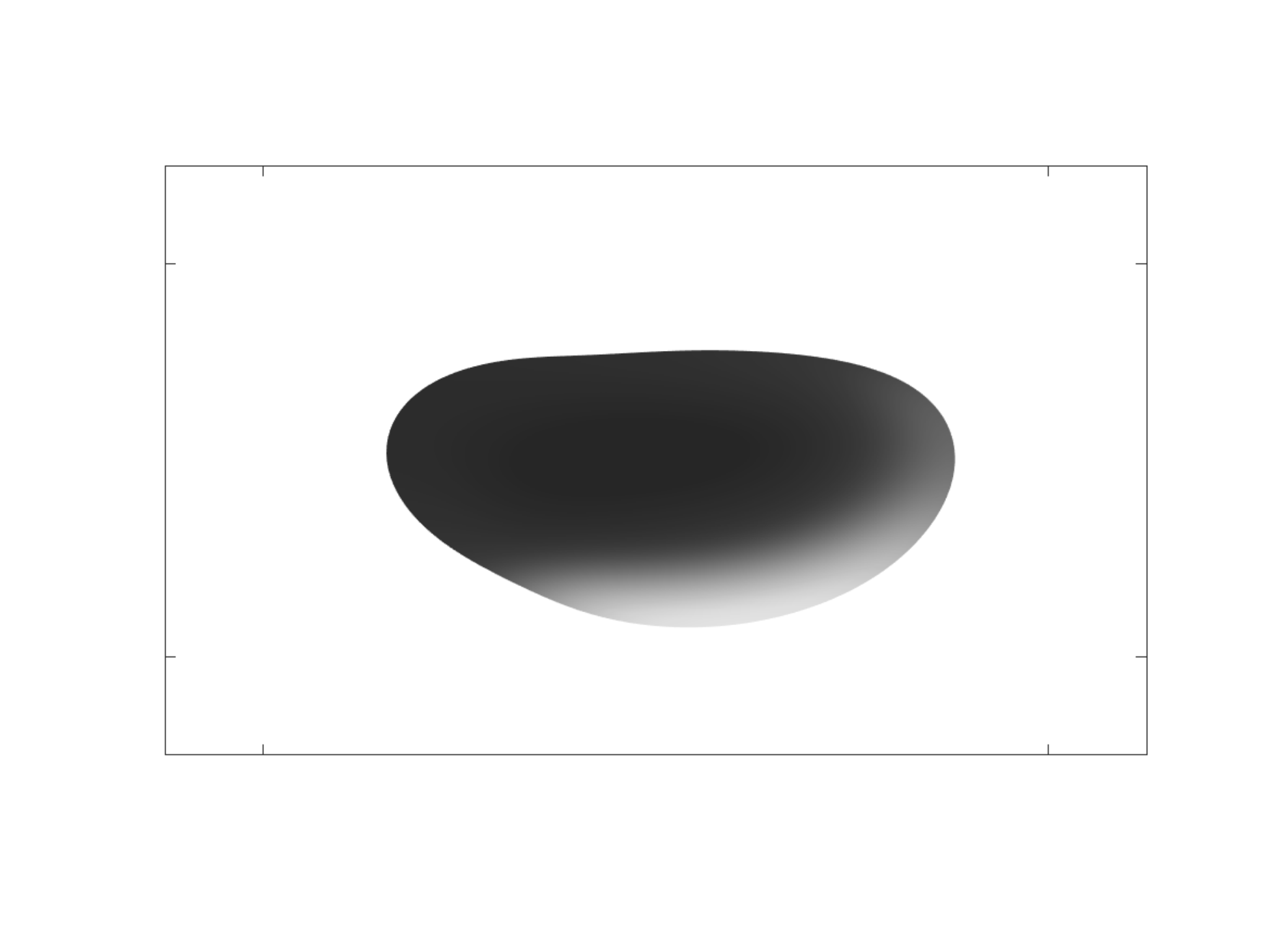}
    \end{subfigure}
    \begin{subfigure}{0.32\textwidth}
        \centering
        \includegraphics[trim = 30 50 30 50, clip, width = \textwidth]{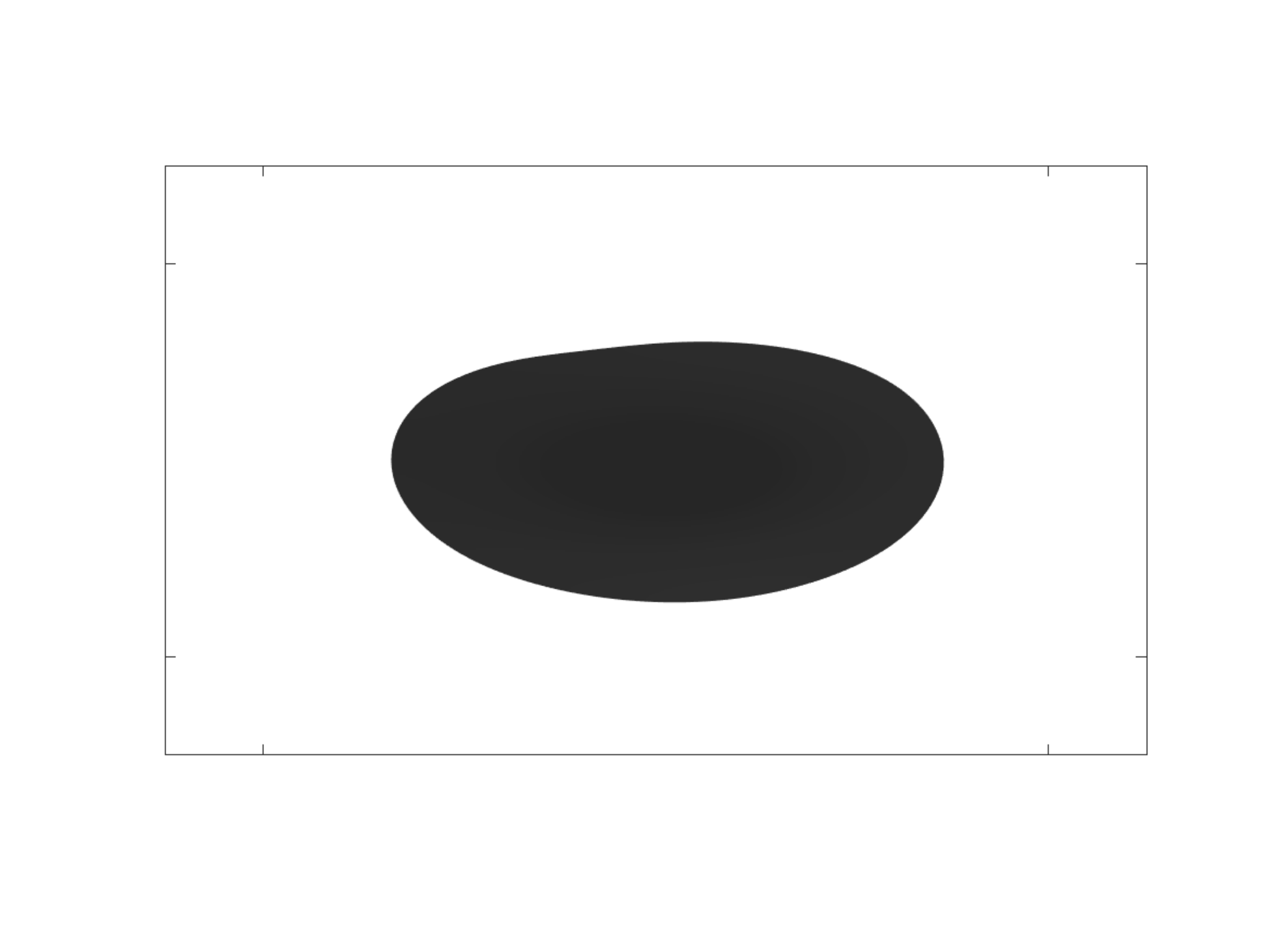}
    \end{subfigure}
    \caption{\label{fig:expl} Example of shape change in two dimensions relative to a negative growth potential (to be read from left to right and up to down). The potential is initialized locally, and spreads while affecting the shape of the domain until reaching saturation.}
\end{figure}

In our main result (see Theorem \ref{thm:main}), we will  prove  that the full system has solutions over arbitrary time intervals, and that the shape domain evolves according to a diffeomorphic flow. This result opens the possibility to formulate optimal control and inverse problems in which one determines initial growth potentials (within a parametrized family) transforming a given initial shape into a target shape. Such inverse problems were considered in the multiplicative strain tensor framework and for thin-plate models in \cite{lewicka2011Foppl}, and for crystal growth control, albeit not in a free-form setting, in \citet{trifkovic2009multivariable, bajcinca2013analytic}. The optimal control of free-form surfaces modeling the interface between two phases was considered in \cite{bernauer2011optimal}. In \citet{bressan2018model}, tissue growth is modeled through a control system evolving as an elastic body experiencing local volume changes controlled by the concentration of a ``morphogen'', which itself evolves according to a linear elliptic equation. Our model can be seen as a regularization of an extension of the model in that paper (we make, in particular, non-isotropic assumptions and our growth potential evolves according to a non-linear equation), our regularization allowing us to obtain long term existence and uniqueness results, that were not available in  \citet{bressan2018model}. Finally, \cite{kulason2020reaction-diffusion} introduce a reaction-diffusion model to analyze disease propagation and  thickness changes in brain cortical surfaces, where changes are happening (unlike the model studied in our paper) within a fixed domain. 

%(NEED MORE REFERENCES)

This paper follows and completes \citet{hsieh2020mechanistic} (see also \citet{Hsieh2019}), which adopts a similar approach with a functional dependency of the growth potential on the diffeomorphic flow. This assumption is relaxed here, since the potential follows its own PDE, with an evolution coupled with that of the shape. This extension will, as we will see, significantly complicate the theoretical study of the equations, as well as their numerical implementation. 

\section{General framework and main theorems.}
%Introduce the main notation, control system and state theorems on existence of solutions for the IVP and for optimal control with respect to initial conditions.

\subsection{Notation} 
\paragraph{Ambient space, vector fields and diffeomorphisms.} We will work in the Euclidean space $\R^\dm$. For an integer $s\geq0$, and open subset $U$ of $\R^\dm$, we let $H^s(U)$ be the Hilbert space of all real functions on $U$ of Sobolev class $H^s=W^{2,s}$. Recall that $H^0(U)=L^2(U).$

We  denote by $\mathcal{C}_0^\reg(\R^\dm,\R^\dm)$ the set of all $\reg$-times continuously differentiable vector fields whose $k$-th derivative $D^kv$  go to zero at infinity for every $k$ between 0 and $\reg$. It is a Banach space under the usual norm
$$
\|v\|_{\reg, \infty} = \sum_{k=0}^\reg \max_{x \,\in\, \mathbb{R}^\dm}|D^k v(x)|,\quad v\in \C_0^\reg(\R^\dm,\R^\dm).
$$

For a generic function $f: [0, T] \times \mathbb{R}^\dm \rightarrow \mathbb{R}^\dm$, we will use the notation $f(t): \mathbb{R}^\dm \rightarrow \mathbb{R}^\dm$ defined by $f(t)(x) = f(t, x)$. We will use $\genC$ to denote a generic constant and $\genC[a]$ to show a generic constant depending on $a$. The value of such constants may change from equation to equation while keeping the same notation.

 Now, assume $\reg\geq1$. Let $\D^\reg(\R^\dm)$ be the space of $\mathcal{C}^\reg$-diffeomorphisms of $\R^\dm$ that go to the identity at infinity, that is, the space of all diffeomorphisms $\varphi$ such that $\varphi-\id\in \C_0^\reg(\R^\dm,\R^\dm)$, with $\id:\R^\dm\rightarrow\R^\dm$ the identity map $\id(x)=x$. Do note that $\D^\reg(\R^\dm)$ is an open subset of the Banach affine space $\id+\C_0^\reg(\R^\dm,\R^\dm)$, with the induced topology. It is also known to be a topological group for the law of composition \citep{bruveris2016completeness}, so we also have $\varphi^{-1}\in \D^\reg(\R^\dm)$. We can then define on $\D^\reg(\R^\dm)$ the distance $d_{\reg,\infty}$ by
\begin{equation}
\label{eq:d.reg.inf}
d_{\reg,\infty}(\varphi,\psi)=\max\big(\Vert \varphi-\psi\Vert_{\reg,\infty},\Vert\varphi^{-1}-\psi^{-1}\Vert_{\reg,\infty}\big),\quad \varphi,\psi\in\D^\reg(\R^d),
\end{equation}
whose open balls will be denoted $B_r(\varphi),\ r>0,\ \varphi\in\D^\reg(\R^d).$ This is easily checked to be a complete distance, and it does not change the topology of $\D^\reg(\R^\dm)$. We introduce it because we will often need to assume bounds on diffeomorphisms and their inverse at the same time.

\paragraph{Operators and controlled curves in Banach spaces.} 

If $B$ and $\widetilde B$ are separable Banach spaces,
$\mathscr{L}(B, \widetilde B)$ will denote the vector space of bounded linear operators from $B$ to  $\widetilde B$. Weak convergence of sequences $(x_n)$ in $B$ will be denoted by  $x_n \rightharpoonup x$. Denoting the topological dual of $B$ by $B^*$, we will use the notation $(\mu \mid v)$ rather than $\mu(v)$ to denote the evaluation of $\mu \in B^*$ at $v \in B$. {
We say that a linear operator $A \in \mathscr{L}(B, B^*)$ is symmetric if the corresponding bilinear form $(v, w) \mapsto (Av \mid w)$ is symmetric. }

For a given $T>0$ and open subset $U$ of a Banach space ${B}$,
we will denote by $L^p([0,T],U),\ p\in[1,+\infty]$ the space of measurable maps $f:[0,T]\rightarrow U$ such that $t\mapsto \Vert f(t) \Vert_B^p$ is integrable. One can then define the Sobolev space $W^{p,1}([0,T],U)$ whose elements $f$ are differentiable almost everywhere, i.e. the differential $\displaystyle t\mapsto\ddt[f] (t)=\lim_{t'\rightarrow t}\frac{f(t')-f(t)}{t'-t}$ 
%which we will also sometimes write as $\dot{f}(t)$, 
exists almost everywhere and 
$$
\forall t_0,t_1\in[0,T],\quad  f(t_1)-f(t_0)=\int_{[t_0,t_1]}\ddt[f](t)dt.
$$
For $p=2,$ we will simply write $H^1$ instead of $W^{2,1}.$

Case in point, for a time-dependent vector field $v\in L^1([0,T],\C^\reg_0(\R^\dm,\R^\dm))$, there is a unique  $\varphi:t\mapsto\varphi(t)$ in $W^{1,1}([0,T],\D^\reg(\R^\dm))$ that satisfies $\varphi(0)=\id$ and $\ddt[\varphi](t)=v(t)\circ\varphi(t)$ for almost every $t$.

\paragraph{RKHS (Reproducing Kernel Hilbert Spaces) of vector fields.} Throughout this paper, $V$ is a Hilbert space of vector fields on $\mathbb{R}^\dm$ that is continuously embedded in $C_0^\reg(\mathbb{R}^\dm, \mathbb{R}^\dm)$ for some $\reg \geq 1$ (we will write $V \hookrightarrow C_0^\reg(\mathbb{R}^\dm, \mathbb{R}^\dm)$), with inner product $\scp{\cdot}{\cdot}_V$ and norm $\|\cdot\|_V$. Since $V \hookrightarrow C_0^\reg(\mathbb{R}^\dm, \mathbb{R}^\dm)$, there exists a constant $c_V$ such that 
\begin{equation}
\label{eq:cv}
\|v\|_{\reg, \infty} \leq c_V \|v\|_V.    
\end{equation}
The duality map $L_V: V \to V^*$ is given by
\[
\lform{L_V \hspace{1pt} v}{w} = \scp{v}{w}_V
\]
and provides an isometry from $V$ onto $V^*$. We denote the inverse of $L_V$ by $K_V \in \mathscr{L}(V^*, V)$, which, because of the embedding assumption, is a kernel operator \citep{aronszajn1950theory}. Note that
\[
	\|v\|_V^2 = (L_V \hspace{1pt} v \mid v) = (K_V^{-1} \hspace{1pt} v \mid v) .
\]
As an example, the space $V$ can be the reproducing kernel Hilbert space (RKHS) associated with a Mat\'ern kernel of some order $s$, and some width $\sigma$, which, in three dimensions, implies that $V$ is a Sobolev space $H^{s+2}$. For the specific value {$s=3$}, which we will use in our experiments, the kernel operator (when applied to a vector measure $\mu\in V^*$) takes the form
\[
(K_V \hspace{1pt} \mu)(x) = \int_{\R^\dm} \kappa(|x-y|/\sigma) \, d\mu(y)
\]
with $\kappa(t) = (1+t+2t^2/15 + t^3/15)e^{-t}$.

\paragraph{Weak derivatives for Hilbert-valued functions.} 
Following \cite[Chapter 1, Section 1.3]{magenes1972non}, we define generalized derivatives of functions $u: (0,T) \to H$, where $T$ is a positive number and $H$ a Hilbert space as follows. Let $\mathscr D\big((0,T)\big)$ denote the Schwartz space of compactly supported infinitely differentiable real-valued functions defined on $(0,T)$. The space of $H$-valued distributions is
\[
	\mathscr{D}^*\big( (0, T), H \big) \colonequals \mathscr{L}(\mathscr{D}\big( (0, T) \big), H) .
\]
If $u \in \mathscr{D}^*\big( (0, T), H \big)$, its generalized derivative, denoted $\wddt u$, is the element of $\mathscr{D}^*\big( (0, T), H \big)$ defined by
\begin{equation}
\label{eq:weak.der.def}    
\wddt u(\varphi) \colonequals -u\Big(\ddt[\varphi]\Big) \,\in\, H
	\ \ \ \mbox{ for all } \varphi \in \mathscr{D}\big( (0, T) \big).
\end{equation}

We can identify any $u \in L_{\mathrm{loc}}^1([0, T], H)$ (i.e., $u\in L^1([a,b], H)$ for all $[a, b] \subset (0, T)$), with the corresponding $\widetilde u \in \mathscr{D}^*\big( (0, T), H \big)$ given by
\[
	\widetilde u(\varphi) \colonequals \int_0^T u(t) \, \varphi(t) \, dt \,\in\, H
	\ \ \ \mbox{ for all } \varphi \in \mathscr{D}\big( (0, T) \big)
\]
and show that $\widetilde u \in \mathscr{D}^*\big( (0, T), H \big)$. We can therefore see  $L_{\mathrm{loc}}^1([0, T], H)$ as a  subset of $\mathscr{D}^*\big( (0, T), H \big)$.

In what follows, we will use the following two results, both taken from \citet{magenes1972non}.

\begin{theorem}
\label{thm:u_dtu_continuous}
Let $T$ be a positive number and $\Omega$  an open subset of $\R^\dm$. 
Assume that $u \in L^2([0,T], H^1(\Omega))$ and that $\wddt u \in L^2([0,T], H^1(\Omega)^*)$. Then $u \in C([0, T], L^2(\varOmega))$. (See \citet[Chapter 1, Proposition 2.1 and Theorem 3.1]{magenes1972non}.)
\end{theorem}

We will also use the following general result on weak solutions of parabolic equations. A bounded linear map $\parabolic: L^2([0,T], H^1(\Omega))\to L^2([0,T], H^1(\Omega)^*)$ is coercive if there exists $\alpha > 0$ such that $(\parabolic \hspace{1pt} u \mid u) \geq \alpha \, \int_0^T \|u\|_{H^1(\Omega)}^2\,dt$ for all $u \in L^2([0,T], H^1(\Omega))$. 

\begin{theorem}
\label{thm:abstract_parabolic_ivp}
Given a coercive bounded linear mapping $\parabolic: L^2([0,T], H^1(\Omega))\to L^2([0,T], H^1(\Omega)^*)$, a function $f$ in $L^2([0,T], H^1(\Omega)^*)$ and an initial condition $u_0 \in L^2(\Omega)$, there exists a unique solution 
$u\in L^2([0,T], H^1(\Omega))$ of the parabolic initial value problem
\begin{equation}
	\left\{
		\begin{array}{l}
			\partial_t \hspace{1pt} u + \parabolic \hspace{1pt} u = f \\
			u(0) = u_0\,.
		\end{array}
	\right. 
	\label{eq:abstract_parabolic_ivp}
\end{equation}
(See \citet[Chapter 3, Theorem 1.1, Section 4.3, and Remark 4.3]{magenes1972non}.)
\end{theorem}

We will also need the following technical lemma. Its proof is a simple application of functional approximation theorems in $L^2$ and we shall omit it for brevity.
\begin{lemma}
\label{lemma:dt_norm_squared}
Let $\mathcal{H}$ be a Banach space and suppose that $w \in L^2([0, T], \mathcal{H})$ with $\partial_t \hspace{1pt} w \in L^2([0, T], \mathcal{H}^*)$. Then the derivative in the sense of distributions $\partial_t \hspace{1pt} \|w(\cdot)\|_{\mathcal{H}}^2$ is a function in $L^1([0, T])$ and equals to $t \mapsto 2 \, \big( (\partial_t \hspace{1pt} w)(t) \mid w(t) \big)$ for almost every $t$.
\end{lemma}

%Of particular interest, for every $\mu(\cdot)\in L^1([0,T],V^*),$ and  $\varphi_0\in \D^s(\R^\dm),$ there is a unique $\varphi(\cdot)\in H^1([0,T],\D^s(\R^\dm))$ such that for almost every $t$, $\varphi(0)=\varphi_0,$ and $\dot\varphi(t)=\varphi(t)\circ\left(K_V\mu(t)\right).

\subsection{Control systems for shapes}
\label{ssec:control_systems}

We want to refine the system introduced in \cite{hsieh2020mechanistic}, that was designed as a mathematical model representing possibly pathological shape changes in human organs or tissues. The control system starts with 
an initial volume, and exhibits a time-dependent deformation induced by a vector field on the domain, where this vector field results from auxiliary variables defined on the volume (e.g., a scalar field) that one can loosely interpret as a manifestation of a ``disease''.
%from  taken for example from MRIs of a patient at different times, and model how the disease could have induced the deformation.

\paragraph{Mixed diffeomorphic-elastic model with fixed potential.} We therefore start with an open domain $\Mzero$ of $\R^\dm$ 
%which will represent a part of a brain, such as the entorhinal cortex, 
and model a deformation $t\mapsto \Mt$.
%caused by the propagation of a degenerative disease. 
We first introduce the ``diffeomorphic'' model, which is the foundation of the LDDMM algorithm (for large deformation diffeomorphic metric mapping \citep{beg2005computing}) in shape analysis. In this model, the deformation is tracked through a time-dependent diffeomorphism $\varphi(\cdot)\in W
^{1,1}([0,T],\D^\reg(\R^\dm))$ which is also the flow of a time-dependent vector field $v(\cdot)\in L^1([0,1],V)$ that belongs to our RKHS $V$, so that for all $ t$ in $ [0,T],$ $\Mt=\varphi(t,\Mzero)$, with 
\begin{equation}
\label{eq:flow}
\varphi(0)=\id \quad \text{and}\quad  \ddt{\varphi}(t)=v(t)\circ\varphi(t)
\ \text{almost everywhere.}
\end{equation}
The vector field $v$ is preferably represented in the form $v(\cdot)=K_Vu(\cdot)$ for some $u(\cdot)\in L^1([0,T],V^*)$, which then acts as a control. This control can be left unspecified and estimated as part of an optimal control problem (as done in \citet{beg2005computing,joshi2000landmark,dupuis1998variation,trouve1995approach,arguillere2014shape}), or modeled as an element of some parametrized class of time-dependent linear forms on $V$ \citep{younes2011constrained,younes2014gaussian,gris2018sub-riemannian,younes2020sub-riemannian}. We note that the relation $v(t) = K_V u(t)$ is equivalent to the variational formulation
$$
v(t)=\underset{v' \,\in\, V}{\arg\min}\ \frac{1}{2} \lform{L_V v'}{v'}-(u(t)\mid v').%\quad a.e. \ t\in[0,T].
$$
%for some $u(\cdot)\in L^1([0,T],V^*).$ This yields  almost everywhere.

Alternatively, we may consider $\Mt$ as a deformable solid, with infinitesimal deformation energy quantified by a linear tensor $\elastictensorgeneric(t)$ which is required to be a symmetric, positive semi-definite operator in $ \mathscr{L}(H^1(\Mt,\R^\dm),H^1(\Mt,\R^\dm)^*)$. Now, exert an infinitesimal force density $\jerkgeneric(t)dt$ on $\Mt$ ($\jerkgeneric(t)$ is a time derivative of a force, also called a \textit{\jerkname}). Assuming this \jerkname\ belongs to $H^1(\Mt,\R^\dm)^*$, the infinitesimal deformation $v(t)dt$ that brings $\Mt$ to equilibrium is given, when it exists, by
$$
v(t)=\underset{v' \,\in\, H^1(\Mt,\R^\dm)}{\arg\min}\ \frac{1}{2}(\elastictensorgeneric(t)v'\mid v')-(\jerkgeneric(t)\mid v').
$$
In this paper, following \cite{hsieh2020mechanistic}, we fix a weight $\weight>0$ to combine the LDDMM model and the deformable solid model, and define
\begin{equation}
\label{eq:vf.var}
v(t)=\underset{v' \,\in\, V}{\arg\min}\ \frac{\weight}{2}\lform{L_V v'}{v'} + \frac{1}{2}(\elastictensorgeneric(t)v'\mid v')-(\jerkgeneric(t)\mid v'),%\quad a.e. \ t\in[0,T].
\end{equation}
using $\jerkgeneric \in L^1([0,T],V^*)$ as a control. Here, we make the abuse of notation identifying $\elastictensorgeneric(t)v$ with its restriction to $\Mt$, the assumption that  $V \hookrightarrow C_0^\reg(\mathbb{R}^\dm, \mathbb{R}^\dm)$ ensuring that this restriction  maps $V$ into $H^1(\Mt,\R^\dm)$. Here, the term $(\weight/2)\lform{L_V v'}{v'}$ can be seen as an internal energy causing permanent change to the shape, or simply as a regularization term ensuring the existence of $v(t)\in V$. Indeed, since $\elastictensorgeneric(t)$ is positive semi-definite, $v(t)$ always exists, and belongs to $L^1([0,T],V)$, so that it generates a well-defined flow $\varphi(\cdot)$ on $\R^\dm$.

We point out important differences between the operators $L_V \doteq K_V^{-1}$ (such that $\|v\|_V^2 = \lform{L_V v}{v}$) and $\elastictensorgeneric(t)$. The former, $L_V: V\to V^*$, is defined on a fixed space of vector fields ($V$) themselves defined on the whole ambient space $\R^\dm$. In contrast, $\elastictensorgeneric(t)$ is defined on  $H^1(\Mt,\R^\dm)$, and therefore applies to vector fields defined on $\Mt$. It is, by definition shape dependent. The global nature of $L_V$ (and higher-order assumption insuring the embedding of $V$ in a space of differentiable functions) is of course what makes possible the diffeomorphic property of the evolving flow over all times intervals.

Although we will work with general assumptions on the deformation energy tensor $\elastictensorgeneric$, the 
%\begin{example}The 
main example of such a tensor in three dimensions comes from linear elasticity \citep{ciarlet1988three,marsden1994mathematical}. 
%Indeed, it is mainly with this example in mind that this paper is written (LINEAR ELASTICITY REFERENCE \cite{.}). 
Generally, such a tensor $\elastictensorgeneric(t)\in \mathscr{L}(H^1(\Mt,\R^3),H^1(\Mt,\R^3)^*)$ is defined so that
$$
(\elastictensorgeneric(t) v\mid w)=\int_{\Mt}\Lform{\mathscr E_t(x, \epsilon_v(x))}{\epsilon_w(x)}dx,
$$
with $\epsilon_v = (Dv + Dv^T)/2$.
Here, $G^T$ is the transpose of a matrix $G$, and for every $x$, $\mathscr E_t(x)$, is a symmetric positive definite tensor on $3\times 3$ symmetric matrices. In particular, one can favor at each point $x$ of $\Mt$ specific directions of deformations by appropriately choosing $\mathscr E_t(x)$. Examples of such tensors that could be used in applications of our model are those given in \citet[Section 4]{hsieh2020mechanistic}. In the simplest situation, one can assume that the material is homogeneous, isotropic and that its elastic properties are also constant in time, in which case for all $x \in \Mt$: 
\begin{equation}
    \label{eq:elastic_tensor_isotropic}
\Lform{\mathscr E_t(x, \epsilon_v)}{\epsilon_w} = \lambda \, \mathrm{tr}(\epsilon_u) \, \mathrm{tr}(\epsilon_v) + 2\mu \, \mathrm{tr}(\epsilon_u^\top \epsilon_v)
\end{equation}
where $\lambda$ and $\mu$ are known as the Lam\'e parameters of the elastic material.
%\end{example}

While not necessarily restricting to elasticity operators such as those above, we will make the additional assumption that $\elastictensorgeneric(t)$ is fully specified by the transformation $\varphi$ (defined by \eqref{eq:flow}) applied to the initial volume. More precisely, we will assume that we are given a mapping
$$
\elastictensor:\D^\reg(\R^\dm)\rightarrow \mathscr{L}\left(H^1(\varphi(\Mzero),\R^\dm),H^1(\varphi(\Mzero),\R^\dm)^*\right),
$$
such that for every diffeomorphism $\varphi,$ $\elastictensor_\varphi$ is a deformation energy tensor on the domain $\varphi(\Mzero)$ and take $\elastictensorgeneric(t) = \elastictensor_{\varphi(t)}$ in \eqref{eq:vf.var}.
% Also fix a weight $\weight>0$. Then, at time $t$, the infinitesimal deformation $v(t)$ is given by
% $$
% v(t)=\underset{v' \,\in\, V}{\arg\min}\ \frac{\weight}{2}\Vert v\Vert_V^2+\frac{1}{2}(\elastictensor_{\varphi(t)} v(t)\mid v(t))-(\jerk(t)\mid v'),
% $$
% with $j(\cdot)\in L^1([0,T],V^*).$ 
%It is understood that $\elastictensor_{\varphi(t)}$ actually acts on the restriction of $v(t)$ to $\Mt$ in that equation. 

\paragraph{Yank model.}
It remains to model a \jerkname\ $\jerkgeneric(\cdot)$ that induces the deformation on $\Mt$. The model considered in \citet{hsieh2020mechanistic} starts with a fixed positive function $g_0\in \C^1(\Mzero,\R)$, with negligible values on the boundary $\partial\Mzero$. In the diseased tissue analogy, this function may be thought of as describing an initial physiological impact of a disease, for example the density of dying cells, or of some protein responsible for tissue remodeling. 
%is meant to represent the distribution of disease at each point. For example, $g_0(x)$ could be the density of diseased cells at $x$, or that of prions inside the cells, which are badly folded proteins responsible for many degenerative diseases. 
Then, for a deformation $\varphi\in\D^\reg(\R^\dm)$, \citep{hsieh2020mechanistic} defines the corresponding \jerkname\ $\jerk_\varphi$  along $\varphi(\Mzero)$ to be the negative gradient of a some function $\forcefunction:\R_+\rightarrow\R_+ $ of the transported function
%disease distribution 
$g_0\circ\varphi^{-1}$, so that the \jerkname\ pulls the shape towards  places where $g_0\circ\varphi^{-1}$ is highest.
%the most disease, 
%through the control of the function $f$, representing the loss of neural connection and cells caused by the sickness. 
Formally, this gives:
$$
\forall\ v'\in V,\quad (\jerk_\varphi\mid v')=\int_{\varphi(\Mzero)}\nabla \left[\forcefunction(g_0\circ\varphi^{-1})\right]^Tv'=\int_{\varphi(\Mzero)}\forcefunction(g_0\circ\varphi^{-1})\hspace{1pt}(-\mathrm{div} \,v'),
$$
where the boundary term is negligible thanks to our assumptions on $g_0$. 

The resulting dynamical system uses $\jerkgeneric(t) = \jerk_{\varphi(t)}$, yielding
\begin{equation}
	\label{eq:nodiffusion}
	\left\{
		\begin{array}{l}
			\partial_t \hspace{1pt} \varphi(t, x)
			= 
			v(t, \varphi(t, x)) , \quad 
			\varphi(0, x) = x ,
			\\[5pt]
			% - - -
			\displaystyle
			v(t)
			=
			\underset{v' \,\in\, V}{\arg\min} \ 
			\frac{\weight}{2} \, \lform{L_Vv'}{v'}
			+
			\frac{1}{2} \, (\elastictensor_{\varphi(t)} \hspace{1pt} v' \mid v')
			-
			(\jerk_{\varphi(t)} \mid v'),
			\\[5pt]
			% - - -
			\displaystyle
			(\jerk_{\varphi(t)}\mid v')=\int_{\Mt}\forcefunction(g_0\circ\varphi^{-1}(t))\hspace{1pt}(-\mathrm{div} \,v'),\quad v'\in V,\quad \Mt=\varphi(t)(M_0).
		\end{array}
	\right. 
\end{equation}
It was studied in \cite{hsieh2020mechanistic} and proved to have solutions over arbitrary time intervals. In the same paper, the issue of identifying  $g_0$ among a parametrized family of candidates given the initial state $M_0$ and a deformed state at time $T$,  $M_T=\varphi(T,M_0)$, was also considered.

The assumption that shape change is driven by a strict advection of the function $g_0$ can be seen as overly restrictive, as it does not allow for independent transformations and external factors possibly affecting this function. 
%process  does not change as time passes, while in fact, prions multiply, and are transmitted from neuron to neuron through axons and neural connections (CITATION?). 
In this paper, we consider a
%This can be modeled by letting the density of the disease propagate through a 
\textit{reaction-diffusion} equation on the moving domain $M_t$ whose solution also controls the shape motion.
%and of the vector field whose motion  which is a well-known type of partial differential equations. However, in our case, the domain on which we work is not only mobile, its motions are induced at each time by the solution of the PDE. 
This kind of coupling, as far as we are aware, has not appeared in the literature.

\paragraph{Reaction-diffusion model.} 
Let us start with a fixed domain $U$ in $\R^\dm$, and consider $\priodensity:[0,T]\rightarrow \C^2(U)$. One can think as  $\priodensity(t,x)$ as some measure of the ``density of a disease" at time $t$ and location $x$ with respect to the Lebesgue measure. A reaction-diffusion equation on $\priodensity$ in the fixed domain $U$ is given by
\begin{equation}\label{eq:fixedrd}
\partial_t\priodensity(t,x)=\mathrm{div}\big(\diffusiontensor(t,x) \nabla \priodensity(t,x)\big)+\reaction(\priodensity(t,x)),\quad a.e.\ t\in[0,T],\ x\in U,
\end{equation}
with given initial value $\priodensity(0)=\priodensity_0:U\rightarrow\R$, and the Neumann boundary condition $\diffusiontensor(t,x) \nabla \priodensity(t,x)=0$ for all time $t$ and $x$ in the boundary $\partial U$. It is understood that the gradient $\nabla$ and the divergence are taken with respect to the $x$ coordinates. %Let us detail the meaning of each term on right-hand side separately.

On the right-hand side,  $\diffusiontensor(t,x)$ is a 3-by-3 symmetric positive definite matrix for each $t$ and $x$. For example, for $\diffusiontensor(t,x)=\diffusiontensor=I_3$, the identity matrix, we get $\mathrm{div}(\diffusiontensor\nabla \priodensity(t,x))=\Delta \priodensity(t,x)$ the Laplacian of $\priodensity$. More generally, for a diffusion at $x$ with rate $\diffusioncoef_i>0$ in an $i$-th direction, $i=1,2,3$, we have $$
\diffusiontensor(t,x)=\diffusioncoef_1e_1e_1^T+\diffusioncoef_2e_2e_2^T+\diffusioncoef_3e_3e_3^T=\mathrm{diag}(\diffusioncoef_1,\diffusioncoef_2,\diffusioncoef_3)=(e_1\ e_2\ e_3)\mathrm{diag}(\diffusioncoef_1,\diffusioncoef_2,\diffusioncoef_3)(e_1\ e_2\ e_3)^T
$$
where $e_i$ is a unit vector pointing to the $i$-th direction and $\mathrm{diag}(\diffusioncoef_1,\diffusioncoef_2,\diffusioncoef_3)$ is the diagonal matrix with corresponding entries. 
%In \cite{hsieh2020mechanistic}, we considered a laminar model reminiscent of the organization of the cortex in cortical surfaces and hypercolumns. 
In this paper, we will work under the following general assumption on $S$ and how it is affected by shape change. We consider a time-dependent field of frames $(t,x)\in[0,T]\times U\mapsto \frame(t,x)\in \gl_3(\R)$ of unit vectors, and let
$$
\diffusiontensor(t,x)=\frame(t,x)\,\mathrm{diag}(\diffusioncoef_1,\diffusioncoef_2,\diffusioncoef_3)\,\frame(t,x)^T.
$$

Finally, $\reaction:\R\rightarrow\R$ is the reaction function,  and models external %growth or atrophy 
factors affecting the function $\priodensity$. It typically satisfies $R(0) = 0$ so that $p\equiv 0$ is a solution of the PDE initialized with $p_0 \equiv 0$. It may have a sigmoidal shape (such that $R(t) = 0$ if $t\leq 0$, and increases on $[0, +\infty)$ with a finite limit at $+\infty$), which results in a growth/atrophy model in which change accelerates until reaching a limit speed. Alternatively, in order to model a growth/atrophy phase over a finite time interval, $R$ on $[0, +\infty)$ may increase to maximal value before decreasing again to 0 (this is the model chosen in Figure \ref{fig:expl}).  
%in our case, the reproduction behavior of the disease as a function of the disease density. 

%$\boxed{\text{ADD MORE ON R?}}$

\paragraph{Integral formulation, and reaction-diffusion on a moving domain.} Equation \eqref{eq:fixedrd} can be written in integral form leading to the weak formulation that we will study specifically. After integrating the equation on a smaller domain, and using the divergence theorem, we can say that the density $(t,x)\mapsto\priodensity(t,x)$ is a solution of Equation \eqref{eq:fixedrd} if and only if, for every domain $U'\subset U$, 
$$
\ddt\int_{U'}\priodensity(t,x)dx = \int_{\partial U'}(\diffusiontensor(t,x)\nabla \priodensity(t,x))^Tn_{\partial U'}(x)d\sigma_{\partial U'}(x) + \int_{U'}\reaction(\priodensity(t,x))dx,
$$
with $n_{\partial U'}$ the outer normal to the boundary of $U'$ and $\sigma_{\partial U'}$ the surface measure on $\partial U'$. 
In other words, the total reduction of $\priodensity$ within $U'$ is equal to the flux of its gradient  along the boundary, modified by the diffusion tensor which takes into account the directions and speed of diffusion. To this is added the total amount created in $U'$ from the reaction $\reaction$.

From our fixed initial volume $\Mzero$, we can give a corresponding formulation for the evolution of a  density on a moving domain $t\mapsto \Mt=\varphi(t)(\Mzero)$, with $\varphi\in H^1([0,T],\D^\reg(\R^\dm))$.  We need, however, to account for changes in the directions of diffusion as the shape is deformed. 

First, we define along each deformation $\varphi(\Mzero)$  of $\Mzero$, with $\varphi\in \D^\reg(\R^\dm)$, a frame $\frame_\varphi$ of unit vectors along $\varphi(\Mzero)$. In other words, $\frame_\varphi$ is a mapping from $\varphi(\Mzero)$ onto $\gl_3(\R)$ whose columns have constant length 1. %This frame describes the layered structure of the brain (see \cite{hsieh2020mechanistic} for an example of such a mapping $\frame$). 
We define a corresponding diffusion tensor $\diffusiontensor_\varphi=\frame_\varphi\,\mathrm{diag}(\diffusioncoef_1,\diffusioncoef_2,\diffusioncoef_3)\frame_\varphi^T.$

Then, we say that the time-dependent  density $\priodensity(t):\Mt=\varphi(t, \Mzero)\rightarrow \R$ with respect to the Lebesgue measure is a solution of the reaction-diffusion equation along the moving domain $t\mapsto\Mt$ if, for every open subset $U_0\subset \Mzero $, and almost every $t$ in $[0,T]$, the equation %$t\mapsto \int_{\varphi(t)(U_0)}\priodensity(t,x)dx$ satisfies
\begin{equation}
\label{eq:weakrdmobile}\begin{aligned}
&\ddt\int_{\varphi(t, \Mzero)}\priodensity(t,x)dx =&&\int_{ \varphi(t, \partial \Mzero)}(\diffusiontensor_{\varphi(t)}(x)\nabla \priodensity(t,x))^Tn_{\varphi(t, \partial \Mzero)}(x)d\sigma_{ \varphi(t, \partial \Mzero)}(x)\\&& +& \int_{\varphi(t, \Mzero)}\reaction(\priodensity(t,x))dx,
\end{aligned}
\end{equation}
is satisfied with Neumann boundary conditions $(\diffusiontensor_{\varphi(t)}(t,x)\nabla \priodensity(t,x))^Tn_{\partial\varphi(t, \Mzero)}(x)=0$ for every $t$ in $[0,T]$ and $x$ in $\partial\Mt=\varphi(t, \partial \Mzero).$

Turning this integral formulation into a pointwise one is difficult, because the support of $\priodensity(t)$ changes as $t$ increases. This results from considering $\priodensity$ in spatial (i.e.,  Eulerian) coordinates. It is easier to deduce the correct PDE for the corresponding density, denoted $\lagdensity$, in material (i.e.,  Lagrangian) coordinates. This density is the pull-back $\varphi(t)^*\priodensity(t)$ of $\priodensity$ through $\varphi(t)$: 
$$ 
\forall f\in L^1(\Mzero),\quad \int_{\Mzero} \lagdensity(t)f dx=\int_{\Mzero} [\varphi(t)^*\priodensity(t)]f dx=\int_{\varphi(t)(\Mzero)}\priodensity(t)f\circ\varphi(t)^{-1}dx.
$$
In other words, $\lagdensity(t)=\priodensity(t)\circ\varphi\,\jac\varphi(t)$, with $\jac\varphi(t)=\det(D\varphi(t))$, the Jacobian of $\varphi(t).$

Note that we get 
$$
\nabla\lagdensity(t)=d\varphi(t)^T\nabla\priodensity(t)\circ\varphi(t)\,\jac\varphi(t)+\lagdensity(t)\frac{\nabla \jac\varphi(t)}{\jac\varphi(t)},
$$ 
so that
\begin{equation}\label{eq:laggradient}
    \nabla\priodensity(t)\circ\varphi(t)=d\varphi(t)^{-T}\left(\nabla\lagdensity(t)-\lagdensity(t)\frac{\nabla \jac\varphi(t)}{\jac\varphi(t)}\right).
\end{equation}

Performing in Equation \eqref{eq:weakrdmobile} the change of variable 
$$
x=\varphi(t,y),
$$
so that
$$
dx=\jac\varphi(t,y)dy,\quad y\in \Mzero,
$$
and
$$
n_{{\varphi(t)}(\partial U_0)}(x)d\sigma_{ {\varphi(t)}(\partial U_0)}(x)
=\jac\varphi(t,y)D\varphi(t,y)^{-T}n_{\partial U_0}(y)d\sigma_{\partial U_0}(y),\quad y\in \partial\Mzero,
$$
we obtain an identity on the fixed domain $U_0$:
%Equation \eqref{eq:weakrdmobile} transforms into
$$\begin{aligned}
\ddt\int_{U_0}\lagdensity(t) =&\ \int_{ \partial U_0}\jac\varphi(t)\left[D\varphi^{-1}(t)\diffusiontensor_{\varphi(t)}\circ\varphi(t)D\varphi(t)^{-T}\left(\nabla\lagdensity(t)-\lagdensity\frac{\nabla \jac\varphi(t)}{\jac\varphi(t)}\right)\right]^Tn_{\partial U_0}d\sigma_{U_0}\\ 
&+ \int_{U_0}\reaction\left(\frac{\lagdensity(t)}{\jac\varphi(t)}\right)\jac\varphi(t).
\end{aligned}
$$
Since the pull-back of a vector field $v:x\mapsto v(x)$ by $\varphi(t)$ is $\varphi(t)^* v:x\mapsto D\varphi(t,x)^{-1}v(\varphi(x))$, the pull-back of the frame field $\frame_{\varphi(t)}$ is $\lagframe_{\varphi(t)}=\varphi(t)^*\frame_{\varphi(t)}D\varphi(t)^{-1}\frame_{\varphi(t)}\circ\varphi(t)$, which means that the pull-back of the diffusion tensor is $\newdifftensor_{\varphi(t)}=\varphi(t)^*S_{\varphi(t)}$ and given by
$$
\newdifftensor_{\varphi(t)}=D\varphi^{-1}(t)\diffusiontensor_{\varphi(t)}\circ\varphi(t)D\varphi(t)^{-T}.
$$
Note that this formula is valid even when replacing $\varphi(t)$ by any diffeomorphism $\phi\in \D^\reg(\R^\dm)$. 

With this new notation, the integral equation reads
$$
\ddt\int_{\Mzero}\lagdensity(t) =\ \int_{ \partial \Mzero}\jac\varphi(t)\left[\newdifftensor_{\varphi(t)}\left(\nabla\lagdensity(t)-\lagdensity(t)\frac{\nabla \jac\varphi(t)}{\jac\varphi(t)}\right)\right]^Tn_{\partial U_0}d\sigma_{U_0}
+ \int_{\Mzero}\reaction\left(\frac{\lagdensity(t)}{\jac\varphi(t)}\right)\jac\varphi(t).
$$
The boundary conditions are then 
$$
\left[\newdifftensor_{\varphi(t)}(x)\left(\nabla\lagdensity(t)-\lagdensity(t)\frac{\nabla \jac\varphi(t)}{\jac\varphi(t)}\right)\right]^Tn_{\partial \Mzero}(x) = 0 ,\quad t\in[0,T],\ x\in \partial\Mzero.
$$

From there, the divergence theorem yields
$$
\ddt\int_{\Mzero}\lagdensity(t,x)dx =\int_{\Mzero}\mathrm{div}\left(\jac\varphi(t)\,\newdifftensor_{\varphi(t)}\left(\nabla\lagdensity(t)-\lagdensity(t)\frac{\nabla \jac\varphi(t)}{\jac\varphi(t)}\right)\right)dx+\int_{\Mzero}\reaction\left(\frac{\lagdensity(t)}{\jac\varphi(t)}\right)\jac\varphi(t)dx,
$$
with boundary condition 
$$
\left[\newdifftensor_{\varphi(t)}(t,x)\nabla\lagdensity(t,x)\right]^T n_0(x)=0,\quad t\in[0,T],\ x\in \partial\Mzero.
$$
Since this should be true for every open $U_0\subset \Mzero$, we get the PDE
\begin{equation}
\label{eq:strongrdmobile}
\ddt\lagdensity(t,x)=\mathrm{div}\left(\jac\varphi(t,x)\,\newdifftensor_{\varphi(t,x)}\left(\nabla\lagdensity(t,x)-\lagdensity(t,x)\frac{\nabla \jac\varphi(t,x)}{\jac\varphi(t,x)}\right)\right)+\reaction\left(\frac{\lagdensity(t,x)}{\jac\varphi(t,x)}\right)\jac\varphi(t,x),
\end{equation}
with boundary condition
$$
\left[\newdifftensor_\varphi(t,x)\left(\nabla\lagdensity(t,x)-\lagdensity(t,x)\frac{\nabla \jac\varphi(t,x)}{\jac\varphi(t,x)}\right)\right]n_0(x)=0, \quad (t,x)\in[0,T]\times \partial\Mzero.$$

\paragraph{PDE-controlled diffeomorphic equation and main result.} Combining the various paragraphs of this section, we obtain a formulation of our model. We start by redefining our various functions and operators.

We fix an initial domain $\Mzero\subset \R^\dm$ and diffusion speeds $\diffusioncoef_1,\diffusioncoef_2,\diffusioncoef_3>0$. For every $\varphi\in\D^\reg(\R^\dm)$, we define
\begin{itemize}
    \item A frame field of unit vectors $\frame_\varphi:\varphi(\Mzero)\rightarrow \gl_3(\R)$ along $\varphi(\Mzero)$, and the corresponding field in spatial coordinates $\lagframe_\varphi=D\varphi^{-1}\frame_\varphi\circ\varphi$.
    \item A diffusion tensor $\diffusiontensor_\varphi:\varphi(\Mzero)\rightarrow M_3(\R),$ with $\diffusiontensor_\varphi(x)=\frame_\varphi\mathrm{diag}(\diffusioncoef_1,\diffusioncoef_2,\diffusioncoef_3)\frame_\varphi^T$, a symmetric positive definite matrix at each point. The corresponding operator in Lagrangian coordinates is $\newdifftensor_\varphi=D\varphi^{-1}\diffusiontensor\circ\varphi\, D\varphi^{-T}.$
    \item A symmetric, positive-definite tensor $\elastictensor_\varphi\in \mathcal{L}(H^1(\varphi(\Mzero),\R^\dm),H^1(\varphi(\Mzero),\R^\dm)^*)$.
\end{itemize}
The PDE-controlled diffeomorphic model with initial condition  $\lagdensity_0:\Mzero\rightarrow\R$ is the system of coupled equations on $\varphi:[0,T]\mapsto\D^\reg(\R^\dm)$ and $\lagdensity:[0,T]\times \Mzero\rightarrow\R$ that follows: for all $(t,x)\in [0,T]\times\Mzero$,
\begin{equation}
    \label{eq:PDELDDMM}
    \left\lbrace
    \begin{aligned}
    &\ddt\lagdensity(t,x)= \mathrm{div}\left(\jac\varphi(t,x)\,\newdifftensor_{\varphi(t,x)}\left(\nabla\lagdensity(t,x)-\lagdensity(t,x)\frac{\nabla \jac\varphi(t,x)}{\jac\varphi(t,x)}\right)\right)+\reaction\left(\frac{\lagdensity(t,x)}{\jac\varphi(t,x)}\right)\jac\varphi(t,x),
    \\
    &\ddt \varphi(t,x)
	= 
	v(t,\varphi(t,x) ),
	\\
	&v(t)=\underset{v' \,\in\, V}{\arg\min} \ \frac{\weight}{2} \, \lform{L_V v'}{v'}
	+
	\frac{1}{2} \, (\elastictensor_{\varphi(t)} \hspace{1pt} v' \mid v')
	-
	(\jerk(t) \mid v'),
	\\
	&(\jerk(t)\mid v')=\int_{\Mt} \forcefunction(\priodensity(t))\hspace{1pt}(-\mathrm{div} \,v'),\quad v'\in V,\quad \Mt=\varphi(t)(M_0),\quad t\in[0,T],
    \end{aligned}\right.
\end{equation}
where $ \priodensity(t)=\lagdensity\circ\varphi(t)^{-1}/\jac \varphi(t)$, with boundary conditions
\begin{equation}
\label{eq:PDELDDMM_bd_conditions}
\left\lbrace
\begin{aligned}
&\left(\newdifftensor_\varphi(t,x)\left(\nabla\lagdensity(t,x)-\lagdensity(t,x)\frac{\nabla \jac\varphi(t,x)}{\jac\varphi(t,x)}\right)\right)^Tn_0(x)=0,&\quad (t,x)\in [0,T]\times \partial \Mzero\\
&\lagdensity(0,x)=\priodensity(0,x)=\priodensity_0(x)&\quad x\in \Mzero\\
&\varphi(0,x) = x,&\quad x\in \Mzero
\end{aligned}
\right.
\end{equation}
By a solution of the above system of differential equations and boundary conditions, we mean a couple $(\varphi,\lagdensity) \in H^1([0,T],\D^\reg(\R^\dm)) \times L^2([0,T],H^1(\Mzero))$ such that $\lagdensity$ is a weak solution of the reaction-diffusion PDE (c.f., next section for the precise definition) with the two first boundary conditions in \eqref{eq:PDELDDMM_bd_conditions} and, for almost all $t\in[0,T]$, $\varphi$ verifies the last three equations in \eqref{eq:PDELDDMM} with the last boundary condition in \eqref{eq:PDELDDMM_bd_conditions}. Our main result is the following existence and uniqueness of the solution under adequate assumptions: 

\begin{theorem}
\label{thm:main}
Assume that $V \hookrightarrow C_0^{\reg+1}(\mathbb{R}^\dm, \mathbb{R}^\dm)$ with $\reg\geq2$, that $\reaction$ and $\forcefunction$ are Lipshitz and bounded, and that on every bounded subset $B$ of $\D^\reg(\R^\dm)$ for the distance $d_{\reg,\infty}$, we have:
\begin{enumerate}
    \item The linear tensor $\varphi\mapsto \elastictensor_\varphi$ is Lipshitz on $B$.
    \item $\Vert \frame_\varphi\Vert_\infty=\sup_{x\in\varphi(\Mzero)}\Vert\frame_\varphi(x)^{-1}\Vert$ is bounded on $B$.
\end{enumerate}
Then, for every $\priodensity_0$ in $L^2(\Mzero),$ there is a unique solution $(\varphi,\lagdensity) \in H^1([0,T],\D^\reg(\R^\dm)) \times L^2([0,T],H^1(\Mzero))$ to \eqref{eq:PDELDDMM}.
\end{theorem}
%\nicolas{REMAIN TO DO: what we mean by weak solution (different from weak formulation above), and give a map of the proof.}

Most of the rest of the paper is devoted to the proof of this result, decomposed into the following steps. In section \ref{sec:preliminary}, we fix a time-dependent deformation $t\mapsto \varphi(t)$ and show the local weak existence and uniqueness of solutions to the reaction-diffusion equation on the corresponding moving domain. Then in section \ref{sec:theorem_proof}, we derive a number of necessary estimates on $\varphi$ which, combined with  section \ref{sec:preliminary}, lead to the result of Theorem \ref{thm:main} by a fixed point argument.

\section{Analysis with prescribed moving domain}
\label{sec:preliminary}
Before studying the fully coupled system \eqref{eq:PDELDDMM}, we will first restrict to the simpler situation of a reaction-diffusion equation on a moving domain but for which the deformation is fixed and prove preliminary results of local and global existence of weak solutions for this case. Note that, in the Lagrangian formulation we consider, this amounts in a system of reaction-diffusion equations with time-dependent diffusion tensor and boundary condition for which several existence and regularity results have been showed in the past, see e.g. \cite{ladyvzenskaja1988linear,burdzy2004heat,goudon2010regularity}. These are however derived with slightly different settings and sets of assumptions than in the present work and thus, for the sake of completeness, we provide detailed proofs of the weak existence results as well as bounds on the solutions that we will need for the proof of our main theorem.

\subsection{Weak existence for the reaction-diffusion PDE on a moving domain}
\label{ssec:local_existence_reaction_diffusion}
In all this section, we assume that $\reg\geq 2$ and we slightly extend our general notation. We let $t_0$ denote the initial time and take $\eta >0$, $[t_0,t_0+\eta]\subset [0,T]$. We assume that an initial  deformation $\varphi_{t_0} \in \D^\reg(\R^\dm)$ is given at $t_0$, together with a time-dependent deformation $\varphi \in H^1([t_0,t_0+\eta],\D^\reg(\R^\dm))$ with $\varphi(t_0) = \id$ (so that $\varphi_{t_0}$ and $\varphi(t_0)$ denote different objects). Both $\varphi_{t_0}$ and $\varphi$ are assumed to be fixed in this section. For convenience, we shift the reference domain to $\Mtzero = \varphi_{t_0}(\Mzero)$. The diffusion tensor $\newdifftensor_{\varphi}$ is then fully specified and given by, for all $t\in [t_0,t_0+\eta]$ and all $x \in \Mtzero$:
\begin{equation}
\label{eq:syst.with.f}
	\newdifftensor_{\varphi}(t,x)
	=
	(D\varphi(t,x))^{-1} \ 
	\left( \vphantom{\sum} F_{\varphi(t) \,\circ\, \varphi_{t_0}} \circ \varphi(t,x) \right) \,
	\mathrm{diag}(\diffusioncoef_1, \diffusioncoef_2, \diffusioncoef_3) \,
	\left( \vphantom{\sum} F_{\varphi(t) \,\circ\, \varphi_{t_0}} \circ \varphi(t,x) \right)^\top
	D\varphi(t,x)^{-\top}
\end{equation}

% This section is dedicated to proving the following result:
% \begin{theorem}
% Suppose that the reaction function $R$ is Lipschitz continuous and that the frame field satisfies
% \[
% 	\sup_{t \,\in\, [t_0, \, t_0 + \eta]} \|F_{\varphi(t) \,\circ\, \varphi_{t_0}}^{-1}\|_\infty < \infty .
% \]
% Then there exists $\lambda(\varphi) > 0$ such that for all $\lagdensitybis_{t_0} \in L^2(M_{t_0})$, the system~\eqref{eq:reaction_diffusion_ivp} has a unique weak solution.
% \end{theorem}
\subsubsection{Preliminary results}
As a first step we consider the simplified setting in which the reaction term is replaced by a time-dependent function $f(t)$, introducing the following system:
\begin{equation}
	\left\{
		\begin{array}{ll}
			\displaystyle
			\ddt\lagdensity(t,x)
			=
			\mathrm{div}\left(\jac\varphi(t,x) \newdifftensor_{\varphi(t)}(t,x)\left(\nabla\lagdensity(t,x)-\lagdensity(t,x)\frac{\nabla \jac\varphi(t,x)}{\jac\varphi(t,x)}\right)\right)
			+
			f(t) 
			\\[10pt]
			\left[\newdifftensor_\varphi(t,x)\left(\nabla\lagdensity(t,x)-\lagdensity(t,x)\frac{\nabla \jac\varphi(t,x)}{\jac\varphi(t,x)}\right)\right]^Tn_0(x)=0,\quad (t,x)\in [t_0,t_0+\eta]\times \partial \Mtzero\\[10pt]
            \lagdensity(t_0,x)= p_{t_0}(x),\quad x\in \Mtzero
		\end{array}
	\right.
	\label{eq:diffusion_ivp}
\end{equation}
We will assume that $f\in L^2([t_0, t_0+\eta], H_1(M_{t_0})^*)$ and rewrite \eqref{eq:diffusion_ivp} in a weak form. We will look for a solution $\lagdensity \in L^2([t_0, t_0+\eta], H_1(M_{t_0}))$. 
Introduce the operator
\[
\mathcal L_{\varphi, \, 0} : L^2([t_0, t_0 + \eta], H^1(M_{t_0})) \rightarrow L^2([t_0, t_0 + \eta], H^1(M_{t_0})^*)
\]
defined by
\[
	(\mathcal{L}_{\varphi, \, 0} \hspace{1pt} \testfun_1 \mid \testfun_2)
	=
	\int_{t_0}^{t_0 + \eta}
	\left(
		\vphantom{\frac{1}{2}}
		\langle \newdifftensor_{\varphi(t)} \, \nabla \testfun_1(t), \nabla \testfun_2(t) \rangle_{L^2}
	-
		\left\langle \testfun_1(t) \, \newdifftensor_{\varphi(t)} \, \frac{\nabla \jac\varphi(t,x)}{\jac\varphi(t,x)}, \nabla \testfun_2(t) \right\rangle_{L^2}
	\right)
	dt
	\]
% 	\begin{align*}
% 	(\mathcal{L}_{\varphi, \, 0} \hspace{1pt} \testfun_1 \mid \testfun_2)
% 	&=
% 	\int_{t_0}^{t_0 + \eta}
% 	\left(
% 		\vphantom{\frac{1}{2}}
% 		\langle \newdifftensor_{\varphi(t)} \, \nabla \testfun_1(t), \nabla \testfun_2(t) \rangle_{L^2}
% 	\right.
% 	\\[5pt]
% 	&\hspace{50pt}
% 	\left.
% 		\phantom{}
% 		-
% 		\left\langle \testfun_1(t) \, \newdifftensor_{\varphi(t)} \, \frac{\nabla \jac\varphi(t,x)}{\jac\varphi(t,x)}, \nabla \testfun_2(t) \right\rangle_{L^2}
% 	\right)
% 	dt 
% \end{align*}
With this notation, the first equation in \eqref{eq:diffusion_ivp} can be rewritten as
$\wddt \lagdensity + \mathcal L_{\varphi, \, 0}\, \lagdensity = f$ (recall that the notation $\wddt\lagdensity$ refers to the weak derivative of $\lagdensity$ with respect to time)
and the second one is automatically derived from identifying boundary terms after  integration by parts. This yields the new formulation
\begin{equation}
    \label{eq:diffusion_ivp_short}
    \left\{
    \begin{aligned}
    & \wddt \lagdensity + \mathcal L_{\varphi, \, 0}\, \lagdensity = f\\
    & \lagdensity(t_0, x) = p_{t_0}(x)
    \end{aligned}
    \right.
\end{equation}
Note that the first equation implies, in particular, that $\wddt \lagdensity\in L^2([t_0, t_0+\eta], H_1(M_{t_0})^*)$), and Theorem \ref{thm:u_dtu_continuous} ensures that prescribing an initial condition at $t=t_0$ is meaningful. 
For technical reasons, it will be convenient to make the change of function $\lagdensitybis(t,x) = e^{-\lambda t} \lagdensity(t,x)$ (for some $\lambda >0$ to be specified later) and rewrite \eqref{eq:diffusion_ivp_short} in terms of $\lagdensitybis$, yielding:
% \begin{equation}
% 	\left\{
% 		\begin{array}{ll}
% 			\displaystyle
% 			\ddt \lagdensitybis(t,x) + \lambda \lagdensitybis(t,x) 
% 			=
% 			\mathrm{div}\left(\newdifftensor_{\varphi(t)}(t,x) \left(\nabla \lagdensitybis(t,x)-\lagdensitybis(t,x) \frac{\nabla \jac\varphi(t,x)}{\jac\varphi(t,x)}\right)\right)
% 			+
% 			e^{-\lambda t}f(t) 
% 			\\[5pt]
% 		\left(\newdifftensor_{\varphi(t)}(t,x) \left(\nabla \lagdensitybis(t,x)-\lagdensitybis(t,x) \frac{\nabla \jac\varphi(t,x)}{\jac\varphi(t,x)}\right)\right)^T n_{t_0}(x) = 0\\[5pt]
%             \lagdensitybis(t_0,x)=0,\quad x\in \Mtzero
% 		\end{array}
% 	\right.
% 	\label{eq:diffusion_ivp2}
% \end{equation}
\begin{equation}
	\left\{
		\begin{array}{ll}
		
			\displaystyle
			\wddt \lagdensitybis(t,x) + \mathcal L_{\varphi, \, \lambda}\, \lagdensitybis = e^{-\lambda t}f(t) 
			\\[5pt]
% 		\left(\newdifftensor_{\varphi(t)}(t,x) \left(\nabla \lagdensitybis(t,x)-\lagdensitybis(t,x) \frac{\nabla \jac\varphi(t,x)}{\jac\varphi(t,x)}\right)\right)^T n_{t_0}(x) = 0\\[5pt]
             \lagdensitybis(t_0,x)=e^{-\lambda t_0} p_{t_0}(x),\quad x\in \Mtzero
		\end{array}
	\right.
	\label{eq:diffusion_ivp2}
\end{equation}
where $\mathcal L_{\varphi, \, \lambda}: L^2([t_0, t_0 + \eta], H^1(M_{t_0})) \rightarrow L^2([t_0, t_0 + \eta], H^1(M_{t_0})^*)$ is defined by $\mathcal{L}_{\varphi, \, \lambda} \hspace{1pt} \testfun = \lambda \testfun + \mathcal{L}_{\varphi, \, 0} \hspace{1pt} \testfun$.
% \begin{align*}
% 	(\mathcal{L}_{\varphi, \, \lambda} \hspace{1pt} \testfun_1 \mid \testfun_2)
% 	&=
% 	\int_{t_0}^{t_0 + \eta}
% 	\left(
% 		\vphantom{\frac{1}{2}}
% 		\lambda \, \langle \testfun_1(t), \testfun_2(t) \rangle_{L^2}
% 		+
% 		\langle \newdifftensor_{\varphi(t)} \, \nabla \testfun_1(t), \nabla \testfun_2(t) \rangle_{L^2}
% 	\right.
% 	\\[5pt]
% 	&\hspace{50pt}
% 	\left.
% 		\phantom{}
% 		-
% 		\left\langle \testfun_1(t) \, \newdifftensor_{\varphi(t)} \, \frac{\nabla \jac\varphi(t,x)}{\jac\varphi(t,x)}, \nabla \testfun_2(t) \right\rangle_{L^2}
% 	\right)
% 	dt 
% \end{align*}
With this notation and these assumptions, we can now state the main result of this section:
\begin{proposition}
\label{thm:reaction_diffusion_fixed_phi_no_R}
With the assumptions above, for all $\priodensity_{t_0} \in L^2(M_{t_0})$, the system~\eqref{eq:diffusion_ivp2} has a unique weak solution on $[t_0,t_0+\eta]$.
\end{proposition}

\medskip

We first address the case of an homogeneous initial condition with the following lemma.
\begin{lemma}
\label{thm:diffusion_nonhomogeneous_initialZero}
Suppose that the frame field satisfies
\[
	\sup_{t \,\in\, [t_0, \, t_0 + \eta]} \|F_{\varphi(t) \,\circ\, \varphi_{t_0}}^{-1}\|_\infty < \infty .
\]
Then there exists $\lambda(\varphi) > 0$ such that for any $g\in L^2([t_0, t_0 + \eta], H^1(M_{t_0})^*))$, the problem 
\begin{equation*}
	\left\{
		\begin{array}{ll}
			\wddt \lagdensitybis + \mathcal{L}_{\varphi, \, \lambda} \lagdensitybis = g\\
			\lagdensitybis(t_0) = 0
		\end{array}
	\right.
\end{equation*}
has a unique weak solution $\lagdensitybis$ that belongs to $L^2([t_0, t_0 + \eta], H^1(M_{t_0}))$.
\end{lemma}
\begin{proof}
The proof is mainly an application of Theorem \ref{thm:abstract_parabolic_ivp}. We only need to choose $\lambda > 0$ such that the operator $\mathcal{L}_{\varphi, \, \lambda}$ is bounded and coercive. Since $\varphi \in H^1([t_0,t_0+\eta],\D^\reg(\R^\dm))$, we have $\varphi \in C([t_0,t_0+\eta],\D^\reg(\R^\dm))$ and as $s \geq 2$
\[
	B_\varphi
	\colonequals
	\max\left\{
		\max_{t \,\in\, [t_0, \, t_0 + \eta]} \|\varphi(t) - \mathit{id}\|_{2, \infty}, 
		\max_{t \,\in\, [t_0, \, t_0 + \eta]} \|\varphi^{-1}(t) - \mathit{id}\|_{1, \infty}
	\right\} < \infty .
\]
Recall also that
\[
	\newdifftensor_{\varphi}(t,\cdot)
	=
	(D\varphi(t))^{-1} \ 
	\left( \vphantom{\sum} F_{\varphi(t) \,\circ\, \varphi_{t_0}} \circ \varphi(t) \right) \,
	\mathrm{diag}(r_1, r_2, r_3) \,
	\left( \vphantom{\sum} F_{\varphi(t) \,\circ\, \varphi_{t_0}} \circ \varphi(t) \right)^\top
	D\varphi(t)^{-\top}
\]
and the columns of $F_{\varphi(t) \,\circ\, \varphi_{t_0}}(x)$ are unit vectors, so
\[
	\|\newdifftensor_{\varphi(t)}\|_\infty \leq \genC \, B_\varphi^2
	\ \ \mbox{ for all } t \in [t_0, t_0 + \eta]
\]
and
\[
	\|\newdifftensor_{\varphi(t)}^{-1}\|_\infty \leq \genC B_\varphi^2 \left( \sup_{t \,\in\, [t_0, \, t_0 + \eta]} \|F_{\varphi(t) \,\circ\, \varphi_{t_0}}^{-1}\|_\infty \right)^2
	\ \ \mbox{ for all } t \in [t_0, t_0 + \eta] .
\]
(Recall that $\genC$ is our notation for a generic constant.)
It follows that there exist two constants $\alpha_\varphi$ and $\beta_\varphi$ (depending on $\varphi$) such that $\alpha_\varphi |z|^2 \leq z^T \newdifftensor_{\varphi} z \leq \beta_\varphi |z|^2 $ for all $t\in [t_0, t_0 + \eta]$ and $z \in \R^\dm$ and therefore we have
\begin{align*}
	&\hspace{15pt}
	|(\mathcal{L}_{\varphi, \, \lambda} \hspace{1pt} \testfun_1 \mid \testfun_2)|
	\\[3pt]
	&=
	\left|
	\int_{t_0}^{t_0 + \eta}
	\left( \vphantom{\sum}
		\lambda \, \langle \testfun_1(t), \testfun_2(t) \rangle_{L^2}
		+
		\langle \newdifftensor_{\varphi(t)} \, \nabla \testfun_1(t), \nabla \testfun_2(t) \rangle_{L^2}
		-
		\left\langle \testfun_1(t) \, \newdifftensor_{\varphi(t)} \, \frac{\nabla \jac\varphi(t,x)}{\jac\varphi(t,x)}, \nabla \testfun_2(t) \right\rangle_{L^2}
	\right)
	dt
	\right|
	\\[3pt]
	&\leq
	(\lambda + \genC[\varphi])
	\int_{t_0}^{t_0 + \eta}
	\left( \vphantom{\sum} 
		\|\testfun_1(t)\|_{L^2} \, \|\testfun_2(t)\|_{L^2}
		+
		\|\nabla \testfun_1(t)\|_{L^2} \, \|\nabla \testfun_2(t)\|_{L^2}
		+
		\|\testfun_1(t)\|_{L^2} \, \|\nabla \testfun_2(t)\|_{L^2}
	\right) dt
	\\[3pt]
	&\leq
	(\lambda + \genC[\varphi])
	\int_{t_0}^{t_0 + \eta} 3 \, \|\testfun_1(t)\|_{H^1} \, \|\testfun_2(t)\|_{H^1} \, dt
	\\[3pt]
	&\leq
	3 \, (\lambda + \genC[\varphi]) \ \|\testfun_1\|_{L^2([t_0, \, t_0 + \eta], \, H^1(M_{t_0}))} \ \|\testfun_2\|_{L^2([t_0, \, t_0 + \eta], \, H^1(M_{t_0}))}
\end{align*}
which shows the boundedness of the operator $\mathcal{L}_{\varphi, \, \lambda}$. Moreover, for any $\varepsilon>0$:
\begin{align*}
	&\hspace{15pt}
	(\mathcal{L}_{\varphi, \, \lambda} \hspace{1pt} \testfun \mid \testfun)
	\\[3pt]
	&=
	\int_{t_0}^{t_0 + \eta}
	\left(
		\lambda \, \|\testfun(t)\|_{L^2}^2
		+
		\langle \newdifftensor_{\varphi(t)} \, \nabla \testfun(t), \nabla \testfun(t) \rangle_{L^2}
		-
		\left\langle \testfun(t) \, \newdifftensor_{\varphi(t)} \, \frac{\nabla \jac\varphi(t,x)}{\jac\varphi(t,x)}, \nabla \testfun(t) \right\rangle_{L^2}
	\right)
	dt
	\\[3pt]
	&\geq
	\int_{t_0}^{t_0 + \eta}
	\left(
		\lambda \, \|\testfun(t)\|_{L^2}^2
		+
		\beta_\varphi \, \|\nabla \testfun(t)\|_{L^2}^2
		-
		\genC[\varphi] \, \|\testfun(t)\|_{L^2} \, \|\nabla \testfun(t)\|_{L^2}
	\right)
	dt
	\\[3pt]
	&\geq
	\int_{t_0}^{t_0 + \eta}
	\left(
		\lambda \, \|\testfun(t)\|_{L^2}^2
		+
		\beta_\varphi \, \|\nabla \testfun(t)\|_{L^2}^2
		-
		\genC[\varphi] \left( \frac{1}{2\varepsilon} \, \|\testfun(t)\|_{L^2}^2 + \frac{\varepsilon}{2} \, \|\nabla \testfun(t)\|_{L^2}^2 \right)
	\right)
	dt
	\\[3pt]
	&=
	\int_{t_0}^{t_0 + \eta}
	\left(
		\left( \lambda - \frac{\genC[\varphi]}{2\varepsilon} \right) \|\testfun(t)\|_{L^2}^2
		+
		\left( \beta_\varphi - \frac{\genC[\varphi] \, \varepsilon}{2} \right) \|\nabla \testfun(t)\|_{L^2}^2
	\right)
	dt .
\end{align*}
Now, choose $\varepsilon(\varphi) > 0$ and $\lambda(\varphi) > 0$ such that
\[
	\beta_\varphi - \frac{\genC[\varphi] \, \varepsilon}{2} > 0
	\ \ \mbox{ and } \ \ 
	\lambda - \frac{\genC[\varphi]}{2\varepsilon} > 0 
\]
and the above inequality leads to
\[
	(\mathcal{L}_{\varphi, \, \lambda} \hspace{1pt} \testfun \mid \testfun)
	\geq
	\min\left\{\lambda - \frac{\genC[\varphi]}{2\varepsilon}, \, \beta_\varphi - \frac{\genC[\varphi] \, \varepsilon}{2}\right\} \,
	\|\testfun\|_{L^2([t_0, \, t_1], \, H^1(M_{t_0}))}^2 .
\]
This shows that of $\mathcal{L}_{\varphi, \, \lambda}$ is coercive and concludes the proof of Lemma \ref{thm:diffusion_nonhomogeneous_initialZero}. 
%As $\min\left\{\lambda - \frac{C_\varphi'}{2\varepsilon}, \, C_\varphi - \frac{C_\varphi' \, \varepsilon}{2}\right\} >0$, we see that $\mathcal{L}_{\varphi, \, \lambda}$ is coercive and Theorem 4.1 in \cite{magenes1972non} gives the existence of a unique weak solution $\lagdensitybis \in L^2([t_0, t_0 + \eta], H^1(M_{t_0}))$.
\end{proof}
Existence and uniqueness in the case of a non-homogeneous initial condition follows as a corollary.
\begin{corollary}
\label{cor:diffusion_nonhomogeneous}
Under the same assumptions as in Lemma~\ref{thm:diffusion_nonhomogeneous_initialZero}, there exists $\lambda(\varphi) > 0$ such that for all $g\in L^2([t_0, t_0 + \eta], H^1(M_{t_0})^*)$ and $\priodensitybis_{t_0} \in L^2(M_{t_0})$, the problem
\begin{equation}
\label{eq:diffusion_nonhomogeneous}
	\left\{
		\begin{array}{ll}
			\wddt \lagdensitybis + \mathcal{L}_{\varphi, \, \lambda} \lagdensitybis = g \\
			\lagdensitybis(t_0) = \priodensitybis_{t_0}
		\end{array}
	\right.
\end{equation}
has a unique weak solution $\lagdensitybis \in L^2([t_0, t_0+\eta], H^1(M_{t_0})) \cap C([t_0, t_0+\eta], L^2(M_{t_0}))$.
\end{corollary}
\begin{proof}
According to Theorem 3.2, Chapter 1, in \cite{magenes1972non}, there exists $w \in L^2([t_0, t_0 + \eta], H^1(M_{t_0}))$ such that $\partial_t \hspace{1pt} w \in L^2([t_0, t_0 + \eta], H^1(M_{t_0})^*)$ and $w(0) = \priodensitybis_{t_0}$. Using the previous lemma, let $\tilde{\lagdensitybis}$ be the solution to
\[
	\left\{
		\begin{array}{ll}
			\partial_t \hspace{1pt} \lagdensitybis + \mathcal{L}_{\varphi, \, \lambda} \hspace{1pt} \lagdensitybis
			=
			g - \partial_t \hspace{1pt} w - \mathcal{L}_{\varphi, \, \lambda} \hspace{1pt} w \\
			\lagdensitybis(t_0) = 0
		\end{array}
	\right. ,
\]
then $\tilde{\lagdensitybis} + w$ is a solution to the original problem. Uniqueness is clear from the uniqueness in Lemma \ref{thm:diffusion_nonhomogeneous_initialZero}.
\end{proof}

\subsubsection{Existence of solutions on a moving domain}
Still assuming that $t\mapsto \varphi(t)$ is given, we now consider general reaction-diffusions on the moving domain, still using a weak formulation, which becomes, introducing the nonlinear operator $\mathcal R_{\varphi,0}: \lagdensity(t,x) \mapsto J\varphi \,R\big(\lagdensity(t,x)/J\varphi(t,x)\Big)$
\begin{equation}
    \label{eq:reaction_diffusion_ivp}
    \left\{
    \begin{aligned}
    & \wddt \lagdensity + \mathcal L_{\varphi, \, 0}\, \lagdensity = \mathcal R_{\varphi,0}(\lagdensity)\\
    & \lagdensity(t_0, x) = \lagdensity_{t_0}(x)
    \end{aligned}
    \right.
\end{equation}

\begin{theorem}
\label{thm:reaction_diffusion_fixed_phi}
Suppose that the reaction function $R$ is Lipschitz continuous. Under the conditions of Lemma~\ref{thm:diffusion_nonhomogeneous_initialZero}, for all $\priodensity_{t_0} \in L^2(M_{t_0})$, the system~\eqref{eq:reaction_diffusion_ivp} has a unique weak solution on $[t_0,t_0+\eta]$.
\end{theorem}
% go back to the original reaction-diffusion system given by: 
% %Relying again on the change of function $\lagdensitybis(t,x) = e^{-\lambda t} \lagdensity(t,x)$, the reaction-diffusion equation with moving domain becomes:
% \begin{equation}
% 	\left\{
% 		\begin{array}{ll}
% 			\displaystyle
% 			\ddt \lagdensity(t,x)
% 			=
% 			\mathrm{div}\left(\newdifftensor_{\varphi(t)}(t,x) \left(\nabla \lagdensity(t,x)-\lagdensity(t,x) \frac{\nabla \jac\varphi(t,x)}{\jac\varphi(t,x)}\right)\right)
% 			+
% 			\reaction\left(\frac{\lagdensity(t,x)}{\jac\varphi(t,x)}\right) \jac\varphi(t,x)
% 			\\[5pt]
% 		\left(\newdifftensor_{\varphi(t)}(t,x) \left(\nabla \lagdensity(t,x)-\lagdensity(t,x) \frac{\nabla \jac\varphi(t,x)}{\jac\varphi(t,x)}\right)\right)^T n_{t_0}(x) = 0\\[5pt]
%             \lagdensity(t_0,x)= \priodensity_{t_0},\quad x\in \Mtzero
% 		\end{array}
% 	\right.
% 	\label{eq:reaction_diffusion_ivp}
% \end{equation}

% \begin{theorem}
% \label{thm:reaction_diffusion_fixed_phi}
% Suppose that the reaction function $R$ is Lipschitz continuous. Under the conditions of Lemma~\ref{thm:diffusion_nonhomogeneous_initialZero}, for all $\priodensity_{t_0} \in L^2(M_{t_0})$, the system~\eqref{eq:reaction_diffusion_ivp} has a unique weak solution on $[t_0,t_0+\eta]$.
% \end{theorem}
\begin{proof}
Let us fix $\priodensity_{t_0} \in L^2(M_{t_0})$. Relying again on the change of function $\lagdensitybis(t,x) = e^{-\lambda t} \lagdensity(t,x)$, the reaction-diffusion system \eqref{eq:reaction_diffusion_ivp} becomes:
\begin{equation}
    \label{eq:reaction_diffusion_ivp2}
    \left\{
    \begin{aligned}
    & \wddt \lagdensitybis + \mathcal L_{\varphi, \, \lambda}\, \lagdensitybis = \mathcal R_{\varphi, \lambda}(\lagdensitybis)\\
    & \lagdensitybis(t_0, x) = \lagdensitybis_{t_0}(x)
    \end{aligned}
    \right.
\end{equation}
with $\lagdensitybis_{t_0} = e^{-\lambda t_0}\lagdensity_{t_0}$ and $
\mathcal R_{\varphi, \lambda}(\lagdensitybis)(t,x)  = e^{-\lambda t}\mathcal R_{\varphi,0}(e^{\lambda t}\lagdensitybis)(t,x).
$

% \begin{equation}
% 	\left\{
% 		\begin{array}{ll}
% 			\displaystyle
% 			\ddt \lagdensitybis(t,x) + \lambda \lagdensitybis(t,x) 
% 			=
% 			\mathrm{div}\left(\newdifftensor_{\varphi(t)}(t,x) \left(\nabla \lagdensitybis(t,x)-\lagdensitybis(t,x) \frac{\nabla \jac\varphi(t,x)}{\jac\varphi(t,x)}\right)\right)
% 			+
% 			e^{-\lambda t}\reaction\left(\frac{e^{\lambda t}\lagdensitybis(t,x)}{\jac\varphi(t,x)}\right) \jac\varphi(t,x)
% 			\\[5pt]
% 		\left(\newdifftensor_{\varphi(t)}(t,x) \left(\nabla \lagdensitybis(t,x)-\lagdensitybis(t,x) \frac{\nabla \jac\varphi(t,x)}{\jac\varphi(t,x)}\right)\right)^T n_{t_0}(x) = 0\\[5pt]
%             \lagdensitybis(t_0,x)= e^{-\lambda t_0} \priodensity_{t_0},\quad x\in \Mtzero
% 		\end{array}
% 	\right.
% 	\label{eq:reaction_diffusion_ivp2}
% \end{equation}
Let $0<\delta\leq \eta$ to be fixed later in the proof and consider the space $X \colonequals C([t_0, t_0 + \delta], L^2(M_{t_0}))$ equipped with the norm
\[
	\|\testfun\|_X \colonequals \max_{t \,\in\, [t_0, \, t_0 + \delta]} \, \|\testfun(t)\|_{L^2(M_{t_0})} .
\]
Given $\testfun \in X$, let
\[
	g_{\lambda, \, \testfun}(t) \colonequals e^{-\lambda t}\reaction\left(\frac{e^{\lambda t}\testfun(t)}{\jac\varphi(t)}\right) \jac\varphi(t),
\]
then $g_{\lambda, \, \testfun} \in L^2([t_0, t_0 + \delta], L^2(M_{t_0}))$,  since $\jac\varphi$ and $1/\jac\varphi$ are uniformly bounded on $[t_0,t_0+\delta] \times M_{t_0}$, and $|\reaction(z)| \leq C \, (1 + |z|)$ by Lipschitz continuity of $\reaction$. From Corollary \ref{cor:diffusion_nonhomogeneous}, we know that there exists $\lambda > 0$ that may depend on $\varphi$ but not on $h$ such that
\[
	\left\{
		\begin{array}{ll}
			\partial_t \hspace{1pt} \lagdensitybis + \mathcal{L}_{\varphi, \, \lambda} \hspace{1pt} \lagdensitybis = g_{\lambda, \, \testfun} \\
			\lagdensitybis(t_0) = e^{-\lambda t_0} \priodensity_{t_0}
		\end{array}
	\right.
\]
has a unique weak solution for all $\testfun \in X$. Denote this weak solution by $\lagdensitybis = \mathcal{S}(\testfun) \in L^2([t_0, t_0 + \delta], H^1(M_{t_0})) \cap X$.

We now show that $\mathcal{S}: X \rightarrow X$ is a contraction when $\delta$ is chosen small enough. Let $\testfun_1,\testfun_2 \in X$ and $\lagdensitybis_1 = \mathcal{S}(\testfun_1)$, $\lagdensitybis_2 = \mathcal{S}(\testfun_2)$, so that
\[
	\partial_t \hspace{1pt} \lagdensitybis_1 + \mathcal{L}_{\varphi, \, \lambda} \hspace{1pt} \lagdensitybis_1 = g_{\lambda, \, \testfun_1} \,\in\, L^2([t_0, t_0 + \delta], H^1(M_{t_0})^*)
\]
and
\[
	\partial_t \hspace{1pt} \lagdensitybis_2 + \mathcal{L}_{\varphi, \, \lambda} \hspace{1pt} \lagdensitybis_2 = g_{\lambda, \, \testfun_2} \,\in\, L^2([t_0, t_0 + \delta], H^1(M_{t_0})^*) .
\]
It follows that for almost all $t \in [t_0, t_0+\delta]$
\[
	\left( \vphantom{\sum} \partial_t \hspace{1pt} (\lagdensitybis_1 - \lagdensitybis_2) \right)\!(t)
	+
	\left( \vphantom{\sum} \mathcal{L}_{\varphi, \, \lambda} \hspace{1pt} (\lagdensitybis_1 - \lagdensitybis_2) \right)\!(t)
	=
	(g_{\lambda, \, \testfun_1} - g_{\lambda, \, \testfun_2})(t) \in L^2(M_{t_0}) \subset H^1(M_{t_0})^*.
\]
Evaluating at $(\lagdensitybis_1 - \lagdensitybis_2)(t) \in H^1(M_{t_0})$ gives
\begin{align}
	\begin{split}
	&\hspace{13pt}
	\left( \left( \vphantom{\sum} \partial_t \hspace{1pt} (\lagdensitybis_1 - \lagdensitybis_2) \right)\!(t) \ \rule[-5pt]{0.5pt}{15pt}\  (\lagdensitybis_1 - \lagdensitybis_2)(t) \right)
	+
	\left( \left( \vphantom{\sum} \mathcal{L}_{\varphi, \, \lambda} \hspace{1pt} (\lagdensitybis_1 - \lagdensitybis_2) \right)\!(t) \ \rule[-5pt]{0.5pt}{15pt}\  (\lagdensitybis_1 - \lagdensitybis_2)(t) \right)
	\\
	&=
	\left( \vphantom{\sum} (g_{\lambda, \, \testfun_1} - g_{\lambda, \, \testfun_2})(t) \ \rule[-5pt]{0.5pt}{15pt}\  (\lagdensitybis_1 - \lagdensitybis_2)(t) \right) .
	\end{split}
	\label{eq:5}
\end{align}
%Since $\partial_t \hspace{1pt} (u_1 - u_2) \in L^2([t_0, t'], L^2(M_{t_0}))$, we actually have
%\begin{align}
%	\begin{split}
%	&\hspace{13pt}
%	\left\langle \left( \vphantom{\sum} \partial_t \hspace{1pt} (u_1 - u_2) \right)\!(t), (u_1 - u_2)(t) \right\rangle_{L^2}
%	+
%	\left( \left( \vphantom{\sum} \mathcal{A} \hspace{1pt} (u_1 - u_2) \right)\!(t) \ \rule[-5pt]{0.5pt}{15pt}\  (u_1 - u_2)(t) \right)
%	\\[5pt]
%	&=
%	\left( \left( \vphantom{\sum} \partial_t \hspace{1pt} (u_1 - u_2) \right)\!(t) \ \rule[-5pt]{0.5pt}{15pt}\  (u_1 - u_2)(t) \right)
%	+
%	\left( \left( \vphantom{\sum} \mathcal{A} \hspace{1pt} (u_1 - u_2) \right)\!(t) \ \rule[-5pt]{0.5pt}{15pt}\  (u_1 - u_2)(t) \right)
%	\\[5pt]
%	&=
%	\left( \vphantom{\sum} (g_{w_1} - g_{w_2})(t) \ \rule[-5pt]{0.5pt}{15pt}\  (u_1 - u_2)(t) \right) .
%	\end{split}
%	\label{eq:5}
%\end{align}
As a result of the coercivity of the operator $\mathcal{L}_{\varphi, \, \lambda}$ shown earlier, for some constant $\genC[\varphi] >0$:
\begin{align}
	\begin{split}
%	&\hspace{13pt}
	\left( \left( \vphantom{\sum} \mathcal{L}_{\varphi, \, \lambda} \hspace{1pt} (u_1 - u_2) \right)\!(t) \ \rule[-5pt]{0.5pt}{15pt}\  (\lagdensitybis_1 - \lagdensitybis_2)(t) \right)
%	\\[5pt]
%	&=
%	\lambda \, \|(u_1 - u_2)(t)\|_{L^2}^2
%	+
%	\langle W_{\varphi(t)} \, \nabla (u_1 - u_2)(t), \nabla (u_1 - u_2)(t) \rangle_{L^2}
%	\\
%	&\hspace{15pt}
%	\phantom{}
%	-
%	\left\langle (u_1 - u_2)(t) \, W_{\varphi(t)} \, \frac{\nabla(\det D\varphi(t))}{\det D\varphi(t)}, \nabla (u_1 - u_2)(t) \right\rangle_{L^2}
%	\\[5pt]
	&\geq
	\genC[\varphi] \, \|(\lagdensitybis_1 - \lagdensitybis_2)(t)\|_{H^1}^2 .
	\end{split}
	\label{eq:6}
\end{align}
We can now combine Lemma~\ref{lemma:dt_norm_squared}, \eqref{eq:6}, and \eqref{eq:5} to obtain
\begin{align*}
	&\hspace{13pt}
	\frac{1}{2} \left( \partial_t \hspace{1pt} \|(\lagdensitybis_1 - \lagdensitybis_2)(\cdot)\|_{L^2}^2 \right)\!(t)
	+
	\genC[\varphi] \, \|(\lagdensitybis_1 - \lagdensitybis_2)(t)\|_{H^1}^2
	\\[5pt]
	&\leq
	\left( \left( \vphantom{\sum} \partial_t \hspace{1pt} (\lagdensitybis_1 - \lagdensitybis_2) \right)\!(t) \ \rule[-5pt]{0.5pt}{15pt}\  (\lagdensitybis_1 - \lagdensitybis_2)(t) \right)
	+
	\left( \left( \vphantom{\sum} \mathcal{L}_{\varphi, \, \lambda} \hspace{1pt} (\lagdensitybis_1 - \lagdensitybis_2) \right)\!(t) \ \rule[-5pt]{0.5pt}{15pt}\  (\lagdensitybis_1 - \lagdensitybis_2)(t) \right)
	\\[5pt]
	&\leq
	\|(g_{\lambda, \, \testfun_1} - g_{\lambda, \, \testfun_2})(t)\|_{L^2} \, \|(\lagdensitybis_1 - \lagdensitybis_2)(t)\|_{L^2}
	\\[5pt]
	&\leq
	\frac{1}{2 \varepsilon} \, \|(g_{\lambda, \, \testfun_1} - g_{\lambda, \, \testfun_2})(t)\|_{L^2}^2
	+
	\frac{\varepsilon}{2} \, \|(\lagdensitybis_1 - \lagdensitybis_2)(t)\|_{H^1}^2
\end{align*}
for any $\varepsilon >0$. Choosing a small enough $\varepsilon$, we can get
\begin{align*}
	\left( \partial_t \hspace{1pt} \|(\lagdensitybis_1 - \lagdensitybis_2)(\cdot)\|_{L^2}^2 \right)\!(t)
% 	&\leq
% 	\genC[\varphi] \, \|(g_{\lambda, \, \testfun_1} - g_{\lambda, \, \testfun_2})(t)\|_{(H^1)^*}^2
% 	\\[2pt]
	&\leq
	\genC[\varphi] \, \|(g_{\lambda, \, \testfun_1} - g_{\lambda, \, \testfun_2})(t)\|_{L^2}^2
	\\[2pt]
	&\leq
	\genC[\varphi,R] \, \|(\testfun_1 - \testfun_2)(t)\|_{L^2}^2 .
\end{align*}
Consequently, for all $t \in [t_0, t_0 +\delta]$,
\begin{align*}
	\|(\lagdensitybis_1 - \lagdensitybis_2)(t)\|_{L^2}^2
	&=
	\|(\lagdensitybis_1 - \lagdensitybis_2)(t_0)\|_{L^2}^2
	+
	\int_{t_0}^t (\partial_t \hspace{1pt} \|(\lagdensitybis_1 - \lagdensitybis_2)(\cdot)\|_{L^2}^2)(s) \, ds
	\\
	&\leq
	\genC[\varphi,R] \int_{t_0}^t \|(\testfun_1 - \testfun_2)(s)\|_{L^2}^2 \, ds
	\\
	&\leq
	\genC[\varphi,R] \ \delta \ \|\testfun_1 - \testfun_2\|_X^2 ,
\end{align*}
which further gives
\[
	\|\lagdensitybis_1 - \lagdensitybis_2\|_X \leq \sqrt{\genC[\varphi,R] \, \delta} \, \|\testfun_1 - \testfun_2\|_X .
\]
Picking $\delta>0$ such that $\sqrt{\genC[\varphi,R] \, \delta} < 1$ then makes the mapping $\mathcal{S}$ contractive and thus, by the Banach fixed point theorem, there exists a unique weak solution $\lagdensitybis$ to \eqref{eq:reaction_diffusion_ivp2} on $[t_0, t_0 +\delta]$, which leads to the weak solution $\lagdensity(t,x) = e^{\lambda t} \lagdensitybis(t,x)$ of \eqref{eq:reaction_diffusion_ivp} on the same interval. Furthermore, we see that $\delta$ only depends on the deformation $\varphi$ and not the initial condition $\priodensity_{t_0}$. Therefore, applying the same reasoning at $t_0+\delta$ with initial condition $\lagdensity(t_0+\delta,\cdot)$, we can extend the solution to $[t_0,t_0+2\delta]$ and by extension to the whole interval $[t_0,t_0+\eta]$. Then the uniqueness of $\lagdensity$ is an immediate consequence of the uniqueness of the solution of \eqref{eq:reaction_diffusion_ivp2} on each subinterval of length $\delta$, thus completing the proof of Theorem \ref{thm:reaction_diffusion_fixed_phi}.
\end{proof}

As a direct consequence of Theorem \ref{thm:reaction_diffusion_fixed_phi}, we can finally obtain the following result of existence of weak solutions to the reaction-diffusion PDE on the full time interval $[0,T]$:
\begin{corollary}
\label{cor:weak_solution_PDE_global}
Assume that $m\geq 2$ and that $R$ is Lipschitz. Let $\varphi \in H^1([0,T],\D^\reg(\R^\dm))$ with $\varphi(0) = \id$ such that $\sup_{t \,\in\, [0, \, T]} \|F_{\varphi(t)}^{-1}\|_\infty < \infty$. Then for all $\priodensity_{0} \in L^2(M_0)$, there exists a unique weak solution of \eqref{eq:reaction_diffusion_ivp} on $[0,T]$.
\end{corollary}

\subsection{Bounds on the solutions}
We now derive some control bounds on the solution of \eqref{eq:reaction_diffusion_ivp2} with respect to the deformations $\varphi$ and $\varphi_{t_0}$ that will be needed in the next section. Let us fix $r>0$ and denote $B_r=\overbar{B}_{r}(\mathit{id})$, where the closed ball is for the distance defined in \eqref{eq:d.reg.inf}. We consider deformations $\varphi \in C([t_0, t_0+\eta],B_r)$ i.e., such for all $t\in[t_0,t_0+\eta]$, $\|\varphi(t) -\mathit{id}\|_{\reg,\infty} \leq r$ and $\|\varphi(t)^{-1} -\mathit{id}\|_{\reg,\infty} \leq r$. Note that it follows from the results and proofs above that we can find $\lambda_r>0$ and $\alpha_r > 0$ such that for all $\varphi \in C([t_0, t_0+\eta],B_r)$ and $\lagdensitybis \in L^2([t_0, t_0+\eta], H^1(M_{t_0}))$:
\begin{equation}
\label{eq:coercivity_L}
	((\mathcal{L}_{\varphi, \, \lambda_r} \hspace{1pt} \lagdensitybis)(t) \mid \lagdensitybis(t)) \geq \alpha_r \, \|\lagdensitybis(t)\|_{H^1(M_{t_0})}^2
\end{equation}
and that we have a unique local solution which we shall rewrite $\lagdensity_{\varphi, \, \varphi_{t_0}}$ of \eqref{eq:reaction_diffusion_ivp} given by Theorem \ref{thm:reaction_diffusion_fixed_phi}. We also write $\lagdensitybis_{\varphi, \, \varphi_{t_0}} = e^{-\lambda_r t} \lagdensity_{\varphi, \, \varphi_{t_0}}$. 
\begin{lemma}
\label{lemma:pde_L_lipschitz}
Let $\lagdensitybis \in L^2([t_0, t_0+\eta], H^1(M_{t_0}))$ and $\varphi, \psi \in C([t_0, t_0+\eta],B_r)$. For almost every $t \in [t_0, t_0+\eta]$,
\[
	\|(\mathcal{L}_{\varphi, \, \lambda_r} \hspace{1pt} \lagdensitybis)(t) - (\mathcal{L}_{\psi, \, \lambda_r} \hspace{1pt} \lagdensitybis)(t)\|_{H^1(M_{t_0})^*}
	\leq
	\genC_r \, \|\varphi(t) - \psi(t)\|_{\reg, \infty} \, \|\lagdensitybis(t)\|_{H^1(M_{t_0})} .
\]
\end{lemma}
\begin{proof}
A direct computation gives for all $\testfun \in H^1(M_{t_0})$
\begin{align*}
	\left| \left( \vphantom{\sum}
		(\mathcal{L}_{\varphi, \, \lambda_r} \hspace{1pt} \lagdensitybis)(t)
		-
		(\mathcal{L}_{\psi, \, \lambda_r} \hspace{1pt} \lagdensitybis)(t)
		\,\ \rule[-5pt]{0.5pt}{15pt}\,\ 
		\testfun
	\right) \right|
	&\leq
	\left|
		\left\langle (\newdifftensor_{\varphi(t)} - \newdifftensor_{\psi(t)}) \, \nabla \lagdensitybis(t), \nabla \testfun \right\rangle_{L^2}
	\right|
	\\[5pt]
	&+
	\left|
		\left\langle \lagdensitybis(t) \, \left( \newdifftensor_{\varphi(t)} \, \frac{\nabla(\jac\varphi(t))}{\jac\varphi(t)} - \newdifftensor_{\psi(t)} \, \frac{\nabla(\jac\psi(t))}{\jac\psi(t)} \right), \nabla \testfun \right\rangle_{L^2}
	\right|
\end{align*}
Since, for all $t \in [t_0, t_0+\eta]$,
\[
	\max\{
		\|\varphi(t) - \mathit{id}\|_{\reg, \infty}, \,
		\|\varphi(t)^{-1} - \mathit{id}\|_{\reg, \infty}, \,
		\|\psi(t) - \mathit{id}\|_{\reg, \infty}, \,
		\|\psi(t)^{-1} - \mathit{id}\|_{\reg, \infty}
	\} \leq r ,
\]
and $\reg \geq 2$ we have
\[
	\max\left\{
		\|\newdifftensor_{\varphi(t)}\|_\infty, \,
		\|\newdifftensor_{\psi(t)}\|_\infty, \,
		\left\|\frac{\nabla(\jac\varphi(t))}{\jac\varphi(t)}\right\|_\infty, \,
		\left\|\frac{\nabla(\jac\psi(t))}{\jac\psi(t)}\right\|_\infty
	\right\}
	\leq
	\genC[r]
\]
and
\[
	\left\|
		\frac{\nabla(\jac\varphi(t))}{\jac\varphi(t)}
		-
		\frac{\nabla(\jac\psi(t))}{\jac\psi(t)}
	\right\|_\infty
	\leq
	\genC[r] \, \|\varphi(t) - \psi(t)\|_{2, \infty} .
\]
The assumption made on the frame field further gives
\[
	\|\newdifftensor_{\varphi(t)} - \newdifftensor_{\psi(t)}\|_\infty \leq \genC[r] \, \|\varphi(t) - \psi(t)\|_{2, \infty} .
\]
Combining the previous estimates and using the Cauchy-Schwarz inequality, we conclude that
\begin{align*}
	&\left| \left( \vphantom{\sum}
		(\mathcal{L}_{\varphi, \, \lambda_r} \hspace{1pt} \lagdensitybis)(t)
		-
		(\mathcal{L}_{\psi, \, \lambda_r} \hspace{1pt} \lagdensitybis)(t)
		\,\ \rule[-5pt]{0.5pt}{15pt}\,\ 
		\testfun
	\right) \right|
	\leq
	\genC[r] \, \|\varphi(t) - \psi(t)\|_{2, \infty} \, \|\lagdensitybis(t)\|_{H^1(M_{t_0})} \, \|\testfun\|_{H^1(M_{t_0})} \\[5pt]
	&\implies
	\|(\mathcal{L}_{\varphi, \, \lambda_r} \hspace{1pt} \lagdensitybis)(t) - (\mathcal{L}_{\psi, \, \lambda_r} \hspace{1pt} \lagdensitybis)(t)\|_{H^1(M_{t_0})^*} \leq \genC[r] \, \|\varphi(t) - \psi(t)\|_{\reg, \infty} \, \|\lagdensitybis(t)\|_{H^1(M_{t_0})}\,.
\end{align*}
\end{proof}
From this result, we get the following estimates for the solution $\lagdensitybis_{\varphi, \, \varphi_{t_0}}$:
\begin{lemma}
\label{lemma:estimate_u_L2_H1}
\[
	\|\lagdensitybis_{\varphi, \, \varphi_{t_0}}\!(t)\|_{L^2} \leq \genC[r,\varphi_{t_0}] \ \mbox{ for all } t \in [t_0, t_0+\eta]
	\ \ \ \mbox{ and } \ \ \ 
	\int_{t_0}^{t_0+\eta} \|\lagdensitybis_{\varphi, \, \varphi_{t_0}}(t)\|_{H^1(M_{t_0})}^2 \, dt \leq \genC[r,\varphi_{t_0}] .
\]
\end{lemma}
\begin{proof}
From the definition of $\lagdensitybis_{\varphi, \, \varphi_{t_0}}$, we see that for almost all $t\in [t_0, t_0+\eta]$:
\begin{align}
\nonumber
    \left( \vphantom{\sum} (\partial_t \hspace{1pt} \lagdensitybis_{\varphi, \, \varphi_{t_0}})(t) \mid \lagdensitybis_{\varphi, \, \varphi_{t_0}}(t) \right)
	&+
	\left( \vphantom{\sum} (\mathcal{L}_{\varphi, \, \lambda_r} \, \lagdensitybis_{\varphi, \, \varphi_{t_0}})(t) \mid \lagdensitybis_{\varphi, \, \varphi_{t_0}}(t) \right)\\
	\nonumber
	&=
	\int_{M_{t_0}} e^{-\lambda_r t} \, R\!\left( \frac{e^{\lambda_r t} \, \lagdensitybis(t)}{\jac \varphi(t)} \right) \jac\varphi(t) \, \lagdensitybis_{\varphi, \, \varphi_{t_0}}(t) \, dx
	\nonumber
	\\[5pt]
	&\leq
	\frac{\genC_r}{2} \, \|R\|_\infty^2 \, \mathrm{vol}(M_{t_0})
	+
	\frac{1}{2} \, \|\lagdensitybis_{\varphi, \, \varphi_{t_0}}(t)\|_{L^2}^2 .
	\label{eq:estimate_dtu_L2_2_1}
\end{align}

Using Lemma~\ref{lemma:dt_norm_squared} and the coercivity of $\mathcal{L}_{\varphi, \, \lambda_r}$ we then get
\begin{align}
	\frac{1}{2} \left( \partial_t \hspace{1pt} \|\lagdensitybis_{\varphi, \, \varphi_{t_0}}(\cdot)\|_{L^2}^2 \right)(t)
	&=\left( \vphantom{\sum} (\partial_t \hspace{1pt} \lagdensitybis_{\varphi, \, \varphi_{t_0}})(t) \mid \lagdensitybis_{\varphi, \, \varphi_{t_0}}(t) \right) \nonumber \\
	&\leq
	\frac{\genC_r}{2} \, \|R\|_\infty^2 \, \mathrm{vol}(M_{t_0})
	+
	\frac{1}{2} \, \|\lagdensitybis_{\varphi, \, \varphi_{t_0}}(t)\|_{L^2}^2 .
	\label{eq:estimate_dtu_L2_2_2}
\end{align}
It follows that
\begin{align*}
	\|\lagdensitybis_{\varphi, \, \varphi_{t_0}}(t)\|_{L^2}^2
	&=
	\|\lagdensitybis_{\varphi, \, \varphi_{t_0}}(t_0)\|_{L^2}^2 + \int_{t_0}^t \left( \partial_t \hspace{1pt} \|\lagdensitybis_{\varphi, \, \varphi_{t_0}}(\cdot)\|_{L^2}^2 \right)(s) \, ds
	\\[3pt]
	&\leq
	\|\lagdensitybis(t_0)\|_{L^2}^2 + \genC[r] \, T \, \|R\|_\infty^2 \, \mathrm{vol}(M_{t_0})
	+
	\int_{t_0}^t \|\lagdensitybis_{\varphi, \, \varphi_{t_0}}(s)\|_{L^2}^2 \, ds ,
\end{align*}
so Gronwall's lemma gives
\begin{equation}
	\label{eq:estimate_u}
	\|\lagdensitybis_{\varphi, \, \varphi_{t_0}}(t)\|_{L^2}
	\leq
	\left( \|\lagdensitybis(t_0)\|_{L^2}^2 + \genC[r] \, T \, \|R\|_\infty^2 \, \mathrm{vol}(M_{t_0}) \right) e^T
	=
	\genC[r,\varphi_{t_0}],
\end{equation}
since $M_{t_0} = \varphi_{t_0}(M_0)$. 

% Using
% \[
% \lform{\wddt \lagdensitybis_{\varphi, \, \varphi_{t_0}}}{\lagdensitybis_{\varphi, \, \varphi_{t_0}}} + \lform{\mathcal L_{\varphi, \lambda} \lagdensitybis_{\varphi, \, \varphi_{t_0}}}{\lagdensitybis_{\varphi, \, \varphi_{t_0}}} =
% 	\int_{M_{t_0}} e^{-\lambda_r t} \, R\!\left( \frac{e^{\lambda_r t} \, \lagdensitybis(t)}{\jac \varphi(t)} \right) \jac\varphi(t) \, \lagdensitybis_{\varphi, \, \varphi_{t_0}}(t) \, dx,
% 	\]
Now, using \eqref{eq:estimate_dtu_L2_2_1}
	and \eqref{eq:coercivity_L}, we obtain
\[
	\alpha_r \, \|\lagdensitybis_{\varphi, \, \varphi_{t_0}}(t)\|_{H^1(M_{t_0})}^2
	\leq
	\frac{\genC_r}{2} \, \|R\|_\infty^2 \, \mathrm{vol}(M_{t_0})
	+
	\frac{1}{2} \, \|\lagdensitybis_{\varphi, \, \varphi_{t_0}}(t)\|_{L^2}^2
	-
	\frac{1}{2} \left( \partial_t \hspace{1pt} \|\lagdensitybis_{\varphi, \, \varphi_{t_0}}(\cdot)\|_{L^2}^2 \right)(t) .
\]

Integrating on $[t_0, t_0+\eta]$, using Lemma~\ref{lemma:dt_norm_squared} and \eqref{eq:estimate_u}, we obtain
\begin{align*}
	&\hspace{14pt}
	\alpha_r \, \int_{t_0}^{t_0+\eta} \|\lagdensitybis_{\varphi, \, \varphi_{t_0}}(t)\|_{H^1(M_{t_0})}^2 \, dt
	\\[3pt]
	&\leq
	\frac{\genC[r]T}{2} \, \|R\|_\infty^2 \, \mathrm{vol}(M_{t_0})
	+
	\frac{1}{2} \int_{t_0}^{t_0+\eta} \|\lagdensitybis_{\varphi, \, \varphi_{t_0}}(t)\|_{L^2}^2 \, dt
	-
	\frac{1}{2} \left( \vphantom{\sum} \|\lagdensitybis_{\varphi, \, \varphi_{t_0}}(t_0 + \eta)\|_{L^2}^2 - \|\lagdensitybis_{\varphi, \, \varphi_{t_0}}(t_0)\|_{L^2}^2 \right)
	\\[3pt]
	&\leq
	\frac{\genC[r]T}{2} \, \|R\|_\infty^2 \, \mathrm{vol}(M_{t_0}) + \left( \frac{T}{2} + 1 \right) \genC[r,\varphi_{t_0}] ,
\end{align*}
that is 
\[
\int_{t_0}^{t_0+\eta} \|\lagdensitybis_{\varphi, \, \varphi_{t_0}}(t)\|_{H^1(M_{t_0})}^2 \, dt \leq \genC[r,\varphi_{t_0}].
\]
\end{proof}
This leads to the following Lipschitz regularity of $\lagdensity_{\varphi, \, \varphi_{t_0}}(t)$ with respect to $\varphi$.
\begin{lemma}
\label{lemma:tau_lipschitz}
For almost all $t \in [t_0, t_0+\eta]$,
\[
	\|\lagdensity_{\varphi, \, \varphi_{t_0}}(t) - \lagdensity_{\psi, \, \varphi_{t_0}}(t)\|_{L^2(M_{t_0})}
	\leq
	\genC[r,\varphi_{t_0}] \, \sup_{s \,\in\, [t_0, t_0+\eta]} \|\varphi(s) - \psi(s)\|_{\reg, \infty}
%	=
%	\genC[r,\varphi_{t_0}] \, \|\varphi - \psi\|_\infty .
\]
\end{lemma}
\begin{proof}
Since
\[
	\|\lagdensity_{\varphi, \, \varphi_{t_0}}(t) - \lagdensity_{\psi, \, \varphi_{t_0}}(t)\|_{L^2(M_{t_0})}
	=
	\left\|
		e^{\lambda_r t} \, \lagdensitybis_{\varphi, \, \varphi_{t_0}}\!(t)
		-
		e^{\lambda_r t} \, \lagdensitybis_{\psi, \, \varphi_{t_0}}\!(t)
	\right\|_{L^2(M_{t_0})} ,
\]
it suffices to show that
\[
	\|\lagdensitybis_{\varphi, \, \varphi_{t_0}}(t) - \lagdensitybis_{\psi, \, \varphi_{t_0}}(t)\|_{L^2(M_{t_0})}
	\leq
	\genC[r,\varphi_{t_0}] \, \sup_{s \,\in\, [t_0, t_0+\eta]} \|\varphi(s) - \psi(s)\|_{\reg, \infty} .
\]
Recall that $\lagdensitybis_{\varphi, \, \varphi_{t_0}}$ and $\lagdensitybis_{\psi, \, \varphi_{t_0}}$ satisfy
\[
	\left\{
		\begin{array}{l}
		\displaystyle
		(\partial_t \hspace{1pt} \lagdensitybis_{\varphi, \, \varphi_{t_0}})(t)
		+
		(\mathcal{L}_{\varphi, \, \lambda_r} \, \lagdensitybis_{\varphi, \, \varphi_{t_0}})(t)
		=
		e^{-\lambda_r t} \, R\!\left( \frac{e^{\lambda_r t} \, \lagdensitybis_{\varphi, \, \varphi_{t_0}}(t)}{\jac \varphi(t)} \right) \jac\varphi(t)
		\ \ \mbox{ for almost every } t
		\\
		\lagdensitybis_{\varphi, \, \varphi_{t_0}}(t_0) = e^{-\lambda_r t_0} \, \lagdensity_{t_0}
		\end{array}
	\right.
\]
and
\[
	\left\{
		\begin{array}{l}
		\displaystyle
		(\partial_t \hspace{1pt} \lagdensitybis_{\psi, \, \varphi_{t_0}})(t)
		+
		(\mathcal{L}_{\psi, \, \lambda_r} \, \lagdensitybis_{\psi, \, \varphi_{t_0}})(t)
		=
		e^{-\lambda_r t} \, R\!\left( \frac{e^{\lambda_r t} \, \lagdensitybis_{\psi, \, \varphi_{t_0}}(t)}{\jac \psi(t)} \right) \jac\psi(t)
		\ \ \mbox{ for almost every } t
		\\
		\lagdensitybis_{\psi, \, \varphi_{t_0}}(t_0) = e^{-\lambda_r t_0} \, \lagdensity_{t_0}
		\end{array}
	\right. .
\]
Lemma~\ref{lemma:dt_norm_squared} and coercivity of $\mathcal{L}_{\varphi, \, \lambda_r}$ again give
\begin{align*}
	&\hspace{14pt}
	\frac{1}{2}\left( \partial_t \hspace{1pt} \|(\lagdensitybis_{\varphi, \, \varphi_{t_0}} - \lagdensitybis_{\psi, \, \varphi_{t_0}})(\cdot)\|_{L^2(M_{t_0})}^2 \right)(t)
	+
	\alpha_r \, \|(\lagdensitybis_{\varphi, \, \varphi_{t_0}} - \lagdensitybis_{\psi, \, \varphi_{t_0}})(t)\|_{H^1(M_{t_0})}^2
	\\[8pt]
	&\leq
	\left( \vphantom{\sum}
		\big( \partial_t \hspace{1pt} (\lagdensitybis_{\varphi, \, \varphi_{t_0}} - \lagdensitybis_{\psi, \, \varphi_{t_0}}) \big)(t)
		\mid
		(\lagdensitybis_{\varphi, \, \varphi_{t_0}} - \lagdensitybis_{\psi, \, \varphi_{t_0}})(t)
	\right)
    +
	\left( \vphantom{\sum} \big(\mathcal{L}_{\varphi, \, \lambda_r} \, (\lagdensitybis_{\varphi, \, \varphi_{t_0}} - \lagdensitybis_{\psi, \, \varphi_{t_0}})\big)(t) \mid (\lagdensitybis_{\varphi, \, \varphi_{t_0}} - \lagdensitybis_{\psi, \, \varphi_{t_0}})(t) \right)
	\\[8pt]
	&=
	-\left( \vphantom{\sum} \big((\mathcal{L}_{\varphi, \, \lambda_r} - \mathcal{L}_{\psi, \, \lambda_r}) \, \lagdensitybis_{\psi, \, \varphi_{t_0}} \big)(t) \mid (\lagdensitybis_{\varphi, \, \varphi_{t_0}} - \lagdensitybis_{\psi, \, \varphi_{t_0}})(t) \right)
	\\
	&\hspace{14pt}
	\phantom{}
	+
	e^{-\lambda_r t} \left\langle R\!\left( \frac{e^{\lambda_r t} \, \lagdensitybis_{\varphi, \, \varphi_{t_0}}(t)}{\jac \varphi(t)} \right)\jac \varphi(t) - R\!\left( \frac{e^{\lambda_r t} \, \lagdensitybis_{\psi, \, \varphi_{t_0}}(t)}{\jac \psi(t)} \right)\jac \psi(t) , (\lagdensitybis_{\varphi, \, \varphi_{t_0}} - \lagdensitybis_{\psi, \, \varphi_{t_0}})(t) \right\rangle_{L^2(M_{t_0})}
\end{align*}
and using Lemma \ref{lemma:pde_L_lipschitz}, we find
\begin{align*}
	&\hspace{14pt}
	\frac{1}{2}\left( \partial_t \hspace{1pt} \|(\lagdensitybis_{\varphi, \, \varphi_{t_0}} - \lagdensitybis_{\psi, \, \varphi_{t_0}})(\cdot)\|_{L^2(M_{t_0})}^2 \right)(t)
	+
	\alpha_r \, \|(\lagdensitybis_{\varphi, \, \varphi_{t_0}} - \lagdensitybis_{\psi, \, \varphi_{t_0}})(t)\|_{H^1(M_{t_0})}^2
	\\[8pt]
	&\leq
	\genC[r] \ \|\varphi(t) - \psi(t)\|_{\reg, \infty} \ \|\lagdensitybis_{\psi, \, \varphi_{t_0}}(t)\|_{H^1(M_{t_0})} \ \|(\lagdensitybis_{\varphi, \, \varphi_{t_0}} - \lagdensitybis_{\psi, \, \varphi_{t_0}})(t)\|_{H^1(M_{t_0})}
	\\
	&\hspace{14pt}
	\phantom{}
	+
	\genC[r] \left( \vphantom{\sum} \|(\lagdensitybis_{\varphi, \, \varphi_{t_0}} - \lagdensitybis_{\psi, \, \varphi_{t_0}})(t)\|_{L^2(M_{t_0})} + \|\varphi(t) - \psi(t)\|_{\reg, \infty} \right) \|(\lagdensitybis_{\varphi, \, \varphi_{t_0}} - \lagdensitybis_{\psi, \, \varphi_{t_0}})(t)\|_{L^2(M_{t_0})}
	\\[8pt]
	&\leq
	\genC[r] \left( \frac{1}{2\varepsilon} \, \|\varphi(t) - \psi(t)\|_{\reg, \infty}^2 \ \|\lagdensitybis_{\psi, \, \varphi_{t_0}}(t)\|_{H^1(M_{t_0})}^2 + \frac{\varepsilon}{2} \, \|(\lagdensitybis_{\varphi, \, \varphi_{t_0}} - \lagdensitybis_{\psi, \, \varphi_{t_0}})(t)\|_{H^1(M_{t_0})}^2 \right)
	\\
	&\hspace{14pt}
	\phantom{}
	+
	\genC[r] \left( \|(\lagdensitybis_{\varphi, \, \varphi_{t_0}} - \lagdensitybis_{\psi, \, \varphi_{t_0}})(t)\|_{L^2(M_{t_0})}^2 + \frac{1}{2} \, \|\varphi(t) - \psi(t)\|_{\reg, \infty}^2 + \frac{1}{2} \, \|(\lagdensitybis_{\varphi, \, \varphi_{t_0}} - \lagdensitybis_{\psi, \, \varphi_{t_0}})(t)\|_{L^2(M_{t_0})}^2 \right) .
\end{align*}
By choosing $\varepsilon > 0$ such that $\genC[r]\varepsilon/2 < \alpha_r$, we obtain
\begin{align*}
	&\hspace{13pt}
	\left( \partial_t \hspace{1pt} \|(\lagdensitybis_{\varphi, \, \varphi_{t_0}} - \lagdensitybis_{\psi, \, \varphi_{t_0}})(\cdot)\|_{L^2(M_{t_0})}^2 \right)(t)
	\\
	&\leq
	\|\varphi(t) - \psi(t)\|_{\reg, \infty}^2 \left( \genC[r] + \frac{\genC[r]}{\varepsilon} \, \|\lagdensitybis_{\psi, \, \varphi_{t_0}}(t)\|_{H^1(M_{t_0})}^2 \right)
	+
    3 \, \genC[r] \, \|(\lagdensitybis_{\varphi, \, \varphi_{t_0}} - \lagdensitybis_{\psi, \, \varphi_{t_0}})(t)\|_{L^2(M_{t_0})}^2
	\\
	&\leq
	\genC[r]
	\left(
		\|\varphi(t) - \psi(t)\|_{\reg, \infty}^2 \left( 1 + \|\lagdensitybis_{\psi, \, \varphi_{t_0}}(t)\|_{H^1(M_{t_0})}^2 \right)
		+
		\|(\lagdensitybis_{\varphi, \, \varphi_{t_0}} - \lagdensitybis_{\psi, \, \varphi_{t_0}})(t)\|_{L^2(M_{t_0})}^2
	\right) .
\end{align*}
Thus for almost all $t \in [t_0, t_0+\eta]$
\begin{align*}
	&\hspace{14pt}
	\|(\lagdensitybis_{\varphi, \, \varphi_{t_0}} - \lagdensitybis_{\psi, \, \varphi_{t_0}})(t)\|_{L^2(M_{t_0})}^2
	\\
	&=
	\|(\lagdensitybis_{\varphi, \, \varphi_{t_0}} - \lagdensitybis_{\psi, \, \varphi_{t_0}})(t_0)\|_{L^2(M_{t_0})}^2
	+
	\int_{t_0}^t \left( \partial_t \hspace{1pt} \|(\lagdensitybis_{\varphi, \, \varphi_{t_0}} - \lagdensitybis_{\psi, \, \varphi_{t_0}})(\cdot)\|_{L^2(M_{t_0})}^2 \right)(s) \, ds
	\\[5pt]
	&\leq
	0
	+
	\genC[r] \, \|\varphi - \psi\|_\infty^2 \left( \eta + \int_{[t_0, t_0+\eta]} \|\lagdensitybis_{\psi, \, \varphi_{t_0}}(t)\|_{H^1(M_{t_0})}^2 \, dt \right)
	+
	\genC[r] \int_{t_0}^t \|(\lagdensitybis_{\varphi, \, \varphi_{t_0}} - \lagdensitybis_{\psi, \, \varphi_{t_0}})(s)\|_{L^2(M_{t_0})}^2 \, ds .
\end{align*}
We conclude by Lemma~\ref{lemma:estimate_u_L2_H1} and Gronwall's inequality that
\[
	\|(\lagdensitybis_{\varphi, \, \varphi_{t_0}} - \lagdensitybis_{\psi, \, \varphi_{t_0}})(t)\|_{L^2} \leq \genC[r,\varphi_{t_0}] \, \|\varphi - \psi\|_\infty .
\]
\end{proof}

\section{Proof of Theorem \ref{thm:main}}
\label{sec:theorem_proof}
We now move on to the proof of the main result. We will first prove that a unique solution to \eqref{eq:PDELDDMM} and \eqref{eq:PDELDDMM_bd_conditions} exists locally using again a  fixed-point argument before finally showing that the solution is defined on $[0,T]$. As done in the previous section, let us again consider $t_0 \in [0,T)$ and $\varphi_{t_0} \in \D^\reg(\mathbb{R}^\dm)$, $\priodensity_{t_0} \in L^2(M_{t_0})$. By the assumption on $A$, there exist $r>0$ and $\ell_A >0$ both depending on $\varphi_{t_0}$ such that, letting $B_r =\overbar{B}_{r}(\mathit{id})$, we have $\{\varphi \circ \varphi_{t_0}: \varphi \in B\} \subset \D^\reg(\mathbb{R}^\dm)$ and
\[
	\|A_{\varphi \,\circ\, \varphi_{t_0}} - A_{\psi \,\circ\, \varphi_{t_0}}\|_{\mathscr{L}(V, \, V^*)}
	\leq
	\ell_A \, \|\varphi \circ \varphi_{t_0} - \psi \circ \varphi_{t_0}\|_{\reg, \infty} 
	\ \ \mbox{ for all } \ \varphi, \psi \in B_r .
\]
Considering an arbitrary interval $[t_0, t_0 + \eta] \subset [0, T]$, let $S_{\eta, \, \varphi_{t_0}} = C([t_0, t_0+\eta], \, B_r)$ and define $\varGamma_{\eta}: S_{\eta, \, \varphi_{t_0}} \rightarrow C([t_0, t_0+\eta], \, \mathit{id} + C_0^\reg(\mathbb{R}^\dm, \mathbb{R}^\dm))$ by
\begin{equation}
\label{eq:def_Gamma}
	\varGamma_{\eta}(\varphi)(t) = \mathit{id} + \int_{t_0}^t v_{\varphi, \, \varphi_{t_0}}(s) \circ \varphi(s) \, ds ,
\end{equation}
where
\[
	\left\{
		\begin{array}{l}
			v_{\varphi, \, \varphi_{t_0}}(s) = (\omega \, K_V^{-1} + A_{\varphi(s) \,\circ\, \varphi_{t_0}})^{-1} \hspace{1pt} \jerk_{\varphi, \, \varphi_{t_0}}(s) \\[5pt]
			\displaystyle (\jerk_{\varphi, \, \varphi_{t_0}}(s) \mid v') = \int_{\varphi(s, \, \varphi_{t_0}(M_0))} \chi \ \forcefunction(\lagdensity_{\varphi, \, \varphi_{t_0}}(s) \circ \varphi^{-1}(s)) \, (-\mathrm{div} \, v') \, dx
		\end{array}
	\right. .
\]
and $\lagdensity_{\varphi, \, \varphi_{t_0}}$ is the solution of \eqref{eq:reaction_diffusion_ivp} given by Theorem \ref{thm:reaction_diffusion_fixed_phi}.

For the mapping $\varGamma_{\eta}$ to be well-defined, one needs to show that the integral in \eqref{eq:def_Gamma} is finite, which is justified by Lemma \ref{lemma:j_uniformly_bounded} below. For the different proofs that follow, we shall recall first a few results on vector fields, flows and diffeomorphisms.
\begin{proposition}
\label{prop:bounds_v_phi}
Let $u,u'\in V$ and $\varphi,\psi \in C([t_0, t_0+\eta], B_r)$. For all $t \in [t_0, t_0+\eta]$, it holds that
\begin{enumerate}[label = (\roman*), ref = \ref{lemma:inequalities}(\roman*)]
\item $\|u\circ \varphi(t)\|_{\reg,\infty} \leq \genC[r] \|u\|_{\reg,\infty}$,
\item $\|u\circ \varphi(t) - u\circ\psi(t)\|_{\reg,\infty} \leq \genC[r] \|u\|_{\reg+1,\infty} \|\varphi(t) - \psi(t)\|_{\reg,\infty}$,
\item $\|u\circ \varphi(t) - u'\circ\varphi(t)\|_{\reg,\infty} \leq \genC[r] \|u - u'\|_{\reg,\infty}$.
\end{enumerate}
\end{proposition}
\begin{proof}
All these inequalities follow from the Fa\`{a} di Bruno's formula on higher order derivatives of composition of two functions. They can be found e.g., in \cite{younes2019shapes} section 7.1.   
\end{proof}

Furthermore, one has the following controls on $\jerk_{\varphi, \, \varphi_{t_0}}$ and $v_{\varphi, \, \varphi_{t_0}}$, which are simply the generalization of the estimates of section 6.2 in \cite{hsieh2020mechanistic} for $m=2$:
\begin{proposition}
\label{prop:bound_v_j}
Let $\varphi,\psi \in C([t_0, t_0+\eta], B_r)$. Then for all $t \in [t_0, t_0+\eta]$, we have 
\begin{enumerate}[label = (\roman*), ref = \ref{lemma:inequalities}(\roman*)]
\item $\|\jerk_{\varphi, \, \varphi_{t_0}}(t)\|_{V^*} \leq c_V \, \|\forcefunction\|_\infty \, \|\chi\|_{L^1} \colonequals J$
\item $\|v_{\varphi, \, \varphi_{t_0}}(t)\|_{\reg+1, \infty} \leq  \frac{c_V}{\omega} \|\jerk_{\varphi, \, \varphi_{t_0}}(t)\|_{V^*} \leq \frac{c_V}{\omega} J$
\item $\|v_{\varphi, \, \varphi_{t_0}}(t) - v_{\psi, \, \varphi_{t_0}}(t)\|_{\reg, \infty} \leq \frac{J}{\omega^2} \ell_A \|\varphi(t) \circ \varphi_{t_0} - \psi(t) \circ \varphi_{t_0}\|_{\reg, \infty} + \frac{c_V}{\omega} \, \|\jerk_{\varphi, \, \varphi_{t_0}}(s) - \jerk_{\psi, \, \varphi_{t_0}}(s)\|_{V^*}$
\end{enumerate}
\end{proposition}
\noindent (The constant $c_V$ was introduced in Equation \eqref{eq:cv}.)
\begin{proof}
\begin{enumerate}[label = (\roman*)]
\item For any $v' \in V$, we see that:
\begin{align*}
    |(\jerk_{\varphi, \, \varphi_{t_0}}(t) | v')| \leq  \int_{\varphi(s, \, \varphi_{t_0}(M_0))} |\chi| \ \|\forcefunction\|_{\infty} \, |\mathrm{div} \, v'| dx 
    &\leq \|\forcefunction\|_{\infty} \|v'\|_{1,\infty} \left(\int_{\varphi(s, \, \varphi_{t_0}(M_0))} |\chi| dx \right) \\ 
    &\leq c_V \|\forcefunction\|_{\infty} \|v'\|_{V} \|\chi\|_{L^1}.
\end{align*}
which thus leads to $\|\jerk_{\varphi, \, \varphi_{t_0}}(t)\|_{V^*} \leq J$.

\item We have $v_{\varphi, \, \varphi_{t_0}}(t) = L^{-1} \jerk_{\varphi, \, \varphi_{t_0}}(t)$ where $L \colonequals \omega \, K_V^{-1} + A_{\varphi(t) \,\circ\, \varphi_{t_0}}$ so to prove (ii), we first show that for all $v \in V$, $\|v\|_V \leq (1/\omega)\left\| L \, v \right\|_{V^*}$. Indeed
\begin{align*}
	\left( \frac{1}{\omega} \left\| L \, v \right\|_{V^*} \right)^2
	&=
	\left( \frac{1}{\omega} \left\| K_V \left( \omega K_V^{-1} + A_{\varphi(t) \,\circ\, \varphi_{t_0}} \right) v \right\|_V \right)^2
	\\
	&=
	\frac{1}{\omega^2} \left\| \omega \, v + K_V A_{\varphi(t) \,\circ\, \varphi_{t_0}} v \right\|_V^2
	\\
	&=
	\|v\|_V^2 + \frac{1}{\omega^2} \, \|K_V A_{\varphi(t) \,\circ\, \varphi_{t_0}} \hspace{1pt} v\|_V^2 + \frac{2}{\omega} \, \langle v, K_V A_{\varphi(t) \,\circ\, \varphi_{t_0}} \hspace{1pt} v \rangle_V
	\\
	&=
	\|v\|_V^2 + \frac{1}{\omega^2} \, \|K_V A_{\varphi(t) \,\circ\, \varphi_{t_0}} \hspace{1pt} v\|_V^2 + \frac{2}{\omega} \, (A_{\varphi(t) \,\circ\, \varphi_{t_0}} \hspace{1pt} v \mid v)
	\, \geq \,
	\|v\|_V^2 ,
\end{align*}
where the last inequality follows from the positive definiteness of the operator $A_{\varphi(s) \,\circ\, \varphi_{t_0}}$. Together with the assumption that $V \hookrightarrow C_0^{\reg+1}(\mathbb{R}^\dm, \mathbb{R}^\dm)$, it follows that:
\begin{equation*}
  \|v_{\varphi, \, \varphi_{t_0}}(t)\|_{\reg+1, \infty} \leq  c_V \|v_{\varphi, \, \varphi_{t_0}}(t)\|_{V} \leq \frac{c_V}{\omega} \|\jerk_{\varphi, \, \varphi_{t_0}}(t)\|_{V^*} \leq  \frac{c_V}{\omega} J.
\end{equation*}

\item Writing now $L_{\varphi} \colonequals \omega \, K_V^{-1} + A_{\varphi(t) \,\circ\, \varphi_{t_0}}$ and $L_{\psi} \colonequals \omega \, K_V^{-1} + A_{\psi(t) \,\circ\, \varphi_{t_0}}$, we have: 
\begin{align*}
  \|v_{\varphi, \, \varphi_{t_0}}(t) - v_{\psi, \, \varphi_{t_0}}(t)\|_{\reg,\infty} &= \|L_{\varphi}^{-1}\jerk_{\varphi, \, \varphi_{t_0}}(t) - L_{\psi}^{-1}\jerk_{\psi, \, \varphi_{t_0}}(t)\|_{\reg,\infty}  \\
  &\leq \|L_{\varphi}^{-1} \left(\jerk_{\varphi, \, \varphi_{t_0}}(t) - \jerk_{\psi, \, \varphi_{t_0}}(t) \right)\|_{\reg,\infty} + \|(L_{\varphi}^{-1}-L_{\psi}^{-1})\jerk_{\psi, \, \varphi_{t_0}}(t)\|_{\reg,\infty}. 
\end{align*}
Note that, from the proof of (ii), we obtain in particular $\|L_\varphi^{-1}\|_{\mathscr{L}(V^*\!, \, V)} \leq 1/\omega$ and therefore:
\begin{align*}
  \|L_{\varphi}^{-1} \left(\jerk_{\varphi, \, \varphi_{t_0}}(t) - \jerk_{\psi, \, \varphi_{t_0}}(t) \right)\|_{\reg,\infty} &\leq \|L_{\varphi}^{-1} \left(\jerk_{\varphi, \, \varphi_{t_0}}(t) - \jerk_{\psi, \, \varphi_{t_0}}(t) \right)\|_{\reg+1,\infty}\\
  &\leq c_V \|L_{\varphi}^{-1} \left(\jerk_{\varphi, \, \varphi_{t_0}}(t) - \jerk_{\psi, \, \varphi_{t_0}}(t) \right)\|_{V}\\
  &\leq  \frac{c_V}{\omega} \|\jerk_{\varphi, \, \varphi_{t_0}}(s) - \jerk_{\psi, \, \varphi_{t_0}}(s)\|_{V^*} 
\end{align*}
Moreover, using (i), $\|(L_{\varphi}^{-1}-L_{\psi}^{-1})\jerk_{\psi, \, \varphi_{t_0}}(t)\|_{\reg,\infty} \leq \|L_{\varphi}^{-1}-L_{\psi}^{-1}\|_{\mathscr{L}(V^*\!, \, V)} \, J$ and we also have:
\begin{align*}
	\|L_\varphi^{-1} - L_\psi^{-1}\|_{\mathscr{L}(V^*\!, \, V)}
	&=
	\left\| L_\varphi^{-1} \left( L_\psi - L_\varphi \right) L_\psi^{-1} \right\|_{\mathscr{L}(V^*\!, \, V)}
	\\
	&=
	\left\| L_\varphi^{-1} \left( A_{\psi(t) \,\circ\, \varphi_{t_0}} - A_{\varphi(t) \,\circ\, \varphi_{t_0}} \right) L_\psi^{-1} \right\|_{\mathscr{L}(V^*\!, \, V)}
	\\
	&\leq
	\|L_\varphi^{-1}\|_{\mathscr{L}(V^*\!, \, V)} \ 
	\|A_{\psi(t) \,\circ\, \varphi_{t_0}} - A_{\varphi(t) \,\circ\, \varphi_{t_0}}\|_{\mathscr{L}(V, \, V^*)} \ 
	\|L_\psi^{-1}\|_{\mathscr{L}(V^*\!, \, V)} \ 
	\\
	&\leq
	\frac{1}{\omega^2} \, \|A_{\psi(t) \,\circ\, \varphi_{t_0}} - A_{\varphi(t) \,\circ\, \varphi_{t_0}}\|_{\mathscr{L}(V, \, V^*)} \\
	&\leq
	\frac{\ell_A}{\omega^2} \, \|\varphi(t) \circ \varphi_{t_0} - \psi(t) \circ \varphi_{t_0}\|_{\reg, \infty}
\end{align*}
where the last inequality follows from the Lipschitz assumption on the operator $A$.

\end{enumerate}
\end{proof}

Using the estimates of the previous section, we can in addition show the following Lispchitz property of $\jerk_{\varphi, \, \varphi_{t_0}}$.
\begin{lemma}
\label{lemma:j_lipschitz}
For all $t \in [t_0, t_0+\eta]$, 
\[
	\|\jerk_{\varphi, \, \varphi_{t_0}}(t) - \jerk_{\psi, \, \varphi_{t_0}}(t)\|_{V^*}
	\leq
	\genC[\varphi_{t_0}] \sup_{s \,\in\, [t_0, t_0+\eta]} \|\varphi(s) - \psi(s)\|_{2, \infty}
	=
	\genC[r,\varphi_{t_0}] \, \|\varphi - \psi\|_\infty .
\]
\end{lemma}
\begin{proof}
From the definition of $\jerk$, we make a change of variables to obtain
\begin{align*}
	&\hspace{13pt}
	\left| \left( \jerk_{\varphi, \, \varphi_{t_0}}(t) - \jerk_{\psi, \, \varphi_{t_0}}(t) \mid v' \right) \right|
	\\[8pt]
	&=
	\Big|
		\int_{\varphi(t)(M_{t_0})} \chi \ \forcefunction(\lagdensity_{\varphi, \, \varphi_{t_0}}(t) \circ \varphi^{-1}(t)) \, (-\mathrm{div} \, v') \, dx
	\\
	&\hspace{20pt}
		\phantom{}
		-
		\int_{\psi(t)(M_{t_0})} \chi \ \forcefunction(\lagdensity_{\psi, \, \varphi_{t_0}}(t) \circ \psi^{-1}(t)) \, (-\mathrm{div} \, v') \, dx
	\Big|
	\\[8pt]
	&\leq
	\int_{M_{t_0}}
	\Big|
		\chi \circ \varphi(t) \ \, \forcefunction(\lagdensity_{\varphi, \, \varphi_{t_0}}(t)) \ \, (-\mathrm{div} \, v' \circ \varphi(t)) \ \, \jac \varphi(t)
	\\[-5pt]
	&\hspace{60pt}
		\phantom{}
		-
		\chi \circ \psi(t) \ \, \forcefunction(\lagdensity_{\psi, \, \varphi_{t_0}}(t)) \ \, (-\mathrm{div} \, v' \circ \psi(t)) \ \, \jac \psi(t)
	\Big|
	dx
	\\[8pt]
	&\leq
	\|\nabla \chi\|_\infty \ \|\varphi(t) - \psi(t)\|_\infty \ \|\forcefunction\|_\infty \ \|v'\|_{1, \infty} \ \|\jac\varphi(t)\|_\infty \ \mathrm{vol}(M_{t_0})
	\\[3pt]
	&\hspace{15pt}
	\phantom{}
	+
	\genC[r,\varphi_{t_0}] \sup_{s \,\in\, [t_0, t_0+\eta]} \|\varphi(s) - \psi(s)\|_{2, \infty} \ \|v'\|_{1, \infty} \ \|\jac\varphi(t)\|_\infty \ \mathrm{vol}(M_{t_0})
	\\[-2pt]
	&\hspace{15pt}
	\phantom{}
	+
	\|\forcefunction\|_\infty \ \|v'\|_{2, \infty} \ \|\varphi(t) - \psi(t)\|_\infty \ \|\jac\varphi(t)\|_\infty \ \mathrm{vol}(M_{t_0})
	\\[3pt]
	&\hspace{15pt}
	\phantom{}
	+
	\genC[r] \ \|\forcefunction\|_\infty \ \|v'\|_{1, \infty} \  \|\varphi(t) - \psi(t)\|_\infty \ \mathrm{vol}(M_{t_0})
	\\[8pt]
	&\leq
	\genC[r,\varphi_{t_0}] \, \|\varphi - \psi\|_\infty \ \|v'\|_{V} ,
\end{align*}
where we have estimated term by term and used Lipschitz continuity of $\forcefunction$ together with Lemma~\ref{lemma:tau_lipschitz}.
\end{proof}

We can now go back to the definition of the mapping $\varGamma_{\eta}$. 
\begin{lemma}
\label{lemma:j_uniformly_bounded}
For all $\varphi \in C([t_0, t_0+\eta], B_r)$, $\varphi_{t_0} \in \D^\reg(\mathbb{R}^\dm)$ and $\priodensity_{t_0} \in L^2(\varphi_{t_0}(M_0))$, the Bochner integral in \eqref{eq:def_Gamma} is uniformly bounded for $t\in [t_0, t_0+\eta]$.
\end{lemma}
\begin{proof}
Using Proposition \ref{prop:bounds_v_phi}, we find that for all $s\in [t_0, t_0+\eta]$:
\begin{align*}
\|v_{\varphi, \, \varphi_{t_0}}(s)\|_{\reg, \infty} \leq  \|v_{\varphi, \, \varphi_{t_0}}(s)\|_{\reg+1, \infty} \leq \frac{c_V J}{\omega}
\end{align*}
which gives for all $t \in [t_0, t_0+\eta]$:
\begin{align*}
	\int_{t_0}^t \| v_{\varphi, \, \varphi_{t_0}}(s) \circ \varphi(s) \|_{\reg, \infty} \, ds
	\leq
	\int_{t_0}^t \genC[r] \, \|v_{\varphi, \, \varphi_{t_0}}(s)\|_{\reg, \infty} \leq
	\frac{\genC[r] \, c_V}{\omega} \, J \, \eta < \infty.
\end{align*}
where the first inequality follows from Proposition \ref{prop:bound_v_j} (i).
% We first note that for all $v' \in V$ and $t\in [t_0, t_0+\eta]$
% \[
% 	|(j_{\varphi, \, \varphi_{t_0}}(t) \mid v')|
% 	\leq
% 	\|\forcefunction\|_\infty \, \|v'\|_{1, \infty} \, \|\chi\|_{L^1} ,
% \]
% and since $\|v'\|_{1, \infty} \leq \|v'\|_{\reg+1, \infty} \leq c_V \|v'\|_{V}$, we have $\|\jerk_{\varphi, \, \varphi_{t_0}}(t)\|_{V^*} \leq c_V \, \|R\|_\infty \, \|\chi\|_{L^1} \colonequals J$. Moreover, for all $s \in [t_0, t_0+\eta]$
% \begin{align*}
% \|v_{\varphi, \, \varphi_{t_0}}(s)\|_{\reg, \infty} &\leq  \|v_{\varphi, \, \varphi_{t_0}}(s)\|_{\reg+1, \infty} \\
% &= \|(\omega \, K_V^{-1} + A_{\varphi(s) \,\circ\, \varphi_{t_0}})^{-1} \hspace{1pt} \jerk_{\varphi, \, \varphi_{t_0}}(s)\|_{\reg+1, \infty} \\
% &\leq c_V  \|(\omega \, K_V^{-1} + A_{\varphi(s) \,\circ\, \varphi_{t_0}})^{-1} \hspace{1pt} \jerk_{\varphi, \, \varphi_{t_0}}(s)\|_{V} \\
% &\leq \frac{c_V}{\omega} \|\jerk_{\varphi, \, \varphi_{t_0}}(s)\|_{V^*} \leq \frac{c_V J}{\omega}
% \end{align*}
% where we used Proposition \ref{prop:bound_v_j} to obtain the second to last inequality. 
% This leads to, for all $t \in [t_0, t_0+\eta]$
% \begin{align*}
% 	\int_{t_0}^t \| v_{\varphi, \, \varphi_{t_0}}(s) \circ \varphi(s) \|_{\reg, \infty} \, ds
% 	\leq
% 	\int_{t_0}^t C_r \, \|v_{\varphi, \, \varphi_{t_0}}(s)\|_{\reg, \infty} \leq
% 	\frac{C_r \, c_V}{\omega} \, J \, \eta < \infty.
% \end{align*}
% %which also implies that $\varGamma_{[t_0, t_0+\eta]}$ is well defined if $\frac{C_r \, c_V}{\omega} \, J \, \eta \leq r$.
\end{proof}

Note that in addition, if $\eta$ is taken small enough such that $\frac{\genC[r] \, c_V}{\omega} \, J \, \eta \leq r$ then $\varGamma_{\eta}$ maps $S_{\eta, \, \varphi_{t_0}}$ to itself. The goal is now to show that $\varGamma_{\eta}$ is a contractive mapping on $S_{\eta, \, \varphi_{t_0}}$. Indeed, for any $\varphi,\psi \in S_{\eta, \, \varphi_{t_0}}$:
\begin{align*}
	&\hspace{14pt}
	\|\varGamma_{\eta}(\varphi) - \varGamma_{\eta}(\psi)\|_\infty =
	\sup_{t \,\in\, [t_0, t_0+\eta]} \|\varGamma_{\eta}(\varphi)(t) - \varGamma_{\eta}(\psi)(t)\|_{\reg, \infty}
	\\
	&\leq
	\sup_{t \,\in\, [t_0, t_0+\eta]} \int_{t_0}^t \| v_{\varphi, \, \varphi_{t_0}}(s) \circ \varphi(s) - v_{\psi, \, \varphi_{t_0}}(s) \circ \psi(s) \|_{\reg, \infty} \, ds
	\\[5pt]
	&\leq
	\int_{t_0}^{t_0 + \eta}
	\left(
		\| v_{\varphi, \, \varphi_{t_0}}(s) \circ \varphi(s) - v_{\varphi, \, \varphi_{t_0}}(s) \circ \psi(s) \|_{\reg, \infty}
		+
		\| v_{\varphi, \, \varphi_{t_0}}(s) \circ \psi(s) - v_{\psi, \, \varphi_{t_0}}(s) \circ \psi(s) \|_{\reg, \infty}
	\right)
	ds
\end{align*}	
Using Proposition \ref{prop:bounds_v_phi} (ii) and (iii), we get:
\begin{align*}
	&\hspace{14pt}
	\|\varGamma_{\eta}(\varphi) - \varGamma_{\eta}(\psi)\|_\infty 
% 	=
% 	\sup_{t \,\in\, [t_0, t_0+\eta]} \|\varGamma_{\eta}(\varphi)(t) - \varGamma_{\eta}(\psi)(t)\|_{\reg, \infty}
	\\
	&\leq
	\genC[r]
	\int_{t_0}^{t_0 + \eta}
	\left(
		\| v_{\varphi, \, \varphi_{t_0}}(s) \|_{\reg+1, \infty} \, \|\varphi(s) - \psi(s)\|_{\reg, \infty}
		+
		\| v_{\varphi, \, \varphi_{t_0}}(s) - v_{\psi, \, \varphi_{t_0}}(s) \|_{\reg, \infty}
	\right)
	ds
	\\[5pt]
	&\leq
	\genC[r]\int_{t_0}^{t_0 + \eta}
	\left(
		\vphantom{\frac{J}{\omega}}
		\frac{c_V}{\omega} \, J \, \|\varphi(s) - \psi(s)\|_{\reg, \infty}
	\right.
	\\
	&\hspace{60pt}
	\left.
		\phantom{}
		+
		c_V
		\left(
			\frac{J}{\omega^2} \, \ell_A \, \|\varphi(s) \circ \varphi_{t_0} - \psi(s) \circ \varphi_{t_0}\|_{\reg, \infty}
			+
			\frac{1}{\omega} \, \|j_{\varphi, \, \varphi_{t_0}}(s) - j_{\psi, \, \varphi_{t_0}}(s)\|_{V^*}
		\right)
	\right)
	ds
\end{align*}
Now using Lemma~\ref{lemma:j_lipschitz} above, we obtain the following inequalities: 
\begin{align}
\label{eq:proof_theorem_lipj}
&\|\varGamma_{\eta}(\varphi) - \varGamma_{\eta}(\psi)\|_\infty \nonumber\\
&\leq
\genC[r]
	\int_{t_0}^{t_0 + \eta}
	\left(
		\left( \frac{c_V J}{\omega} + \frac{c_V J}{\omega^2} \, \ell_A \, \genC[\varphi_{t_0}]\right) \|\varphi(s) - \psi(s)\|_{\reg, \infty}
		+
		\frac{c_V}{\omega} \, \genC[r,\varphi_{t_0}] \, \sup_{s \,\in\, [t_0, t_0+\eta]} \|\varphi(s) - \psi(s)\|_{\reg, \infty}
	\right)
	ds\\
	&\leq
	\genC[r,\varphi_{t_0}] \, \eta \, \|\varphi - \psi\|_\infty \nonumber,
\end{align}
It follows that there exists a small enough $\eta > 0$ depending on $\varphi_{t_0}$ and $r$ such that $\varGamma_{\eta}$ is a well-defined contraction on $S_{\eta, \, \varphi_{t_0}}$ and, by Banach fixed point theorem, we get the local existence and uniqueness of a solution on $[t_0, t_0+\eta]$.

By concatenating local solutions, we can construct a unique maximal solution $\varphi$ defined on a maximal interval $I_{\max}$, and either $I_{\max} = [0, T')$ for some $T' < T$ or $I_{\max} = [0, T]$. To show that the solution is defined over the entire interval $[0, T]$, we first prove that $\|\varphi(t) - \mathit{id}\|_{\reg, \infty}$ is bounded on $I_{\max}$. For all $t \in I_{\max}$, a solution $\varphi$ satisfies
\begin{equation*}
	%\label{eq:proof_sol_phi}
	\varphi(t, x) = x + \int_0^t v_{\varphi}(s, \varphi(s, x)) \, ds ,
\end{equation*}
 which gives for all $x \in \R^\dm$:
\begin{equation}
	\label{eq:solution_bound_0}
	|\varphi(t, x) - x|
	\leq
	\int_0^t \|v_{\varphi}(s)\|_\infty \, ds
	\leq
	\int_0^t \frac{c_V}{\omega} \, \|j_\varphi(s)\|_{V^*} \, ds
	\leq
	\frac{c_V}{\omega} \, J T .
\end{equation} 
Moreover, we have: 
\begin{equation*}
	%\label{eq:proof_sol_Dphi}
	D\varphi(t, x) = {\boldsymbol I}_\dm + \int_0^t Dv_{\varphi}(s, \varphi(s, x)) \, D\varphi(s, x) \, ds ,
\end{equation*}
where $\boldsymbol{I}_\dm$ denotes the $\dim$-by-$\dim$ identity matrix which leads to
\begin{align*}
	|D\varphi(t, x) - \boldsymbol{I}_\dm|
	&=
	\left|
		\int_0^t
		\left(
			\vphantom{\sum}
			Dv_{\varphi}(s, \varphi(s, x))
			+
			Dv_{\varphi}(s, \varphi(s, x)) (D\varphi(s, x) - \boldsymbol{I}_\dm)
		\right)
		ds
	\right|
	\\
	&\leq
	\frac{c_V}{\omega} \, JT
	+
	\int_0^t \, \frac{c_V}{\omega} \, J \ |D\varphi(s, x) - \boldsymbol{I}_\dm| \, ds.
\end{align*}
By Gr\"{o}nwall's inequality
\begin{equation}
	\label{eq:solution_bound_1}
	|D\varphi(t, x) - \boldsymbol{I}_\dm|
	\leq
	\frac{c_V}{\omega} \, JT
	\exp\!\left( \frac{c_V}{\omega} \, JT \right).
\end{equation}
For the second order derivatives, we see that:
\begin{align*}
	%\label{eq:proof_sol_D2phi}
	\begin{split}
	|D^2\varphi(t, x)|
	&\leq
	\int_0^t
	\left(
		\vphantom{\sum}
		|D^2 v_{\varphi}(s, \varphi(s, x))| \, |D\varphi(s, x)|^2
	\right.
	\\
	&\hspace{30pt}
	\left.
		\vphantom{\sum}
		\phantom{}
		+ 
		|Dv_{\varphi}(s, \varphi(s, x))| \, |D^2\varphi(s, x)|
	\right)
	ds ,
	\end{split}
\end{align*}
and inserting the bound \eqref{eq:solution_bound_1} into the above, we obtain
\[
	|D^2 \varphi(t, x)|
	\leq
	\left( 1 + B_J \right)^2
	\left( \frac{c_V}{\omega} \, JT \right) \,
	+
	\int_0^t \frac{c_V}{\omega} \, J \ |D^2\varphi(s, x)| \, ds .
\]
Using again Gr\"{o}nwall's inequality, we get that $|D^2 \varphi(t, x)|$ is bounded by a constant dependent on $J$ uniformly in $t$ and $x$. Then by a simple recursive argument, we show similarly that there exists a constant $B_J$ such that for any $2\leq k \leq \reg$: 
\begin{equation}
  \label{eq:solution_bound_2}
  |D^k \varphi(t, x)| \leq B_J, \ \forall t \in [0,T'), \ \forall x\in \R^\dm.
\end{equation}
Finally, with \eqref{eq:solution_bound_0}, \eqref{eq:solution_bound_1}, and \eqref{eq:solution_bound_2}, we conclude that
\[
	\|\varphi(t) - \mathit{id}\|_{\reg, \infty} \leq \genC_J .
\]
The same inequality holds for $\varphi^{-1}(t)$ for $T\in [0, T')$. Indeed, from standard results on flows (c.f. for instance \cite{younes2019shapes} Chap.\ 7), one has that for $t\in [0,T')$, the inverse map $\psi(t) \colonequals \varphi(t)^{-1}$ is obtained as the flow of the ODE $dz/ds = \tilde{v}^{(t)}(s,z)$, with $\tilde{v}^{(t)}(s) = -v_{\varphi(t-s)}$, and one can repeat the analysis above with $\tilde v$ in place of $v$. Importantly, this tells us that we can choose $r = \genC[J]$ independently of $T'$.

Now we can show that $\varphi(t)$ has a limit in $\D^\reg(\mathbb{R}^\dm)$ as $t\uparrow T'$ by the Cauchy criterion. Let $(t_k)_{k = 1}^\infty \subset I_{\max}$ be a sequence such that $t_k \uparrow T'$. For $k < l$, we have
\begin{align*}
	\|\varphi(t_k) - \varphi(t_l)\|_{\reg, \infty}
	&\leq
	\int_{t_k}^{t_l} \|v_{\varphi}(s) \circ \varphi(s)\|_{\reg, \infty} \, ds
	\\
	&\leq
	\int_{t_k}^{t_l} \genC_J \ c_V \, \|v_\varphi(s)\|_{V} \, ds
	\\
	&\leq
	\int_{t_k}^{t_l} \genC_J \ \frac{c_V}{\omega} \, \|j_\varphi(s)\|_{V^*} \, ds
	\\
	&\leq
	\frac{\genC_J \, c_V}{\omega} \, J \ (t_k - t_l) ,
\end{align*}
which shows that $(\varphi(t_n))_{n = 1}^\infty$ is a Cauchy sequence in $\D^\reg(\R^\dm)$ for $\|\cdot\|_{\reg,\infty}$. It follows that $t \mapsto \varphi(t)$ has a limit $\varphi(T')$ as $t\uparrow T'$ in the complete space $\id+\C_0^\reg(\R^\dm,\R^\dm)$. Similarly, replacing $v$ by $\tilde v$, we find that $\varphi(t)^{-1}$ also has a limit at $T'$, which is necessarily $\varphi(T')^{-1}$.  This shows that $\varphi(T')\in B_r$, so that the solution can be continued at $t = T'$, which contradicts that $[0, T')$ is the maximal interval of existence.

From the above analysis, we also obtain that $t \mapsto \varphi(t)$ is bounded on $[0,T]$ for $d_{\reg,\infty}$ and  therefore we have $\sup_{t \,\in\, [0, \, T]} \|F_{\varphi(t)}^{-1}\|_\infty < \infty$. Thus Corollary \ref{cor:weak_solution_PDE_global} applies and it follows that we get a weak solution $\lagdensity$ to the reaction-diffusion PDE that is also well-defined on $[0,T]$, which concludes the proof of Theorem \ref{thm:main}.  

\section{Discussion}
We introduced a new general longitudinal model to describe the shape of a material deforming through the action of an internal growth potential which itself evolves according to an advection-reaction-diffusion process. This model extends our previous work in \cite{hsieh2020mechanistic}, which did not include any dynamics on the growth potential beyond pure advection. The present paper was mainly dedicated to proving the long time existence of solutions to the resulting system of coupled PDEs on moving domains. In contrast with other related reaction-diffusion systems on moving domains which often only yield short-time existence, the global existence is here made possible in part thanks to the use of a particular regularization energy on the deformation.  

\begin{figure}[hbt!]
	\centering
	\begin{subfigure}[t]{0.4\textwidth}
		\centering
		\includegraphics[height = 90pt]{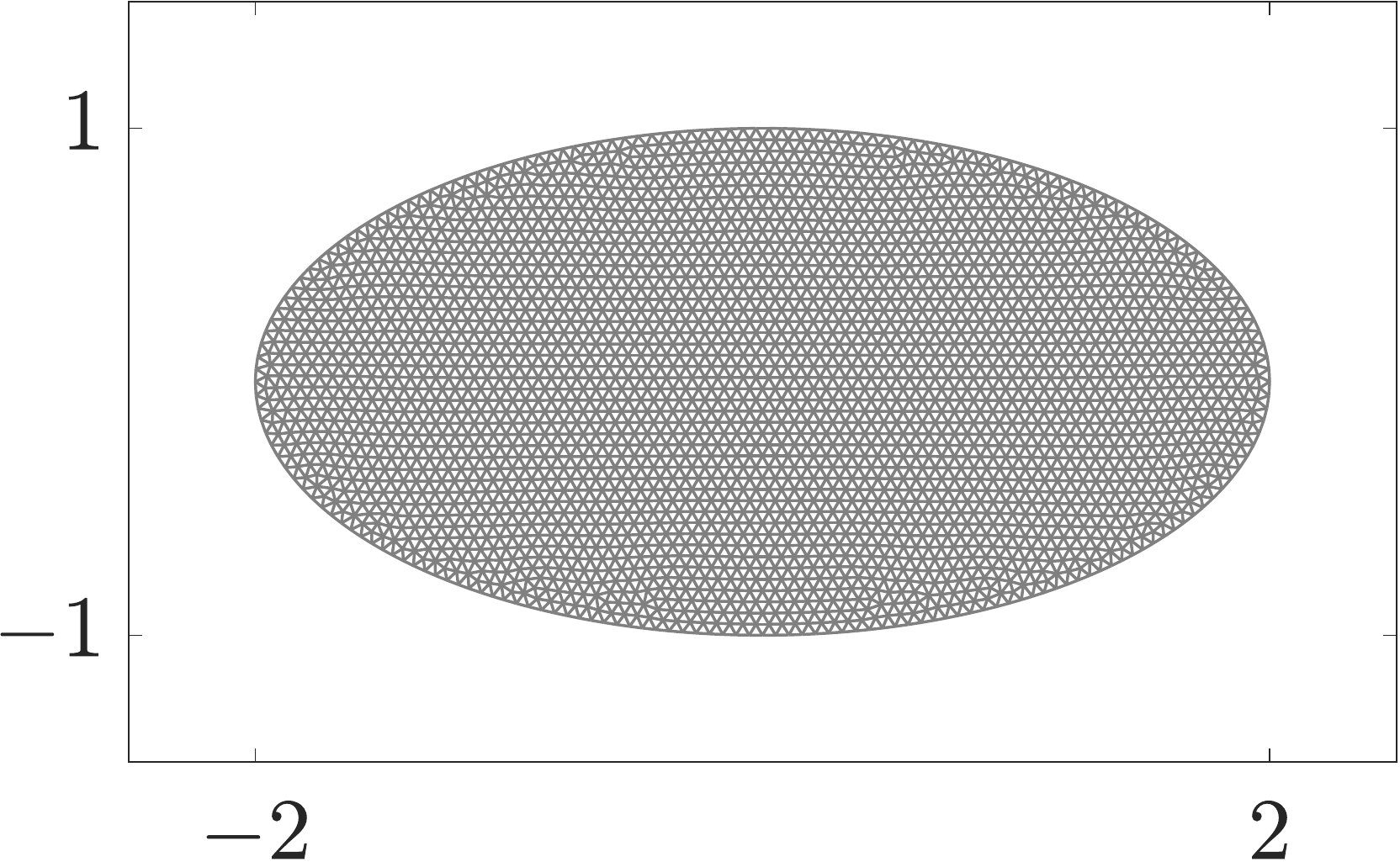}
		\caption{Mesh of the initial shape.}
	\end{subfigure}
	\quad
	\begin{subfigure}[t]{0.4\textwidth}
		\centering
		\includegraphics[trim = 25 50 45 70, clip, height = 91pt]{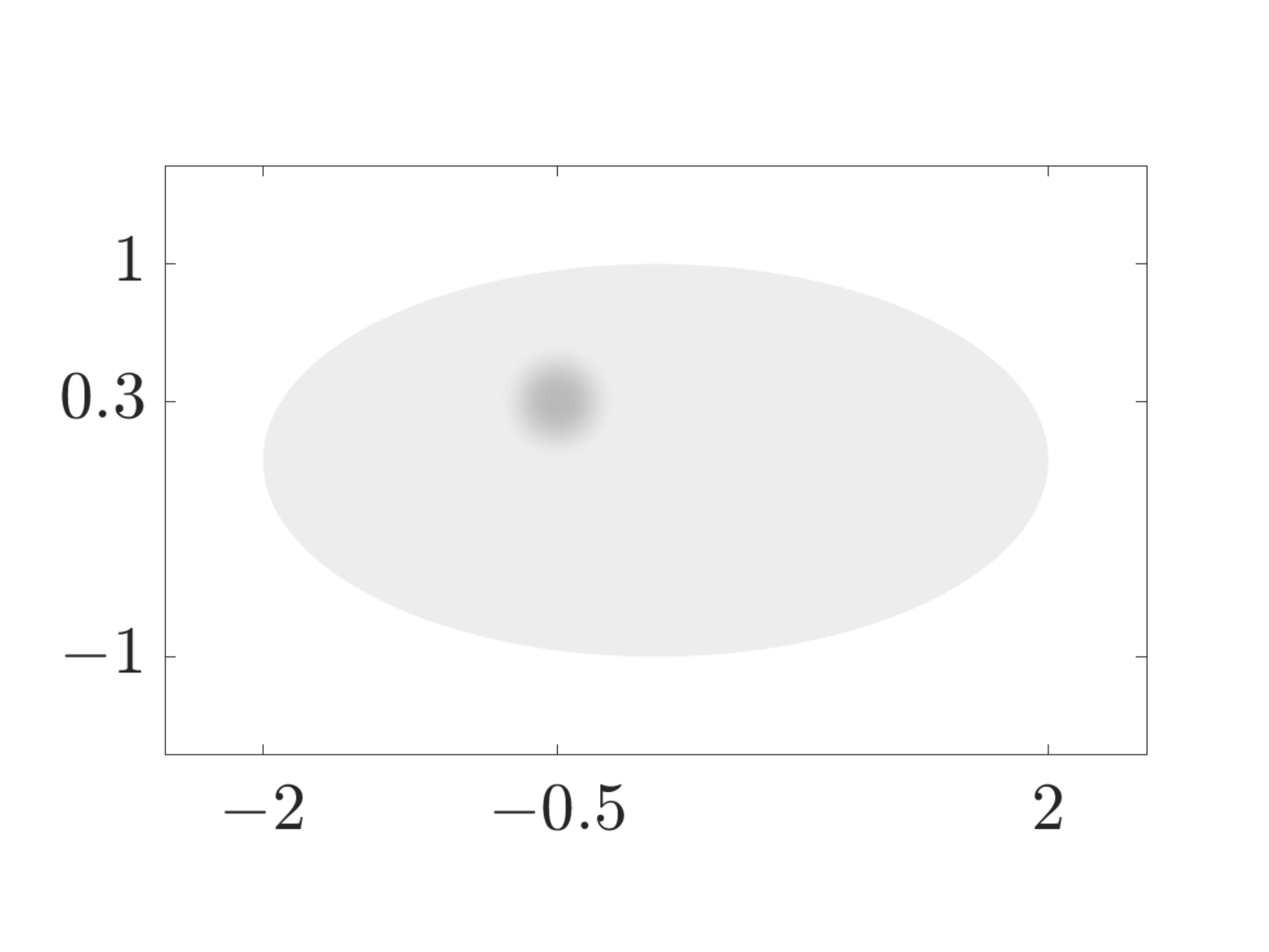}
		\caption{Initial potential centered at $(-0.5, 0.3)$.}
	\end{subfigure}
	\caption{Synthetic initial shape and growth potential used in the numerical simulations.} 
	\label{fig:initial_mesh_potential}
\end{figure}

Although this paper focuses on mathematical aspects, simple numerical simulations of the evolution equations given by \eqref{eq:PDELDDMM} and \eqref{eq:PDELDDMM_bd_conditions} can further illustrate the potential interest of this model in future applications to the study of growth or atrophy of biological tissues, which was the original motivation behind our work. We present a few such preliminary simulations using the simple synthetic 2D domain shown in Figure \ref{fig:initial_mesh_potential} (a) as initial shape $M_0$. We choose the tensor $A_\varphi$ to be the isotropic elastic tensor given by \eqref{eq:elastic_tensor_isotropic} with Lam{\'e} parameters $\lambda = 0$ and $\mu = 1$ on $\varphi(M_0)$ as described earlier in Section \ref{ssec:control_systems}. The initial potential $\priodensity_0$ is a shifted radial function compactly supported in a ball centered at $c_{\mathrm{true}} = (-0.5, 0.3)$ as shown in Figure \ref{fig:initial_mesh_potential} (b). Specifically, it takes the form
\begin{equation}
\label{eq:initial_growth_pot_example}
	\priodensity_0(x; c, r, h) = h \left( \frac{|x - c|^2}{r^2} - 1 \right)^2 \mathbbm{1}_{B(c, r)}(x) .
\end{equation}
with $c\in \Mzero$, $r>0$ and $h>0$ being the center, radius and height of the potential function respectively. We also adopt simple reaction-diffusion and yank models for the purpose of illustration. For the reaction-diffusion model, we let the diffusion tensor be a constant $S_\varphi(t, x) = \mathrm{diag}(0.025, 0.005)$, which diffuses five times faster along the $x$-direction than along the $y$-direction. The reaction and yank functions $\reaction$ and $\forcefunction$ are both $C^2$ piecewise polynomial supported on $[p_{\min}, p_{\max}] = [0.01, 1]$. Their plots are displayed in Figure \ref{fig:R_Q_functions}.

\begin{figure}[hbt!]
	\centering
	\includegraphics[width = 0.35\textwidth]{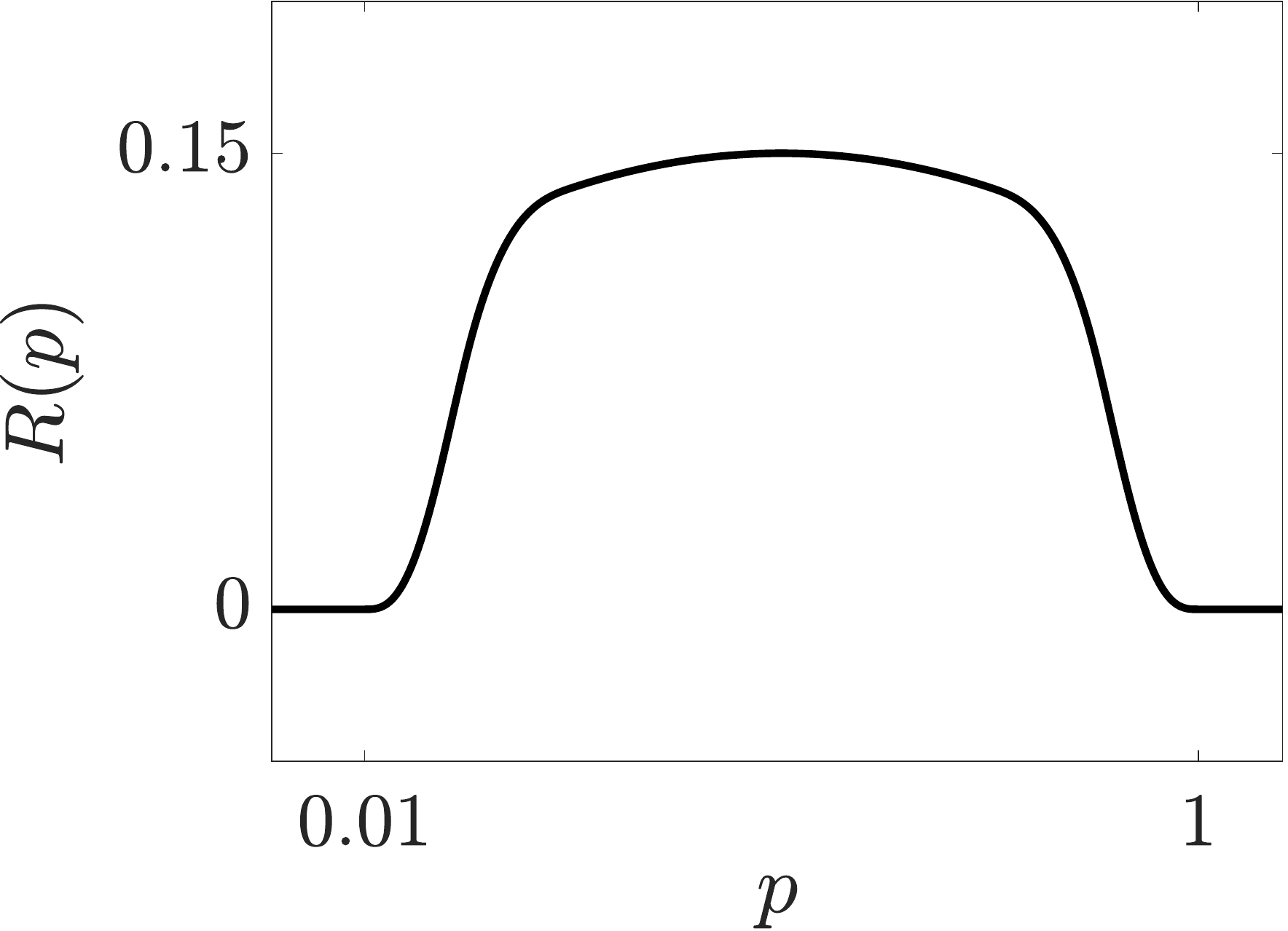}
	\quad\quad
	\includegraphics[width = 0.35\textwidth]{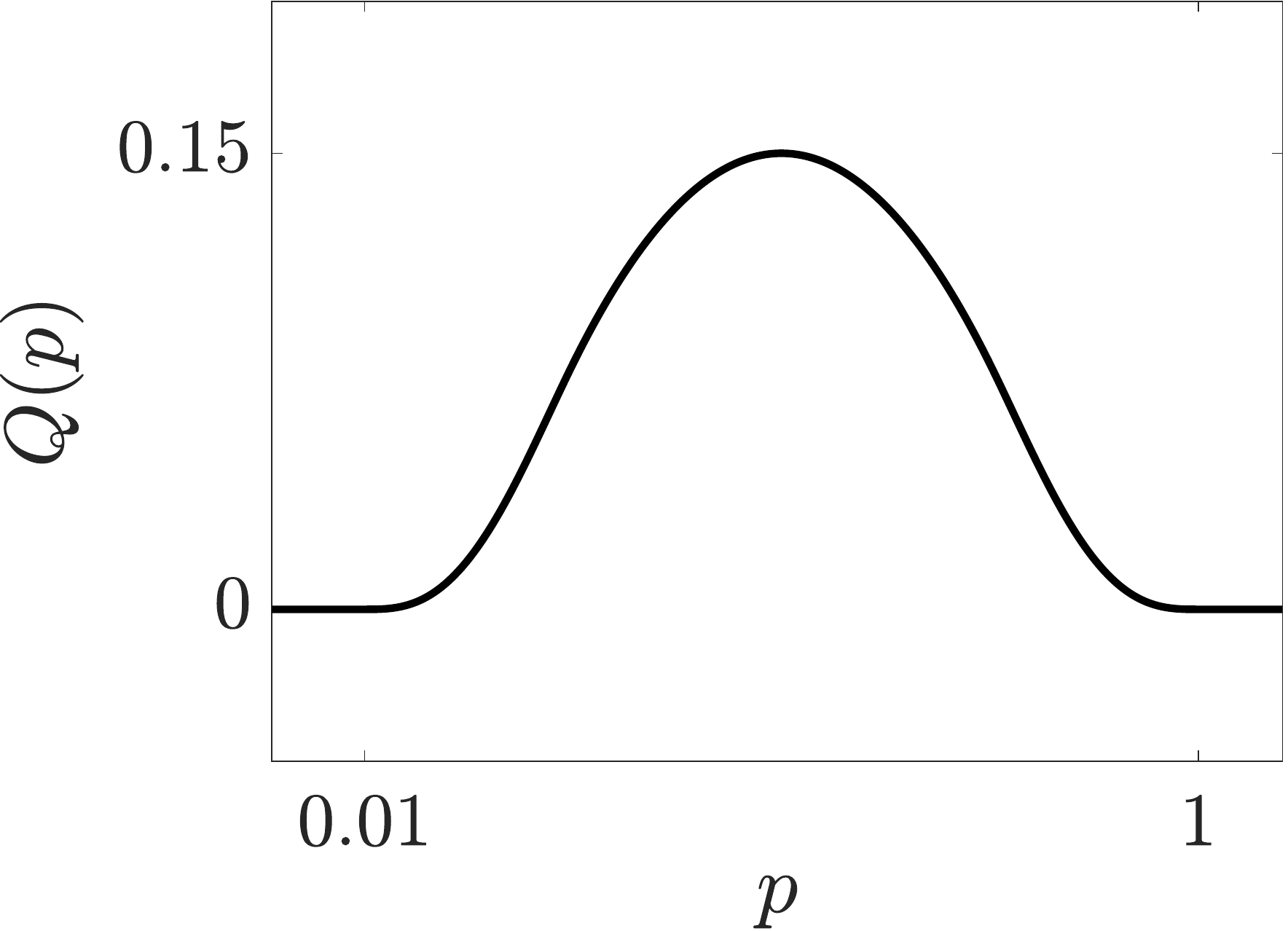}
	\caption{Plots of the functions $R$ and $Q$ used for the reaction and yank expressions in the simulations.} \label{fig:R_Q_functions}
\end{figure}

With the above selection of parameters and initial conditions, the evolution of the growth potential and the resulting deformation of the domain's shape are shown in Figure \ref{fig:expl}. We note that the potential eventually becomes constant over the whole domain after which the deformation stops. One of our main future subject of investigation will be to tackle the inverse problem associated to this longitudinal model, generalizing the work done in \cite{hsieh2020mechanistic}. In other words, if we observe the initial and final (plus possibly some intermediate) domain's shapes and if a parametric representation of the initial growth potential as e.g. \eqref{eq:initial_growth_pot_example} is given, is it possible to recover this initial potential, in particular its location? This issue relates to a long-term goal, in medical imaging, to infer the early impact of neuro-degenerative diseases based on later observations, allowing for a better understanding of their pathogenesis.

\begin{figure}[hbt!]
	\centering
	\begin{subfigure}{\textwidth}
		\centering
		\includegraphics[height = 75pt]{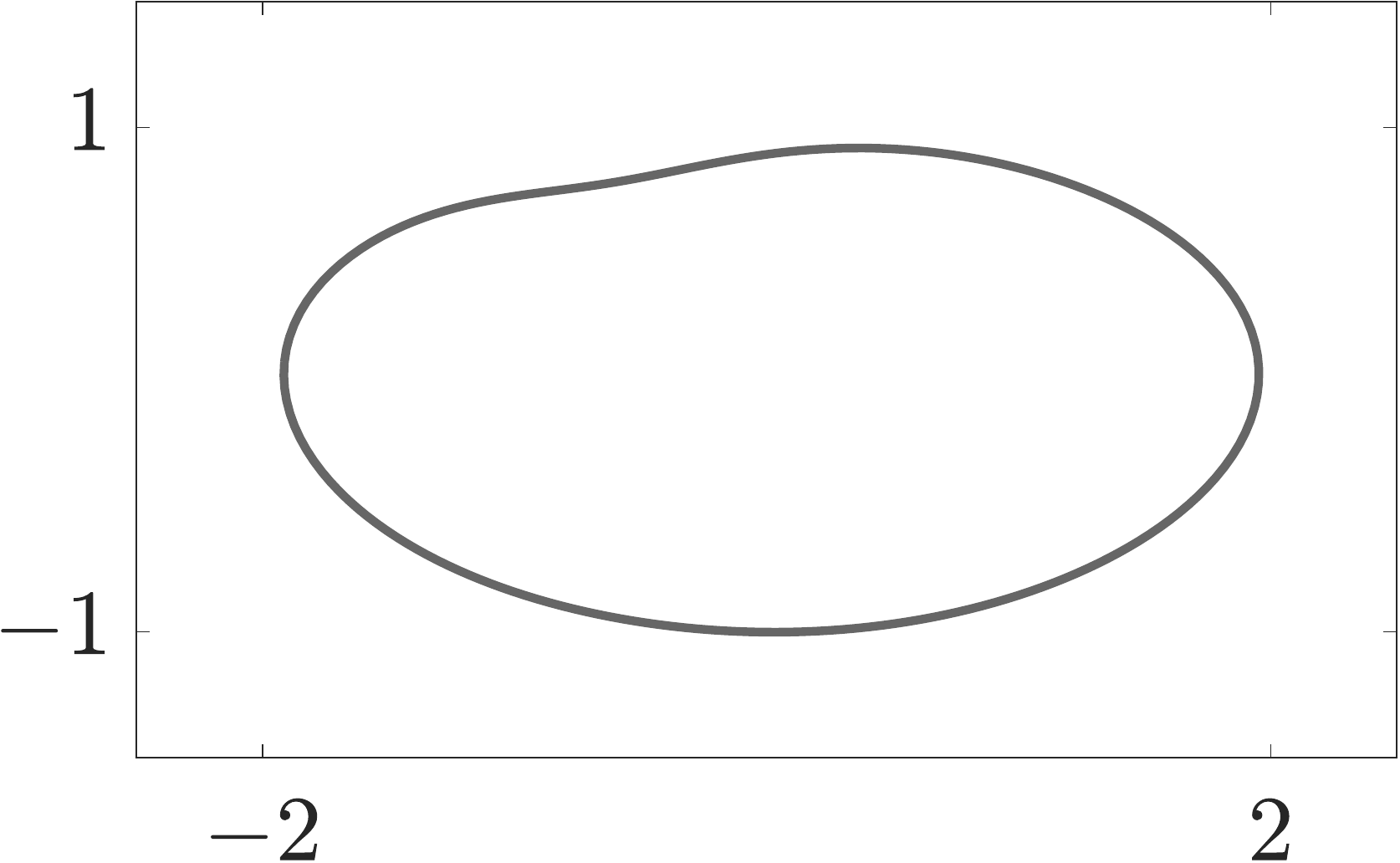}
		\raisebox{-20pt}{\includegraphics[height = 120pt]{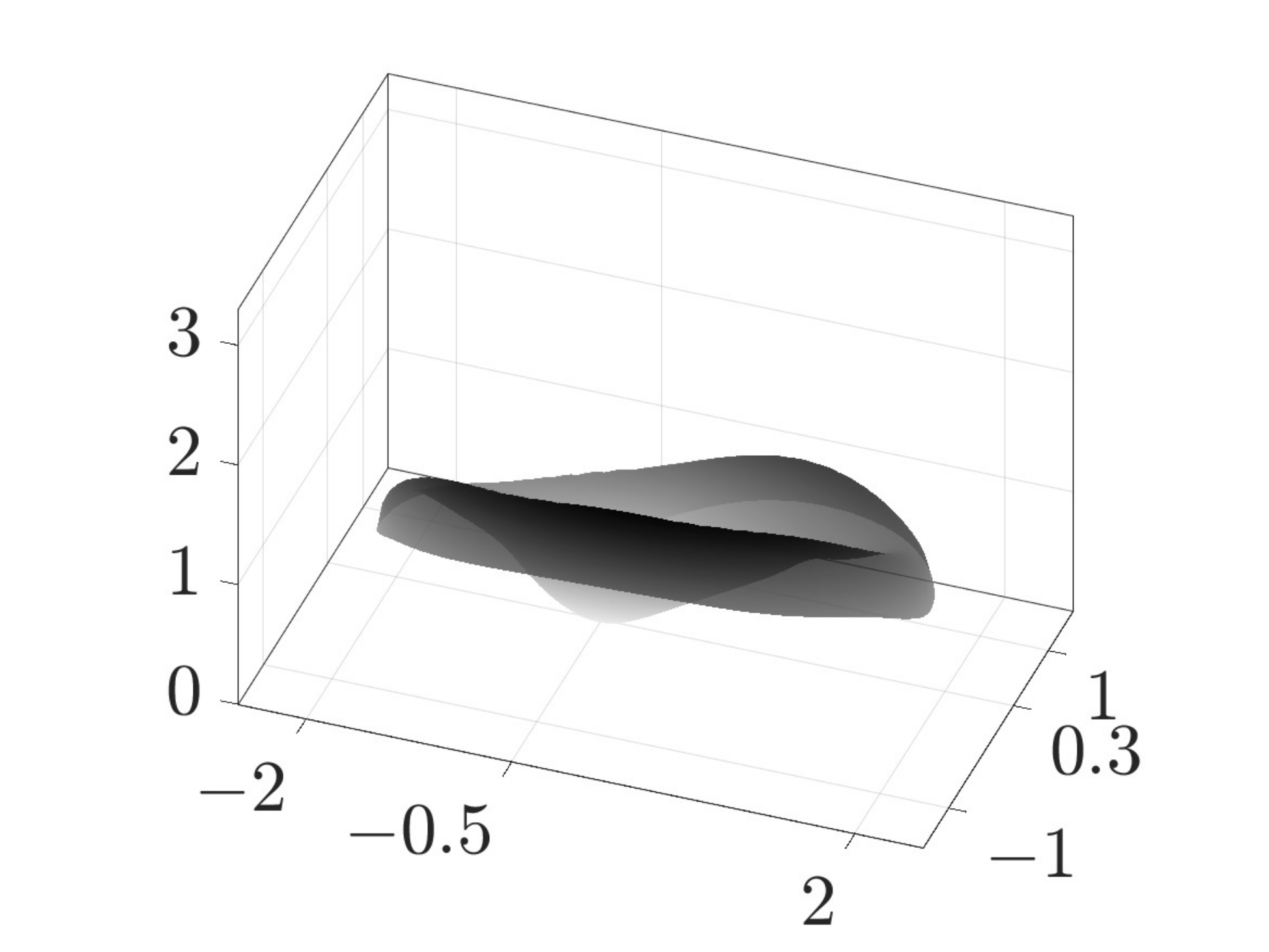}}
		\includegraphics[trim = 15 60 0 81, clip, height = 75pt]{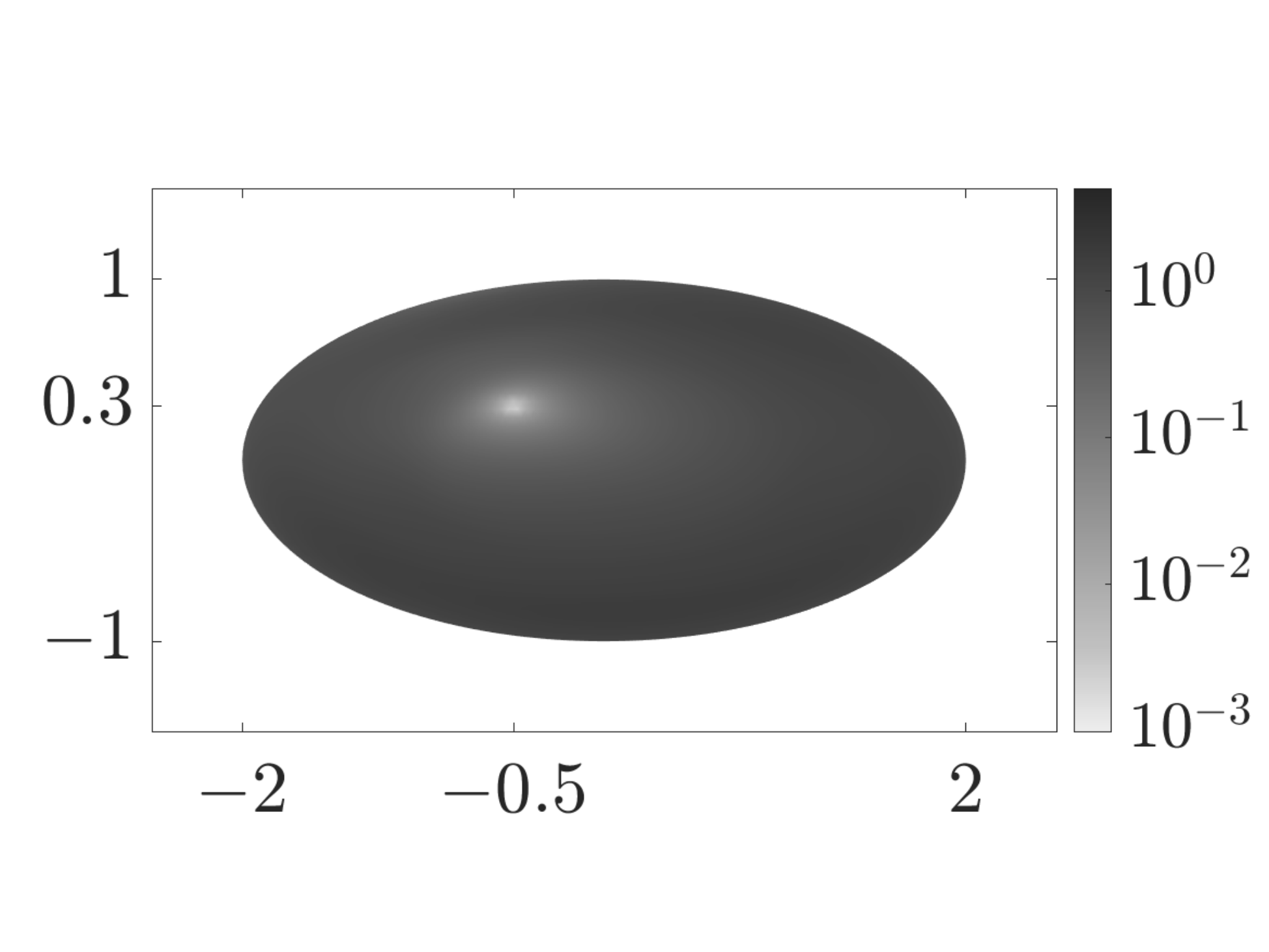}
		%\caption{Target at $T' = 10$ and plots of the objective function in 3D and 2D.}
		\caption{$T' = 10$}
	\end{subfigure}
	\\
	\begin{subfigure}{\textwidth}
		\centering
		\includegraphics[height = 75pt]{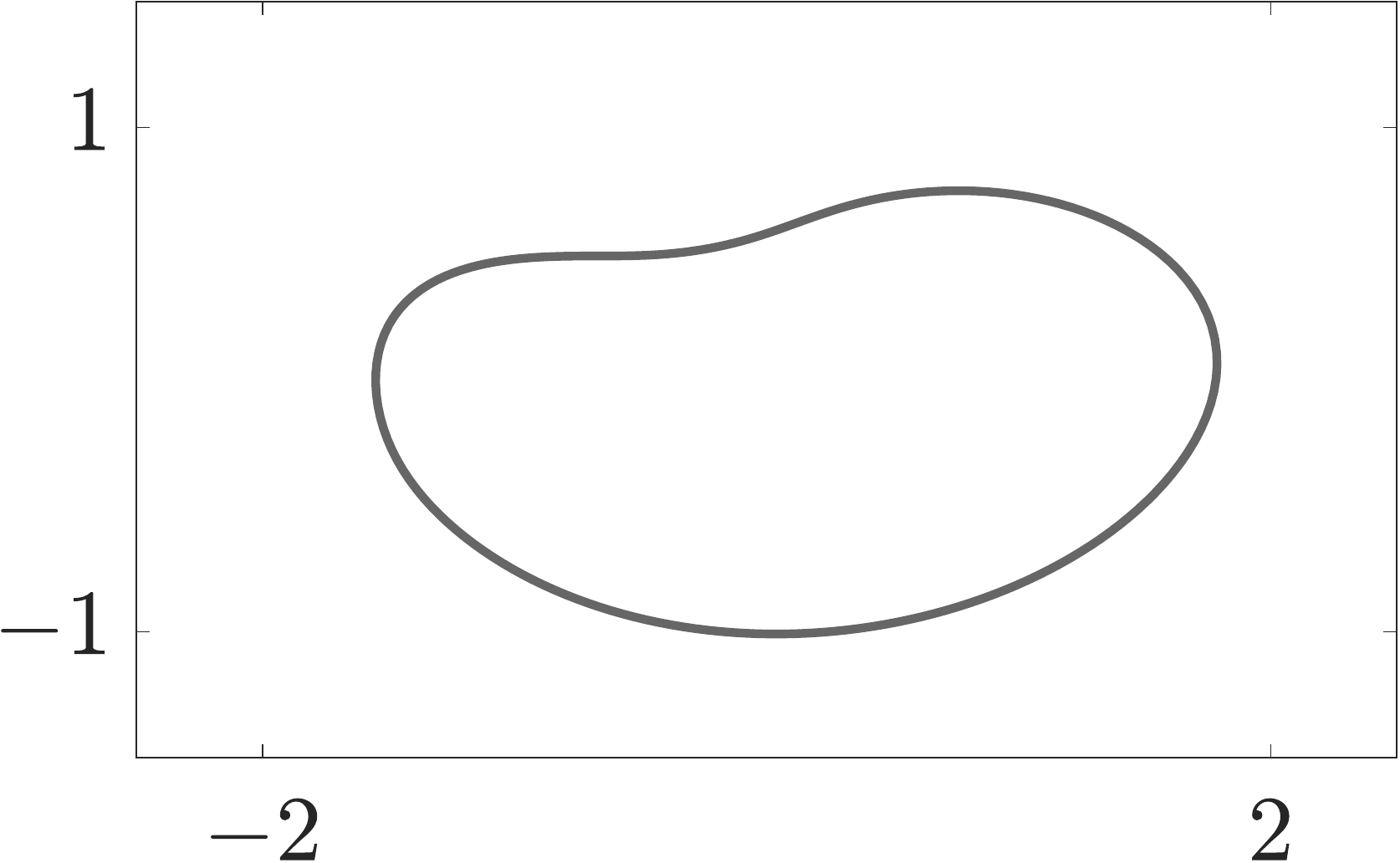}
		\raisebox{-20pt}{\includegraphics[height = 120pt]{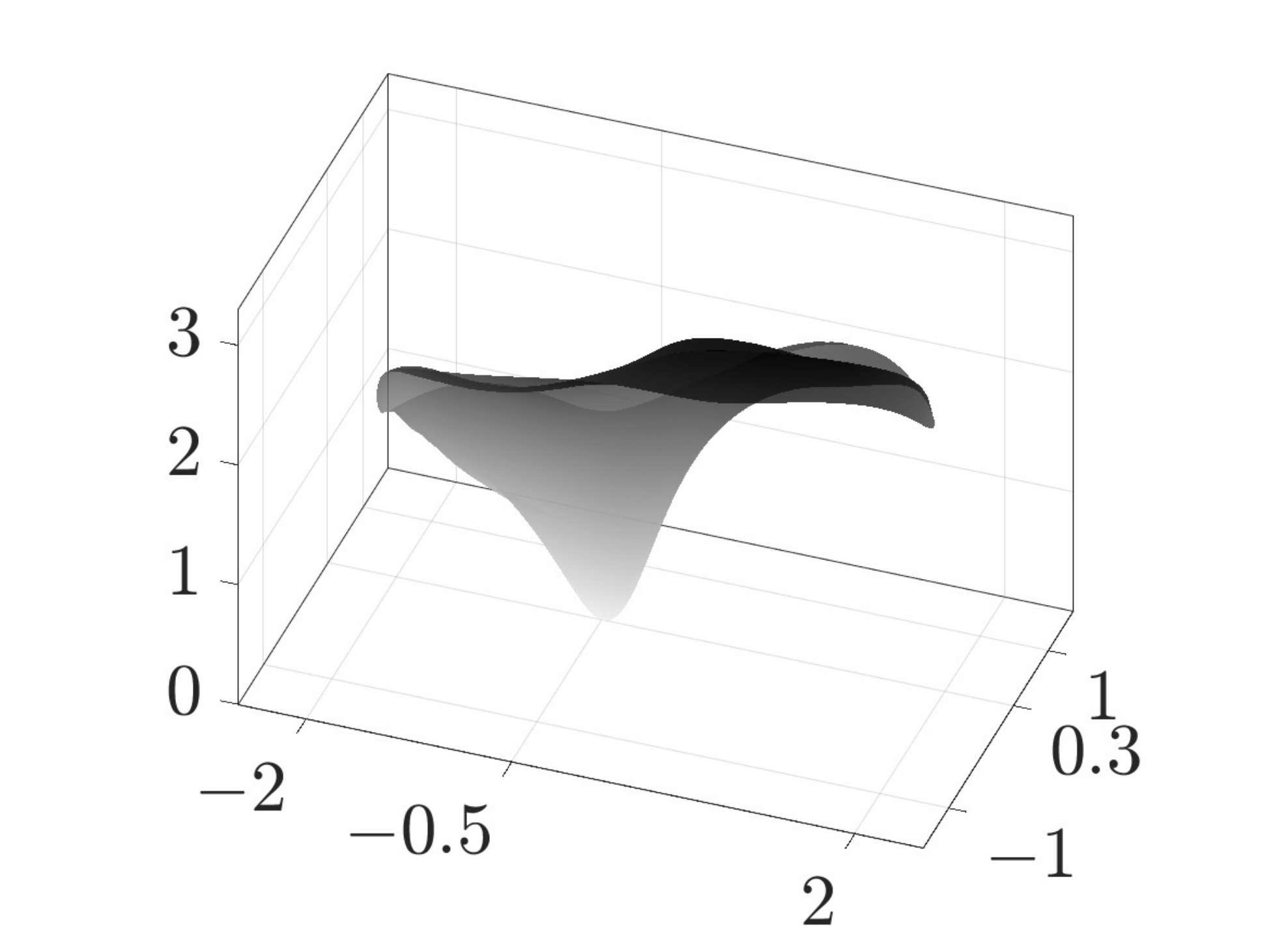}}
		\includegraphics[trim = 15 60 0 81, clip, height = 75pt]{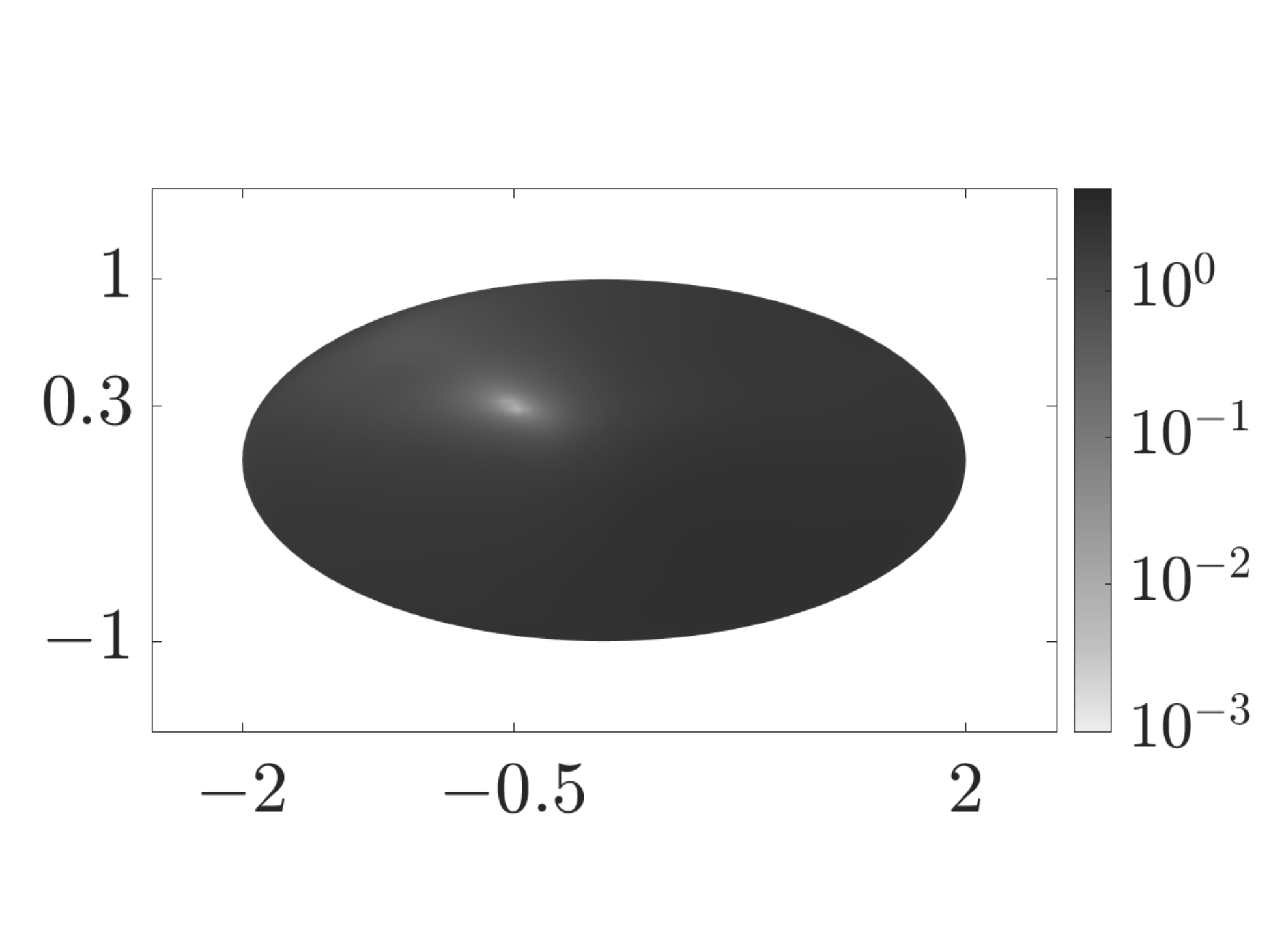}
		\caption{$T' = 15$}
	\end{subfigure}
	\\
	\begin{subfigure}{\textwidth}
		\centering
		\includegraphics[height = 75pt]{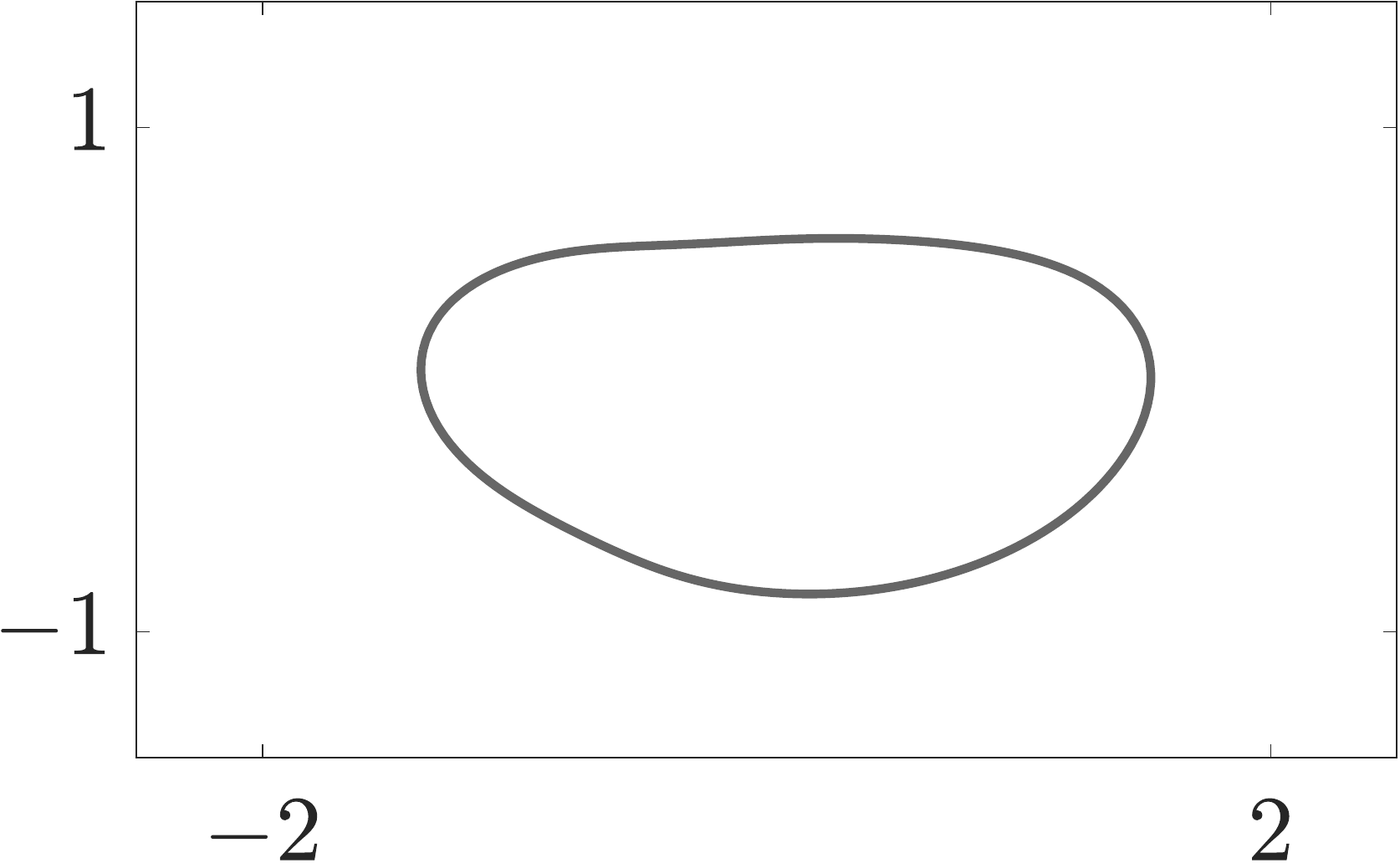}
		\raisebox{-20pt}{\includegraphics[height = 120pt]{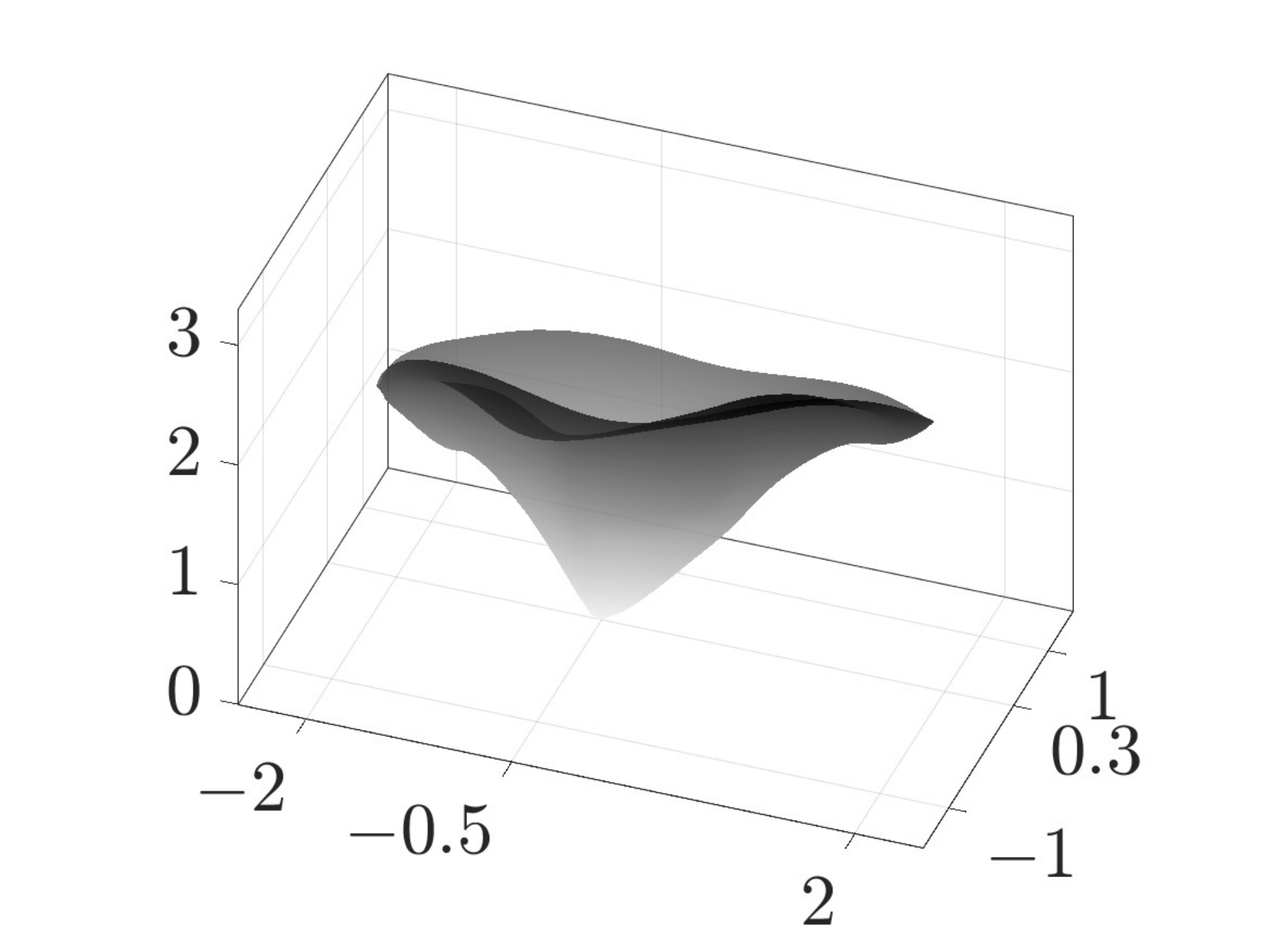}}
		\includegraphics[trim = 15 60 0 81, clip, height = 75pt]{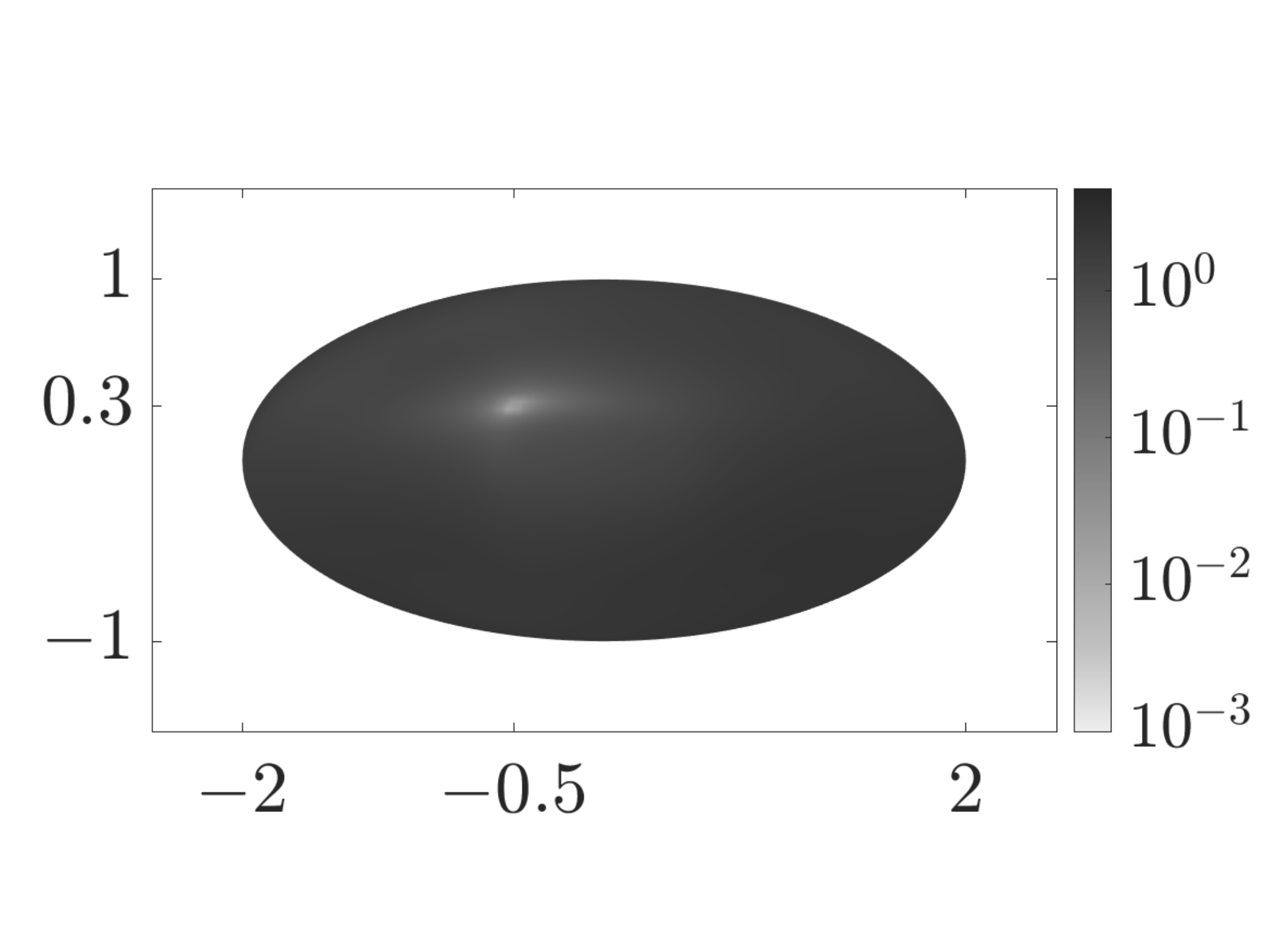}
		\caption{$T' = 20$}
	\end{subfigure}
	\\
	\begin{subfigure}{\textwidth}
		\centering
		\includegraphics[height = 75pt]{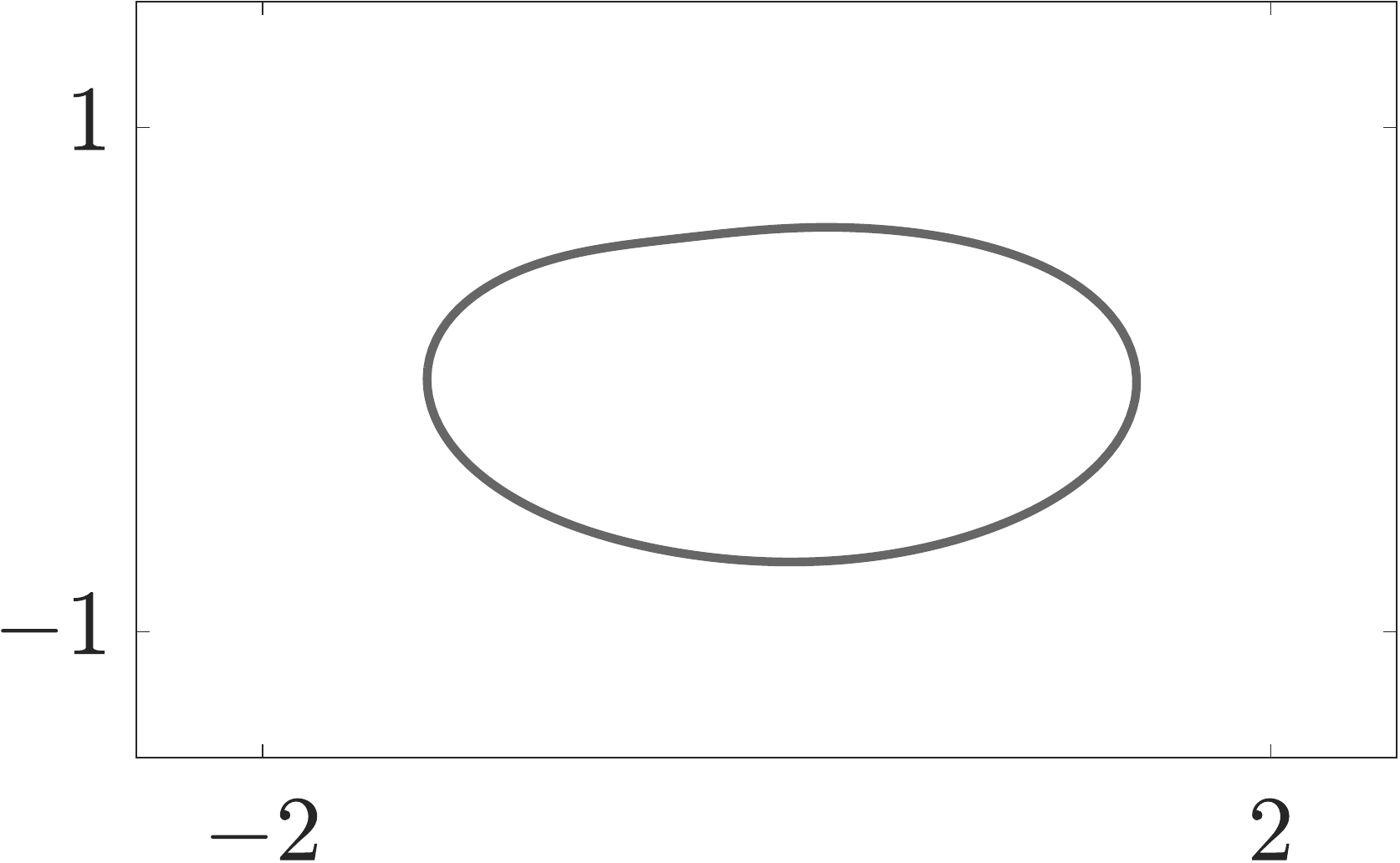}
		\raisebox{-20pt}{\includegraphics[height = 120pt]{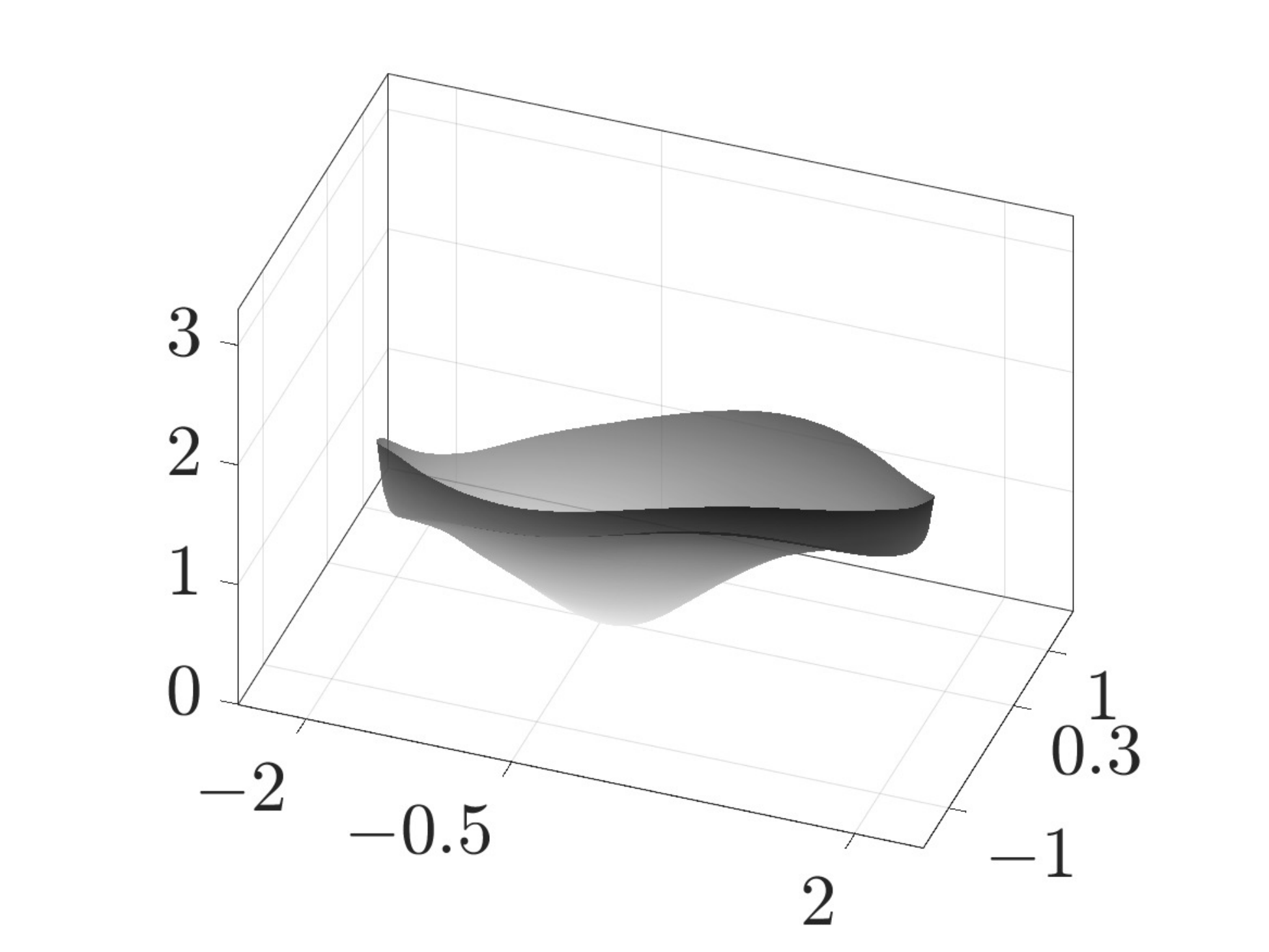}}
		\includegraphics[trim = 15 60 0 81, clip, height = 75pt]{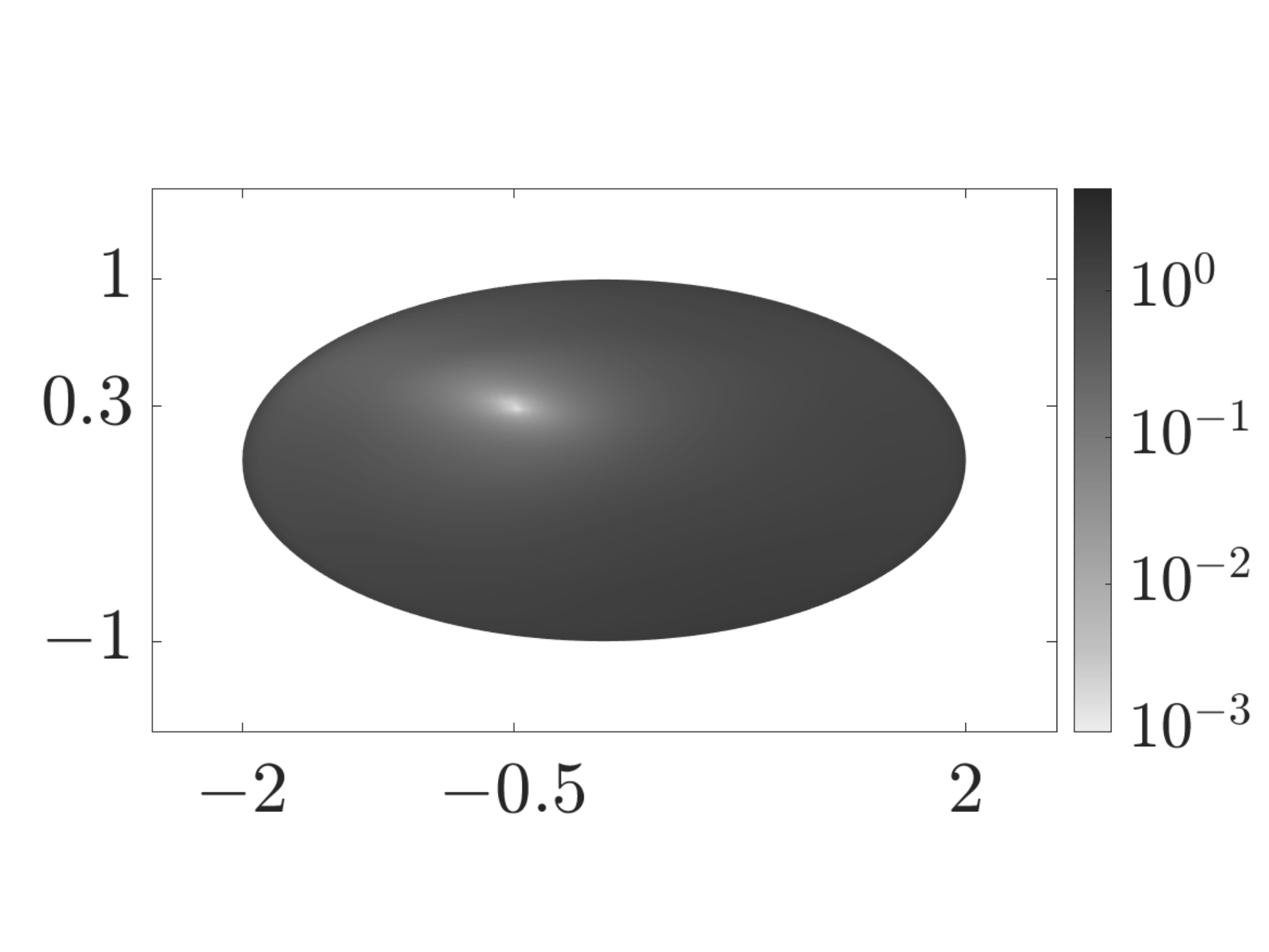}
		\caption{$T' = 25$}
	\end{subfigure}
	\caption{Effect of the growth potential's center $c$ on the deformed domain at different times $T'$. On the left column are the ground truth domains obtained with $c=c_{\mathrm{true}}=(-0.5, 0.3)$. The middle and right column are plots of the varifold distance to this ground truth domain when varying $c$.}
	\label{fig:min_problem}
\end{figure}

To give a hint at the feasibility of such an inverse problem in a simple controlled setting, we consider the deformed domains obtained with the simulation of Figure \ref{fig:expl} at different times $T'$ and for each $T'$, we run our evolution model up to $T'$ but by varying the center $c$ of the initial growth potential in \eqref{eq:initial_growth_pot_example} (all other parameters in the model being kept the same). The shape of the domain's boundary at $T'$ for the different choices of $c$ is then compared to the ground truth (i.e. the one obtained for $c=c_{true}$). To quantify this difference between two boundary curves, we evaluate their distance for the varifold metric introduced in \cite{charon2013varifold} that is known to provide a robust measure of proximity between curves. The results are shown in Figure \ref{fig:min_problem} in which the left column displays the ground truth domains for the different $T'$ while the middle and right columns are plots of the varifold energy with respect to the two coordinates of $c$ with bright colors corresponding to lower values of the varifold distance i.e., closer proximity to the ground truth domain. As can be seen and expected, for each time $T'$, we obtain a minimum distance of $0$ at $c=c_{true}$ but one can further notice that the energy is relatively well behaved around that minimum: for instance we do not observe empirically the presence of additional local minimums. We also note that the global minimums appear more pronounced at intermediate times than at early or late times. 

Although very preliminary, those results suggest that formulating the inverse problem as the minimization of the varifold distance to the observed final domain over the parameters of the initial potential is an a priori viable approach for this problem. In future work, we therefore plan to analyze the well-posedness of such a minimization problem and investigate efficient methods for numerical optimization, in particular to evaluate the gradient of the energy.

\section*{Acknowledgements}
Nicolas Charon acknowledges the support of the NSF through the grant DMS-1945224.

\bibliography{growth}

\end{document}